\documentclass[a4paper, 11pt]{amsart}
\usepackage[utf8]{inputenc}
\usepackage{a4wide}
\usepackage{lmodern}
\usepackage{mathtools}
\usepackage{amsthm}
\usepackage{amsfonts}
\usepackage{amssymb}
\usepackage{tikz}
\usepackage{tikz-cd}
\usetikzlibrary{tikzmark}
\usepackage{enumitem}
\usepackage{hyperref}
\usepackage{xcolor}
\usepackage{array}
\usepackage{booktabs}
\usepackage{colortbl}
\usepackage{wrapfig}
\usepackage{ifthen}
\usepackage{tensor}
\usepackage{cutwin}
\usepackage{tabularx}
\usepackage{float}
\usepackage{xspace} 
\usepackage{longtable}
\usepackage{comment}
\usepackage{arydshln}

\theoremstyle{plain}
\newtheorem{Th}{Theorem}[section]
\newtheorem{Lemma}[Th]{Lemma}
\newtheorem{Cor}[Th]{Corollary}
\newtheorem{Prop}[Th]{Proposition}

\theoremstyle{definition}
\newtheorem{Def}[Th]{Definition}

\newtheorem{Rem}[Th]{Remark}
\newtheorem{?}[Th]{Problem}

\newtheorem{Conv}[Th]{Convention}
\newtheorem{Not}[Th]{Notation}

\definecolor{chartgray}{gray}{0.4}
\definecolor{darkgreen}{rgb}{0, 0.7, 0}
\definecolor{darkcyan}{rgb}{0, 0.7, 0.7}

\newcommand{\tmf}{\textit{tmf}\xspace}
\newcommand{\mmf}{\textit{mmf}\xspace}

\newcommand{\Ext}{\text{Ext}\xspace}
\newcommand{\Sq}{\text{Sq}\xspace}
\newcommand{\cl}{\text{cl}\xspace}

\newcommand{\D}{\Delta\xspace}
\newcommand{\Dc}{\Delta c\xspace}
\newcommand{\Du}{\Delta u\xspace}
\newcommand{\Dh}{\Delta h_1\xspace}

\renewcommand{\deg}[3]{\textcolor{blue}{$(#1, #2, #3)$}}

\title{The cohomology of $\mathbb{R}$-motivic $\mathcal{A}(2)$}
\author{Konstantin Emming}
\address{Department of Mathematics, Wayne State University, Detroit, MI 48202, USA}
\email{konstantin.emming@gmail.com}



\begin{document}

\begin{abstract}
    We compute the cohomology of the quotient algebra $\mathcal{A}(2)$ of the $\mathbb{R}$-motivic dual Steenrod algebra. We do so by running a $\rho$-Bockstein spectral sequence whose input is the cohomology of $\mathbb{C}$-motivic $\mathcal{A}(2)$. The purpose of our computation is that the cohomology of $\mathcal{A}(2)$ is the input to an Adams spectral sequence of a hypothetical $\mathbb{R}$-motivic modular forms spectrum. This Adams spectral sequence computes the homotopy groups of such an $\mathbb{R}$-motivic modular forms spectrum, which in turn can be used to make inferences about the homotopy groups of the $\mathbb{R}$-motivic sphere spectrum and eventually about the classical stable stems.
\end{abstract}

\maketitle

\section{Introduction}\label{section: introduction}

The aim of this article is to compute the tri-graded cohomology of $\mathbb{R}$-motivic $\mathcal{A}(2)$, meaning $\Ext_{\mathcal{A}(2)}(\mathbb{M}_2, \mathbb{M}_2)$. Here, $\mathcal{A}(2)$ is dual\footnote{We avoid using additional subscripts or superscripts for dual objects. Everything is phrased in terms of the dual Steenrod algebra.} to the subalgebra of the $\mathbb{R}$-motivic mod $2$ Steenrod algebra generated by $\Sq^1$, $\Sq^2$, and $\Sq^4$, and $\mathbb{M}_2 = \mathbb{F}_2[\rho, \tau]$ is the $\mathbb{R}$-motivic homology of a point with $\mathbb{F}_2$-coefficients.

Recent breakthroughs in stable homotopy theory \cite{Isa09} \cite{IWX} have been achieved by making computations in motivic homotopy theory. In \cite{IWX} the authors employ the $\mathbb{C}$-motivic modular forms spectrum $\mmf^{\,\,\mathbb{C}}$ as constructed in \cite{GIKR} and use its $\mathbb{C}$-motivic stable homotopy groups to approximate the $\mathbb{C}$-motivic stable homotopy groups of spheres. One way to compute the stable homotopy groups of $\mmf^{\,\,\mathbb{C}}$ is by running the $\mathbb{C}$-motivic Adams spectral sequence \cite{Isa18}. The input to this spectral sequence is the cohomology of $\mathbb{C}$-motivic $\mathcal{A}(2)$, as computed in \cite{Isa09}. The logical next step that we take here is to mimic these computations over $\mathbb{R}$. If we were to assume the existence of an $\mathbb{R}$-motivic modular forms spectrum $\mmf^{\,\,\mathbb{R}}$, then we could follow the very same program and eventually further our knowledge of the classical stable homotopy groups of spheres. We should emphasize that, as of writing this article, an $\mathbb{R}$-motivic modular forms spectrum has not yet been constructed. However, we believe that our computation should serve as further motivation to do so.

As we will see, the cohomology of $\mathbb{R}$-motivic $\mathcal{A}(2)$ has an extremely rich and complicated structure. To compare, the cohomology of classical $\mathcal{A}(2)^\cl$ has a presentation as a bigraded $\mathbb{F}_2$-algebra with $13$ generators, see e.g. \cite[section 8]{SI67}. The cohomology of $\mathbb{C}$-motivic $\mathcal{A}(2)^\mathbb{C}$ can be given a presentation with $16$ generators \cite[table 7]{Isa09}\footnote{The table in question contains $15$ elements because the generator $\tau$ is not listed explicitly.}. In contrast, table \ref{tab:Indecomposables on E_infty} shows that the cohomology of $\mathbb{R}$-motivic $\mathcal{A}(2)$ has $56$ generators.

We should mention previous work on the cohomology of the $\mathbb{R}$-motivic Steenrod algebra: In \cite{Hill} the cohomology of $\mathcal{A}(1)$ is computed in its entirety. In \cite{DI} and \cite{BI} the cohomology of $\mathcal{A}$ is computed in a range of degrees.

\subsection{The \texorpdfstring{$\rho$}{rho}-Bockstein spectral sequence}
We will use the $\rho$-Bockstein spectral sequence to compute the cohomology of $\mathbb{R}$-motivic $\mathcal{A}(2)$. This type of spectral sequence has been employed many times before \cite{Hill} \cite{DI} \cite{BI} to compute the cohomology of $\mathbb{R}$-motivic $\mathcal{A}(1)$ and $\mathcal{A}$. To provide a short derivation of this spectral sequence, we can filter $\mathcal{A}(2)$ by powers of the ideal generated by an element called $\rho$. This extends to a filtration on the cobar complex which computes the cohomology of $\mathcal{A}(2)$. The spectral sequence associated to this filtration is called the $\rho$-Bockstein spectral sequence. Its signature is
\[E_1^{s,f,w} = \Ext_{\mathcal{A}(2)^{\mathbb{C}}}^{s,f,w}(\mathbb{M}_2^{\mathbb{C}},\mathbb{M}_2^{\mathbb{C}})[\rho] \implies \Ext_{\mathcal{A}(2)}^{s,f,w}(\mathbb{M}_2,\mathbb{M}_2),\]
i.e. the input is the tri-graded cohomology of $\mathbb{C}$-motivic $\mathcal{A}(2)^{\mathbb{C}}$ with a variable $\rho$ freely adjoined, and the output is the tri-graded cohomology of $\mathbb{R}$-motivic $\mathcal{A}(2)$.

Our task is to find the differentials in this spectral sequence. We will use two methods:
\begin{itemize}
    \item $\rho$-localization: This idea was first introduced in \cite[Section 4]{DI} to study the cohomology of $\mathbb{R}$-motivic $\mathcal{A}$. In our case, localizing the Hopf algebroid $\mathcal{A}(2)$ with respect to $\rho$ leads to a splitting. This makes it possible to compute the $\rho$-localized cohomology of $\mathcal{A}(2)$ before running the $\rho$-Bockstein spectral sequence. It turns out that the $\rho$-localized cohomology of $\mathcal{A}(2)$ is isomorphic to a shifted version of the cohomology of classical $\mathcal{A}(1)^{\cl}$, together with a freely adjoined generator $\tau^8$.

    \item Internal coweight method: This method was already used in \cite[Strategy 5.3]{BI} for $\mathbb{R}$-motivic $\mathcal{A}$. The idea is to fix a specific linear combination of degrees, the so-called internal coweight. The $\rho$-Bockstein differentials as well as $\rho$-multiplication preserve the internal coweight by design. Therefore, the $\rho$-Bockstein spectral sequence splits into smaller spectral sequences, each indexed by an internal coweight. These smaller spectral sequences can often be solved individually, for example by appealing to the $\rho$-localization.
\end{itemize}
Due to our knowledge of the $\rho$-localization of the cohomology of $\mathcal{A}(2)$ we will find that the $\rho$-Bockstein spectral sequence ends with $E_{11} \cong E_\infty$, or in other words the longest differential is a $d_{10}$. This is not apparent from the description of the $\rho$-localization as given in Corollary \ref{representation of rho-localized target}, but it is apparent once we have computed $E_{11}$ and compare it with the $\rho$-localization.

After having computed $E_\infty$ we need to consider hidden extensions. The multiplicative structure of the $E_\infty$-page only determines the multiplicative structure on the cohomology of $\mathcal{A}(2)$ up to higher $\rho$-filtration. Due to the fact that we have $56$ indecomposables on $E_\infty$, finding all hidden extensions is a daunting task. We will restrict ourselves to only a few indecomposables and describe hidden extensions by them. In particular, we will not give an explicit presentation for the cohomology of $\mathcal{A}(2)$. We do have all of the indecomposables as they coincide with those of $E_\infty$ for formal reasons, see table \ref{tab:Indecomposables on E_infty}. However, giving all relations involves solving all hidden extensions. While the author believes it to be possible to resolve most hidden extensions using the methods of section \ref{section: hidden extensions}, it would take a prohibitive amount of time.

Two interesting hidden extensions are $\tau^8 \cdot h_1^4 = \rho^4 \tau^4 P$, and $\tau^8 \cdot \rho^6 g^2  = \rho^{14} \D^2$. For one, the elements involved in these relations are integral to the computation of the Adams spectral sequence for $\tmf$. The element $h_1$ classically represents the Hopf fibration $\eta$. The elements $P$, $g$, and $\D^2$ govern the structure of the Adams spectral sequence for $\tmf$, see \cite[Part 2, Chapter 5]{BR21}\footnote{Our names $P$, $g$, and $\D^2$ translate to $w_1$, $g$, and $w_2$ in \cite{BR21}.}. Secondly, these examples show that the existence of the elements $\tau$ and $\rho$ in $\mathbb{R}$-motivic homotopy theory leads to relations in the cohomology of $\mathcal{A}(2)$ which have no classical, or even $\mathbb{C}$-motivic, analog. The first hidden extension $\tau^8 \cdot h_1^4 = \rho^4 \tau^4 P$ will prove extremely useful when determining other hidden extensions. It implies that any element which has a non-trivial $\rho^4 \tau^4 P$-multiple on $E_\infty$ also needs to support four $h_1$-multiplications, and each one of these $h_1$-multiples has to support a $\tau^8$-multiplication. Oftentimes, the targets of these multiplications are determined for degree reasons. If they turn out to be in higher $\rho$-filtration, then the multiplication in question is in fact a hidden extension. This happens quite frequently.

\subsection{An Adams spectral sequence?}
In the last section, we make the assumption that an $\mathbb{R}$-motivic modular forms spectrum exists, or more precisely that there exists an object in the $\mathbb{R}$-motivic homotopy category such that its mod $2$ cohomology is $\mathcal{A}//\mathcal{A}(2)$. By a standard change of rings argument the $E_2$-page of the Adams spectral sequence associated to this spectrum is the cohomology of $\mathcal{A}(2)$. We compute the $d_2$-differentials on all indecomposables in this hypothetical Adams spectral sequence. The main method is a comparison to the Adams $d_2$-differentials for the $\mathbb{C}$-motivic modular forms spectrum. In principle, once we have the $d_2$-differentials on the indecomposables, all other $d_2$-differentials follow from the Leibniz rule. But since the Leibniz rule involves multiplication, computing the $d_2$-differentials of all elements of the Adams $E_2$-page means solving all hidden extensions. Therefore, we opt not to push the computation of the Adams spectral sequence further.

\subsection{Charts}
There is a sizable amount of charts attached to this article. They depict the $\rho$-Bockstein spectral sequence, as well as the $\rho$-localization and certain quotients of the cohomology of $\mathcal{A}(2)$. Due to their size and number, these charts are not contained in this manuscript but instead can be found at \cite{Charts}. Nevertheless, they are integral to our computation. More details on charts are in section \ref{section: charts}.

\subsection{Outline}
We begin in section \ref{section: notation} by setting up the relevant notation and defining the algebraic objects of interest.

In section \ref{section: the rho-Bockstein spectral sequence} we set up the $\rho$-Bockstein spectral sequence and prove some structural theorems. We also compute the $\rho$-localization of the cohomology of $\mathcal{A}(2)$.

In section \ref{section: charts} we provide a guide on how to read the charts attached to this manuscript. This section serves as a reference and should only be consulted once a chart becomes relevant.

Section \ref{section: rho-Bockstein differentials} contains a computation of all differentials in the $\rho$-Bockstein spectral sequence.

In section \ref{section: hidden extensions} we determine some hidden extensions on the $\rho$-Bockstein $E_\infty$-page. We are mainly concerned with hidden extensions by $\tau^8$, $h_0$, $h_1$, and $h_2$.

Section \ref{section: An Adams spectral sequence?} contains a computation of a hypothetical Adams spectral sequence for an $\mathbb{R}$-motivic modular forms spectrum. We determine the Adams $d_2$-differentials on all indecomposables.

\subsection{Future directions}
There are several interesting open problems arising from the work carried out in this manuscript.

\begin{?}
    Construct an $\mathbb{R}$-motivic modular forms spectrum $\mmf^{\,\,\mathbb{R}}$ whose cohomology is $\mathcal{A}//\mathcal{A}(2)$. In particular, the $E_2$-page of its Adams spectral sequence would be the cohomology of $\mathcal{A}(2)$. Possible methods for this construction include
    \begin{itemize}
        \item abstracting the original construction of $\tmf$ due to Goerss, Hopkins, and Miller, as outlined by Behrens in \cite[Chapter 12]{DFHH} to the motivic setting,
        
        \item abstracting the construction of $\tmf$ due to Lurie \cite{Lur09} \cite{LurEC1} \cite{LurEC2} to the motivic setting,
        
        \item use Galois reconstruction as described by Burklund, Hahn, and Senger in \cite{BHS22}. This would require a $C_2$-equivariant $\tmf$, which has also not been constructed yet.
    \end{itemize}
    
\end{?}

\begin{?}
    Fully describe the multiplicative structure of the cohomology of $\mathbb{R}$-motivic $\mathcal{A}(2)$. Given the size of our computation this will probably involve machine computation.
\end{?}

\begin{?}
    An interesting construction in $\mathbb{C}$-motivic homotopy theory is the cofiber of $\tau$ \cite{GWX}. See \cite{Bae24} for an Adams spectral sequence computation of the $\mathbb{C}$-motivic cofiber of $\tau$ for $\mmf$. In $\mathbb{R}$-motivic homotopy theory, the element $\tau$ and its powers can support $\rho$-Bockstein differentials and therefore not represent maps of spectra. For example, when computing the cohomology of the full $\mathbb{R}$-motivic Steenrod algebra $\mathcal{A}$, every element of the form $\tau^n$, $n \geq 1$, supports a $\rho$-Bockstein differential \cite[Proposition 3.2]{DI}. So there is no map corresponding to $\tau^n$ for the $\mathbb{R}$-motivic sphere spectrum for any $n \geq 1$. As we will see, the situation for $\mathcal{A}(2)$ is slightly different. Because the elements $h_i$ are zero for $i \geq 3$, we have that $\tau^{8}$ and its powers survive the $\rho$-Bockstein spectral sequence. If we assume the existence of an $\mathbb{R}$-motivic modular forms spectrum, then there is a map corresponding to $\tau^8$. So we could take the cofiber of $\tau^8$, written $\mmf^{\,\,\mathbb{R}}/\tau^8$. This leads to multiple natural questions: What does the homotopy of $\mmf^{\,\,\mathbb{R}}/\tau^8$ look like? What is its relation to the $\mathbb{C}$-motivic cofiber of $\tau$? Does the Adams spectral sequence for $\mmf^{\,\,\mathbb{R}}/\tau^8$ admit an algebraic description similar to the $\mathbb{C}$-motivic cofiber of $\tau$, or more generally is there an algebraic description of the category of $\mmf^{\,\,\mathbb{R}}/\tau^8$-modules?
\end{?}

\subsection{Acknowledgements}

The author would like to thank Dan Isaksen for sharing this problem with him and for many helpful conversations and comments. The author would also like to thank Joey Beauvais-Feisthauer for creating the software (most notably the SeqSee package \cite{SeqSee}) used for making the charts accompanying this manuscript.

\section{Notation}\label{section: notation}

Let $\mathbb{M}_2$ denote the $\mathbb{R}$-motivic homology of a point with $\mathbb{F}_2$ coefficients. Recall from \cite{Voe03} that $\mathbb{M}_2 \cong \mathbb{F}_2[\rho, \tau]$ where the degree of $\rho$ is $(-1, -1)$ and the degree of $\tau$ is $(0, -1)$.

Let $\mathcal{A}$ denote the $\mathbb{R}$-motivic dual Steenrod algebra. Recall from \cite{Voe10} that $\mathcal{A}$ is described as an algebra by
\[\mathcal{A} \cong \frac{\mathbb{M}_2[\tau_0, \tau_1, \dotsc, \xi_0,  \xi_1, \dotsc]}{(\xi_0 = 1, \tau_k^2 = \tau \xi_{k+1} + \rho \tau_{k+1} + \rho \tau_0 \xi_{k+1})},\]
where the degree of $\tau_k$ is $(2^{k+1}-1, 2^k-1)$ and the degree of $\xi_k$ is $(2^{k+1}-2, 2^k-1)$. Furthermore, $\mathcal{A}$ has the structure of a Hopf algebroid via
\[\eta_L(\tau) = \tau,~\,\eta_R(\tau) = \tau + \rho \tau_0,~\,\eta_L(\rho) = \eta_R(\rho) = \rho,\]
\[\Delta(\tau_k) = \tau_k \otimes 1 + \sum \xi_{k-i}^{2^i} \otimes \tau_{i},\]
\[\Delta(\xi_k) = \sum \xi_{k-i}^{2^i} \otimes \xi_{i}.\]
Note that $\Delta(\tau) = \tau \otimes 1$ and $\Delta(\rho) = \rho \otimes 1$ are forced by the axioms of a Hopf algebroid.

Then $\mathcal{A}(2)$ is defined to be the Hopf algebroid obtained by quotienting out the ideal $(\tau_3, \tau_4, \dotsc, \xi_1^4, \xi_2^2, \xi_3, \dotsc)$. This quotient is dual to the subalgebra of the $\mathbb{R}$-motivic Steenrod algebra generated by $\Sq^1$, $\Sq^2,$ and $\Sq^4$. As an algebra $\mathcal{A}(2)$ has the structure
\[\mathcal{A}(2) \cong \frac{\mathbb{M}_2[\tau_0, \tau_1, \tau_2, \xi_0, \xi_1, \xi_2]}{(\xi_0 = 1, \xi_1^4 = 0, \xi_2^2 = 0, \tau_k^2 = \tau \xi_{k+1} + \rho \tau_{k+1} + \rho \tau_0 \xi_{k+1})},\]
and it is again a Hopf algebroid via the same formulas for the units and coproduct as for $\mathcal{A}$. In the relations, treat $\tau_3$ and $\xi_3$ as zero.

We write a superscript $\mathbb{C}$ for the same objects in the $\mathbb{C}$-motivic world, e.g. $\mathbb{M}_2^\mathbb{C}$ is the $\mathbb{C}$-motivic homology of a point. Then all of the objects $\mathbb{M}_2^\mathbb{C}$, $\mathcal{A}^\mathbb{C}$, and $\mathcal{A}(2)^\mathbb{C}$ can be obtained from the above formulas by setting $\rho$ equal to $0$.

We write a superscript $\cl$ for the same classical objects, e.g. $\mathcal{A}^\cl$ is the classical dual Steenrod algebra. The classical objects can be obtained from the above by setting $\rho$ equal to $0$ and $\tau$ equal to $1$.

\section{The \texorpdfstring{$\rho$}{rho}-Bockstein spectral sequence}\label{section: the rho-Bockstein spectral sequence}

We recall some basics about the construction of the $\rho$-Bockstein spectral sequence. For more details see \cite{Hill} or \cite{DI}. Take the cobar complex that computes the cohomology of $\mathbb{R}$-motivic $\mathcal{A}(2)$. Filter it by powers of the ideal generated by $\rho$. This gives rise to a spectral sequence. Because of $\mathcal{A}(2)^{\mathbb{C}} \cong \mathcal{A}(2)/\rho$ and $\mathbb{M}_2^{\mathbb{C}} \cong \mathbb{M}_2/\rho$ and the fact that $\rho$-multiplication is injective on $\mathcal{A}(2)$ and $\mathbb{M}_2$, the $E_1$-page of this spectral sequence can be identified with $\Ext_{\mathcal{A}(2)^{\mathbb{C}}}(\mathbb{M}_2^{\mathbb{C}},\mathbb{M}_2^{\mathbb{C}})[\rho]$. So the signature of the spectral sequence is
\[E_1^{s,f,w} = \Ext_{\mathcal{A}(2)^{\mathbb{C}}}^{s,f,w}(\mathbb{M}_2^{\mathbb{C}},\mathbb{M}_2^{\mathbb{C}})[\rho] \implies \Ext_{\mathcal{A}(2)}^{s,f,w}(\mathbb{M}_2,\mathbb{M}_2).\]

\begin{Not}[Gradings]
    We explain the gradings used in the above signature. Recall from the introduction that our goal is to eventually make computations involving an Adams spectral sequence for an alleged $\mathbb{R}$-motivic modular forms spectrum. The target $\Ext_{\mathcal{A}(2)}^{s,f,w}(\mathbb{M}_2,\mathbb{M}_2)$ of our $\rho$-Bockstein spectral sequence is the $E_2$-page of that Adams spectral sequence. The gradings $s$, $f$, and $w$ correspond to the stem, Adams filtration, and motivic weight on that Adams spectral sequence $E_2$-page. So in an Adams-type chart of our spectral sequence the $x$-axis is $s$ and the $y$-axis is $f$.
\end{Not}

Note that we do not explicitly record the grading arising from the $\rho$-filtration as another superscript. As a consequence, all $d_r$-differentials have degree $(-1, 1, 0)$. Note also that by definition of the $\rho$-Bockstein spectral sequence the $d_r$-differentials increase the $\rho$-filtration by $r$. So for example if $x \in \Ext_{\mathcal{A}(2)^{\mathbb{C}}}^{s,f,w}(\mathbb{M}_2^{\mathbb{C}},\mathbb{M}_2^{\mathbb{C}})$ then $d_r(x) = \rho^r y$ for some $y \in \Ext_{\mathcal{A}(2)^{\mathbb{C}}}^{s+r-1,f+1,w+r}(\mathbb{M}_2^{\mathbb{C}},\mathbb{M}_2^{\mathbb{C}})$. This $y$ may of course very well be zero. Additionally each $d_r$ satisfies a Leibniz rule with respect to the multiplication on the $E_r$-page. If $a$ and $b$ are elements on the $E_r$-page then
\[d_r(a \cdot b) = a \cdot d_r(b) + d_r(a) \cdot b.\]
Note that we are working modulo the prime $2$ so we do not have to worry about signs.

We list some very important structural Propositions about this spectral sequence.

\begin{Prop}{\cite[Lemma 3.4]{DI}}\label{only differentials between rho-free elements}
    Every $\rho$-power-torsion element on the $\rho$-Bockstein $E_r$-page is annihilated by $\rho^{r-1}$. As a consequence, if $d_r(x)$ is non-trivial in the $\rho$-Bockstein spectral sequence, then $x$ and $d_r(x)$ are both $\rho$-torsion free on the $E_r$-page. 
\end{Prop}

\begin{proof}
    By definition, if $x$ is an element on the $\rho$-Bockstein $E_r$-page then $d_r(x)$ is a $\rho^r$-multiple. Since the $\rho$-Bockstein $E_1$-page is $\rho$-torsion free, it follows inductively that any $\rho$-power-torsion element on $E_r$ is annihilated by $\rho^{r-1}$.
    
    Now assume $d_r(x)$ is non-trivial. Then $d_r(x)$ must be $\rho$-torsion free as there are no non-trivial $\rho$-power-torsion elements on $E_r$ that are $\rho^r$-multiples. Then note that $\rho$ is a cycle in the cobar complex as its cobar complex differential is $\eta_L(\rho) + \eta_R(\rho) = \rho + \rho = 0$. Consequently, $\rho$ must also be a permanent cycle in the $\rho$-Bockstein spectral sequence. In particular, the $\rho$-Bockstein differential is $\rho$-linear. Now because $d_r(x)$ is $\rho$-torsion free, $x$ also has to be $\rho$-torsion free.
\end{proof}

The second Proposition is a slight variant of \cite[Theorem 4.1]{DI}.

\begin{Prop}\label{rho-local isomorphism}
    There is an isomorphism
    \[\text{\rm Ext}_{\mathcal{A}(1)^{\text{\rm cl}}}(\mathbb{F}_2, \mathbb{F}_2)[\rho^{\pm 1}, \tau^8] \cong \text{\rm Ext}_{\mathcal{A}(2)}(\mathbb{M}_2,\mathbb{M}_2)[\rho^{-1}]\]
    that takes elements of degree $(s, f)$ in $\text{\rm Ext}_{\mathcal{A}^{\text{\rm cl}}(1)}(\mathbb{F}_2, \mathbb{F}_2)$ to elements of degree $(2s + f, f, s + f)$ in $\text{\rm Ext}_{\mathcal{A}(2)}(\mathbb{M}_2,\mathbb{M}_2)$.
\end{Prop}

\begin{proof}
    The proof is almost identical to that of \cite[Theorem 4.1]{DI}. Localizing at $\rho$ commutes with taking cohomology since localization is exact. Note that in $\mathcal{A}(2)[\rho^{-1}]$ we have
    \[\tau_{k+1} = \rho^{-1} \tau_k^2 + \rho^{-1} \tau \xi_{k+1} + \tau_0 \xi_{k+1}.\]
    So $\tau_1$ and $\tau_2$ become decomposable:
    \[\tau_1 = \rho^{-1} \tau_0^2 + \rho^{-1} \tau \xi_{1} + \tau_0 \xi_{1},\]
    \[\tau_2 = \rho^{-1} \tau_1^2 + \rho^{-1} \tau \xi_{2} + \tau_0 \xi_{2}.\]
    Use $\tau_2^2 = 0$, $\xi_2^2 = 0$, and $\xi_1^4 = 0$ to get
    \[0 = \tau_2^2 = \rho^{-2} \tau_1^4 + \rho^{-2} \tau^2 \xi_{2}^2 + \tau_0^2 \xi_{2}^2 = \rho^{-2} (\rho^{-4} \tau_0^8 + \rho^{-4} \tau^4 \xi_{1}^4 + \tau_0^4 \xi_{1}^4) = \rho^{-6} \tau_0^8.\]
    So the Hopf algebroid $(\mathbb{M}_2[\rho^{-1}], \mathcal{A}(2)[\rho^{-1}])$ is described by
    \begin{align*}
        &\mathcal{A}(2)[\rho^{-1}] = \frac{\mathbb{M}_2[\rho^{-1}][\tau_0, \xi_0, \xi_1, \xi_2]}{(\tau_0^8 = 0,\, \xi_0 = 1,\, \xi_1^4 = 0,\, \xi_2^2 = 0)},\\
        &\eta_L(\tau) = \tau,~\,\eta_R(\tau) = \tau + \rho \tau_0,~\,\eta_L(\rho) = \eta_R(\rho) = \rho,\\
        &\Delta(\tau_0) = \tau_0 \otimes 1 + 1 \otimes \tau_0,\\
        &\Delta(\xi_k) = \sum \xi_{k-i}^{2^i} \otimes \xi_i.
    \end{align*}
    Notice how $\tau_0$ and $\xi_k$ do not appear in the same formulas together. This means there is a splitting
    \[(\mathbb{M}_2[\rho^{-1}], \mathcal{A}(2)[\rho^{-1}]) \cong (\mathbb{M}_2[\rho^{-1}], \mathcal{A}(2)') \otimes_{\mathbb{F}_2} (\mathbb{F}_2, \mathcal{A}(2)''),\]
    where $(\mathbb{M}_2[\rho^{-1}], \mathcal{A}(2)')$ is the Hopf algebroid
    \begin{align*}
        &\mathcal{A}(2)' = \mathbb{M}_2[\rho^{-1}][\tau_0]/(\tau_0^8),\\
        &\eta_L(\tau) = \tau,~\eta_R(\tau) = \tau + \rho \tau_0,\\
        &\Delta(\tau_0) = \tau_0 \otimes 1 + 1 \otimes \tau_0,
    \end{align*}
    and $(\mathbb{F}_2, \mathcal{A}(2)'')$ is the Hopf algebra
    \begin{align*}
        &\mathcal{A}(2)'' = \mathbb{F}_2[\xi_0, \xi_1, \xi_2]/(\xi_0 = 1, \xi_1^4 = 0, \xi_2^2 = 0),\\
        &\Delta(\xi_k) = \sum \xi_{k-i}^{2^i} \otimes \xi_i.
    \end{align*}
    The Hopf algebra $(\mathbb{F}_2, \mathcal{A}(2)'')$ is up to a degree shift exactly the same as $\mathcal{A}(1)^{\cl}$. So its cohomology agrees up to a degree shift with the known cohomology of $\mathcal{A}(1)^{\cl}$. For $(\mathbb{M}_2[\rho^{-1}], \mathcal{A}(2)')$, setting $x = \rho \tau_0$ yields
    \[(\mathbb{M}_2[\rho^{-1}], \mathcal{A}(2)') \cong \mathbb{F}_2[\rho^{\pm 1}] \otimes_{\mathbb{F}_2} (\mathbb{F}_2[\tau], \mathbb{F}_2[\tau, x]/(x^8)).\]
    The following Lemma (which is a slight variant of \cite[Lemma 4.3]{DI}) then shows that the cohomology of the second tensor factor is $\mathbb{F}_2[\tau^8]$.
\end{proof}

\begin{Lemma}
    Consider the Hopf algebroid $(\mathbb{F}_2[t], \mathbb{F}_2[t,x]/(x^8))$ with the structure maps
    \begin{align*}
        \eta_L(t) = t,~\eta_R(t) = t + x,~\Delta(x) = x \otimes 1 + 1 \otimes x.
    \end{align*}
    Then the cohomology of this Hopf algebroid is $\mathbb{F}_2[t^8]$.
\end{Lemma}

\begin{proof}
    Run the $x$-Bockstein spectral sequence to compute the cohomology of $(\mathbb{F}_2[t], \mathbb{F}_2[t,x]/(x^8))$. Arguments similar to \cite[Lemma 4.3]{DI} show that the $E_1$-page is $\mathbb{F}_2[t, h_0, h_1, h_2]$, where $h_i = [x^{2^i}]$. Again following \cite[Lemma 4.3]{DI} the differentials in this spectral sequence are $d_1(t) = h_0$, $d_2(t^2) = h_1$, and $d_4(t^4) = h_2$. So the $E_\infty$-page is $\mathbb{F}_2[t^8]$ and there is no room for hidden extensions.
\end{proof}

From this point onward we will freely make use of the computation \cite{Isa09} of the cohomology of $\mathcal{A}(2)^{\mathbb{C}}$, meaning $\Ext_{\mathcal{A}(2)^{\mathbb{C}}}(\mathbb{M}_2^{\mathbb{C}}, \mathbb{M}_2^{\mathbb{C}})$. In the time that has passed since that computation was carried out, there have been slight changes in notation. We will use the names of the generators as presented in \cite[Table 1]{Isa18}. The only changes are that $\alpha$ became $a$ and $\nu$ became $n$. We reproduce a table of these generators below in Table \ref{tab:Generators of the cohomology of A(2)^C}. The first column provides the name of the generator, the second column describes its stem $s$, Adams filtration $f$, and weight $w$.

\begin{longtable}{ll}
\caption{Multiplicative generators of the cohomology of $\mathcal{A}(2)^{\mathbb{C}}$
\label{tab:Generators of the cohomology of A(2)^C}
} \\
\toprule
generator & $(s,f,w)$ \\
\midrule \endfirsthead
\caption[]{Multiplicative generators of the cohomology of $\mathcal{A}(2)^{\mathbb{C}}$} \\
\toprule
generator & $(s,f,w)$ \\
\midrule \endhead
\bottomrule \endfoot
    $\tau$ & $(0, 0, -1)$ \\
    $h_0$ & $(0, 1, 0)$ \\
    $h_1$ & $(1, 1, 1)$ \\
    $h_2$ & $(3, 1, 2)$ \\
    $c$ & $(8, 3, 5)$ \\
    $P$ & $(8, 4, 4)$ \\
    $u$ & $(11, 3, 7)$ \\
    $a$ & $(12, 3, 6)$ \\
    $d$ & $(14, 4, 8)$ \\
    $n$ & $(15, 3, 8)$ \\
    $e$ & $(17, 4, 10)$ \\
    $g$ & $(20, 4, 12)$ \\
    $\Dh$ & $(25, 5, 13)$ \\
    $\Dc$ & $(32, 7, 17)$ \\
    $\Du$ & $(35, 7, 19)$ \\
    $\D^2$ & $(48, 8, 24)$ \\
\end{longtable}

The relations that these generators satisfy are spelled out in \cite[Theorem 4.13]{Isa09}. Since the $E_1$-page of the $\rho$-Bockstein spectral sequence is isomorphic to $\Ext_{\mathcal{A}(2)^{\mathbb{C}}}(\mathbb{M}_2^{\mathbb{C}},\mathbb{M}_2^{\mathbb{C}})[\rho]$, the generators and relations referred to above also describe that $E_1$-page.

Next, we provide a presentation and a table of generators for the $\rho$-localized cohomology of $\mathcal{A}(2)$. The generators $h_1, h_2, u,$ and $g$ listed below are the degree-shifted versions (according to Proposition \ref{rho-local isomorphism}) of the classical generators of the cohomology of $\mathcal{A}(1)^{\cl}$. Similarly, the relations are given by taking the relations from $\mathcal{A}(1)^{\cl}$ and substituting $h_1, h_2, u,$ and $g$ for the usual variables. A chart of this algebra is available at \cite{Charts} in the folder \texttt{Cohomology of A(2) rho-localized}. For a guide on how to read it see section \ref{section: charts} and specifically section \ref{charts: cohomology of A(2) rho-localized}.

\begin{Cor}\label{representation of rho-localized target}
    Under the localization map
    \[\text{\rm Ext}_{\mathcal{A}(2)}(\mathbb{M}_2,\mathbb{M}_2) \to \text{\rm Ext}_{\mathcal{A}(2)}(\mathbb{M}_2,\mathbb{M}_2)[\rho^{-1}] \cong \text{\rm Ext}_{\mathcal{A}(1)^{\text{\rm cl}}}(\mathbb{F}_2, \mathbb{F}_2)[\rho^{\pm 1}, \tau^8]\]
    the four elements of $\text{\rm Ext}_{\mathcal{A}(2)}(\mathbb{M}_2,\mathbb{M}_2)$ represented by $h_1$, $h_2$, $u$, and $g$ on the $\rho$-Bockstein $E_1$-page map to the four indecomposables of $\text{\rm Ext}_{\mathcal{A}(1)^{\text{\rm cl}}}(\mathbb{F}_2, \mathbb{F}_2)$. If we denote their images by the same symbols, then there is an isomorphism of algebras
    \[\text{\rm Ext}_{\mathcal{A}(2)}(\mathbb{M}_2,\mathbb{M}_2)[\rho^{-1}] \cong \frac{\mathbb{F}_2[\rho^{\pm 1}, \tau^8, h_1, h_2, u, g]}{(h_1 h_2, h_2^3, h_2 u, u^2 + h_1^2 g)},\]
    where the degree of each generator is listed in table \ref{tab:rho-localized cohomology of A(2)}.

    \begin{longtable}{ll}
    \caption{Multiplicative generators of $\text{\rm Ext}_{\mathcal{A}(2)}(\mathbb{M}_2,\mathbb{M}_2)[\rho^{-1}]$
    \label{tab:rho-localized cohomology of A(2)}
    } \\
    \toprule
    \rm{generator} & $(s,f,w)$\\
    \midrule \endfirsthead
    \caption[]{Multiplicative generators of $\text{\rm Ext}_{\mathcal{A}(2)}(\mathbb{M}_2,\mathbb{M}_2)[\rho^{-1}]$} \\
    \toprule
    \rm{generator} & $(s,f,w)$\\
    \midrule \endhead
    \bottomrule \endfoot
        $\rho$ & $(-1, 0, -1)$\\
        $\tau^8$ & $(0, 0, -8)$\\
        $h_1$ & $(1, 1, 1)$\\
        $h_2$ & $(3, 1, 2)$\\
        $u$ & $(11, 3, 7)$\\
        $g$ & $(20, 4, 12)$\\
    \end{longtable}
\end{Cor}

\begin{proof}
    By Proposition \ref{rho-local isomorphism} we have a degree-shifting isomorphism
    \[\text{\rm Ext}_{\mathcal{A}(1)^{\text{\rm cl}}}(\mathbb{F}_2, \mathbb{F}_2)[\rho^{\pm 1}, \tau^8] \cong \text{\rm Ext}_{\mathcal{A}(2)}(\mathbb{M}_2,\mathbb{M}_2)[\rho^{-1}].\]
    Computing $\text{\rm Ext}_{\mathcal{A}(1)^{\text{\rm cl}}}(\mathbb{F}_2, \mathbb{F}_2)$ is classical, see for example \cite[3.1.25. Theorem]{Rav86}. Applying the computation \cite[3.1.25. Theorem]{Rav86} to the Hopf algebra $(\mathbb{F}_2, \mathcal{A}(2)'')$ from the proof of Proposition \ref{rho-local isomorphism} will give us cobar level representatives of the indecomposables in its cohomology. Recall that $(\mathbb{F}_2, \mathcal{A}(2)'')$ is given by
    \begin{align*}
        &\mathcal{A}(2)'' = \mathbb{F}_2[\xi_0, \xi_1, \xi_2]/(\xi_0 = 1, \xi_1^4 = 0, \xi_2^2 = 0),\\
        &\Delta(\xi_k) = \sum \xi_{k-i}^{2^i} \otimes \xi_i.
    \end{align*}
    Then computing its cohomology can be done in exactly the same way as \cite[3.1.25. Theorem]{Rav86} using a Cartan-Eilenberg spectral sequence. The resulting cohomology will be an algebra on four indecomposables, represented on the cobar level by $\xi_1$, $\xi_1^2$, $\xi_1 \otimes \bar{\xi}_2 \otimes \bar{\xi}_2$, and $\bar{\xi}_2 \otimes \bar{\xi}_2 \otimes \bar{\xi}_2 \otimes \bar{\xi}_2$. Here, a bar means applying the antipode in the Hopf algebra. Using the gradings provided in section \ref{section: notation}, one finds that these four elements are located in an Adams-type chart in the degrees $(1, 1, 1)$, $(3, 1, 2)$, $(11, 3, 7)$, and $(20, 4, 12)$. Comparing this with our $\rho$-Bockstein $E_1$-page, we see that there is only one element in each given degree, namely $h_1$, $h_2$, $u$, and $g$, respectively. So these four elements must map to the four indecomposables in $\text{\rm Ext}_{\mathcal{A}(1)^{\text{\rm cl}}}(\mathbb{F}_2, \mathbb{F}_2)$ under the localization map
    \[\text{\rm Ext}_{\mathcal{A}(2)}(\mathbb{M}_2,\mathbb{M}_2) \to \text{\rm Ext}_{\mathcal{A}(2)}(\mathbb{M}_2,\mathbb{M}_2)[\rho^{-1}] \cong \text{\rm Ext}_{\mathcal{A}(1)^{\text{\rm cl}}}(\mathbb{F}_2, \mathbb{F}_2)[\rho^{\pm 1}, \tau^8].\]
    
    The claimed presentation for $\text{\rm Ext}_{\mathcal{A}(2)}(\mathbb{M}_2,\mathbb{M}_2)[\rho^{-1}]$ is obtained by taking the classical presentation of $\text{\rm Ext}_{\mathcal{A}(1)^{\text{\rm cl}}}(\mathbb{F}_2, \mathbb{F}_2)$ and replacing the names of the generators.
\end{proof}

\begin{Cor}\label{t^8, h1, h2, u, g are permanent cycles}
    The $\rho$-Bockstein $E_1$-page elements $\tau^8$, $h_1$, $h_2$, $u$, and $g$ are permanent cycles in the $\rho$-Bockstein spectral sequence, i.e. they do not support any differentials.
\end{Cor}

\begin{proof}
    By Corollary \ref{representation of rho-localized target} the four elements in question represent non-zero cohomology classes after $\rho$-localization. Since localization commutes with taking cohomology the claim follows.
\end{proof}

\begin{Rem}
    From the multiplicative structure of the $\rho$-localization of the cohomology of $\mathcal{A}(2)$ given in Corollary \ref{representation of rho-localized target} we can already see that there will be many hidden extensions on the $\rho$-Bockstein $E_\infty$-page. For example, many $E_1$-page elements that are $\tau$-power-torsion need to represent a $\tau^8$-torsion free element. Specific examples include $h_1^4$, $u$, $h_1^2 g$, $h_1 g^2$, and all of their $h_1$-multiples. There will be hidden $\tau^8$-extensions on all of these classes on $E_\infty$.

    In particular, Corollary \ref{representation of rho-localized target} does not tell us the exact degree of the $\rho$-torsion free classes that will represent these products on $E_\infty$. What it does provide us is the value for $f$ because $\rho$-multiplication preserves $f$. It also yields a lower bound for $s$ and $w$ which is given by summing up the values of $s$ and $w$ of the factors as given in table \ref{tab:rho-localized cohomology of A(2)}.
\end{Rem}

\begin{Rem}
    Note that Corollary \ref{representation of rho-localized target} does not imply that multiples of $\tau^8$, $h_1$, $h_2$, $u$, and $g$ cannot be involved in differentials. For example we will see in Remark \ref{d_6 hitting tau^8 g^2} that $\rho^6 \cdot \tau^8 \cdot g^2$ has to get hit by a differential for multiplicative reasons. That means $\tau^8 \cdot g^2$ is a $\rho$-power-torsion element in our spectral sequence. But by Corollary \ref{representation of rho-localized target} we know that it maps to a non-zero element in the $\rho$-localized cohomology of $\mathcal{A}(2)$. This implies that there must be a hidden extension in the $\rho$-Bockstein spectral sequence turning $\tau^8 \cdot g^2$ into a $\rho$-torsion free element. For degree reasons, that hidden extension will have to involve $\Delta^2$.
\end{Rem}

Just as in \cite[Proposition 3.2]{DI}, we can find the $\rho$-Bockstein differentials on the powers of $\tau$ by making cobar complex computations.

\begin{Lemma}\label{Bockstein d_r on powers of tau}
    The $\rho$-Bockstein differentials on the powers of $\tau$ are given by
    \begin{align*}
        d_1(\tau) &= \rho h_0,\\
        d_2(\tau^2) &= \rho^2 \tau h_1,\\
        d_4(\tau^4) &= \rho^4 \tau^2 h_2,\\
        d_r(\tau^8) &= 0 \text{ for all $r \geq 1$}.
    \end{align*}
\end{Lemma}

\begin{proof}
    The same cobar complex computations as in \cite[Proposition 3.2]{DI} prove the differentials on $\tau$, $\tau^2$, and $\tau^4$. Then $\tau^8$ must be a permanent cycle for degree reasons (or because its cobar differential is $\eta_L(\tau^8) + \eta_R(\tau^8) = \tau^8 + \tau^8 = 0$).

    Alternatively, one could also consider Corollary \ref{representation of rho-localized target} which implies that the degrees which contain the elements $\tau$, $\tau^2$, and $\tau^4$ cannot have any $\rho$-torsion free elements on $E_\infty$. The claimed differentials are the only way for that to happen.
\end{proof}

\begin{Not}
    We will make use of compound names throughout. We already did this in Corollary \ref{representation of rho-localized target} by calling an indecomposable $\tau^8$, or in Lemma \ref{Bockstein d_r on powers of tau} by referring to $\tau h_1$ even though this element is indecomposable on the $E_2$-page. Whenever it is helpful to know an explicit multiplicative decomposition for an element we will indicate it. Then instead of $x y$ we will write $x \cdot y$ to indicate that $x y$ is the product of $x$ and $y$.
\end{Not}

\section{Charts}\label{section: charts}

For a computation of this size, charts are indispensable. Due to the large number of charts needed, they will not be included in this manuscript. Instead, the charts can be found at \cite{Charts}.

This section is meant to serve as a reference. We chose to put it here to notify the reader that from now on charts will become relevant. In practice, this section can be skipped until the charts in question are needed.

\subsection{Generalities}

Let us first collect some general properties that hold true for all charts:
\begin{itemize}
    \item The horizontal axis describes the stem $s$,
    \item the vertical axis describes the Adams filtration $f$.
\end{itemize}
Consequently, we have to plot $\tau$-multiples implicitly since $\tau$ has both stem and Adams filtration equal to $0$. Most often, we do the following:
\begin{itemize}
    \item Colors indicate behavior with respect to $\tau$-multiplication. The exact meaning of each color depends on the individual chart and is discussed in the sections below.
\end{itemize}
Another important element is $\rho$. By Proposition \ref{only differentials between rho-free elements} we can work exclusively with $\rho$-torsion free classes until we reach $E_{\infty}$. Therefore, we do the following:
\begin{itemize}
    \item We try to avoid plotting $\rho$-multiples whenever possible.
\end{itemize}
As a consequence, the only charts where $\rho$-multiples are plotted explicitly are the $\rho$-power-torsion charts and the charts for the $\rho$-free quotient of the cohomology of $\mathcal{A}(2)$.

For each chart there are two versions available: A pdf version and an html version. The html version always provides more information. Hovering over a dot in an html chart shows its name, degrees and other data. In particular, we can quickly find the weight of an element in the chart, assuming we know its stem and Adams filtration. The drawback of the html charts is that the name of each element is not directly printed onto the chart, so it might be harder to find a desired element at first glance. The pdf version has the name of each element\footnote{Except for $h_0$-, $h_1$-, and $h_2$-multiples, but the fact that they are multiples is typically visualized by a line in the chart so their name can be inferred easily from the dot that they are connected to via that line.} printed on the chart. That can make it easier to find a desired element, as long as there aren't too many elements in the same stem and Adams filtration. In summary, the pdf charts are simpler to use when unfamiliar with the names of elements and their degrees, but the html charts are a more efficient tool because the weights are integral to our computation. Ultimately, which version to use depends on the specific chart and personal preference.

We will now give detailed instructions on how to read each set of charts.

\subsection{Cohomology of \texorpdfstring{$\mathcal{A}(2)$}{A(2)} \texorpdfstring{$\rho$}{rho}-localized}\label{charts: cohomology of A(2) rho-localized}

This chart shows the $\rho$-localization of the cohomology of $\mathbb{R}$-motivic $\mathcal{A}(2)$ as computed in Corollary \ref{representation of rho-localized target}. It contains the following information:
\begin{itemize}
    \item A dot represents the module $\mathbb{F}_2[\rho^{\pm 1}, \tau^8]$,
    \item lines of slope $1$ represent $h_1$-multiplication, lines of slope $1/3$ represent $h_2$-multiplication,
    \item an arrow of slope $1$ indicates that a given dot supports infinitely many $h_1$-multiplications.
\end{itemize}

\subsection{\texorpdfstring{$\rho$}{rho}-Bockstein \texorpdfstring{$E_r$}{E\_r}-pages \texorpdfstring{$\rho$}{rho}-free quotient}\label{charts: rho-Bockstein E_r-pages rho-free quotient}

By Proposition \ref{only differentials between rho-free elements} there are only $\rho$-Bockstein differentials between $\rho$-torsion free classes. As such, we can work in the $\rho$-free quotient $E_r/(\text{$\rho$-power-torsion})$. Since every element in this quotient is $\rho$-torsion free, we leave out all $\rho$-multiples in the charts.

Generally, the following will hold true for all pages:
\begin{itemize}
    \item Dots represent a quotient of $\mathbb{F}_2[\rho, \tau^n]$, where the power $n$ depends on the page,
    \item the color of a dot determines the $\tau^n$-torsion of that quotient,
    \item vertical lines represent $h_0$-multiplication, lines of slope $1$ represent $h_1$-multiplication, lines of slope $1/3$ represent $h_2$-multiplication,
    \item the color of a line simply corresponds to the color of its target dot, except when the line is colored magenta\footnote{Here is an example of what a magenta line looks like. \textcolor{magenta}{\begin{tikzcd}[ampersand replacement = \&]{}\ar[r,thick,-]\&{}\end{tikzcd}}}, which indicates that the given multiplication hits the $\tau^n$-multiple of the target dot,
    \item two lines originating from the same dot implies that the $h_i$-multiplication in question hits the sum of those two targets,
    \item arrows indicate that a given dot supports infinitely many multiplications of the type corresponding to the slope of the arrow, e.g. a vertical arrow indicates infinitely many $h_0$-multiples, an arrow of slope $1$ indicates infinitely many $h_1$-multiples,
    \item the color of an arrow indicates the $\tau^n$-torsion of the implied multiples.
\end{itemize}

\subsubsection{\texorpdfstring{$E_1$}{E\_1}}\label{charts: rho-Bockstein E_1-page rho-free quotient}

Recall that by construction the $\rho$-Bockstein $E_1$-page is $\Ext_{\mathcal{A}(2)^{\mathbb{C}}}(\mathbb{M}_2^{\mathbb{C}},\mathbb{M}_2^{\mathbb{C}})[\rho]$, which is entirely $\rho$-torsion free. Hence the $\rho$-free quotient of $E_1$ is just $E_1$ itself. Since we avoid plotting $\rho$-multiples, our chart will look like $\Ext_{\mathcal{A}(2)^{\mathbb{C}}}(\mathbb{M}_2^{\mathbb{C}},\mathbb{M}_2^{\mathbb{C}})$. Charts of this already exist in the literature, e.g. \cite[p. 11]{Isa18}. Our chart will be modeled exactly after this cited one\footnote{The cited chart also includes $\mathbb{C}$-motivic Adams differentials. Those will of course not be in our chart.}. We provide a table on how to interpret the colors in our chart.

\begin{longtable}{lcl}
\caption{Colors for the $\rho$-Bockstein $E_1$-page $\rho$-free quotient
\label{tab:Colors for rho-Bockstein E_1 rho-free quotient}
} \\
\toprule
Color & Example & Module\\
\midrule \endfirsthead
\caption[]{Colors for the $\rho$-Bockstein $E_1$-page $\rho$-free quotient} \\
\toprule
Color & Example & Module\\
\midrule \endhead
\bottomrule \endfoot
    gray & \textcolor{chartgray}{$\bullet$} & $\mathbb{F}_2[\rho, \tau]$\\
    red & \textcolor{red}{$\bullet$} & $\mathbb{F}_2[\rho, \tau]/\tau$\\
    blue & \textcolor{blue}{$\bullet$} & $\mathbb{F}_2[\rho, \tau]/\tau^2$\\
    green & \textcolor{darkgreen}{$\bullet$} & $\mathbb{F}_2[\rho, \tau]/\tau^3$\\
\end{longtable}

\subsubsection{\texorpdfstring{$E_2$}{E\_2}}\label{charts: rho-Bockstein E_2-page rho-free quotient}

By Lemma \ref{Bockstein d_r on powers of tau} there is no element named $\tau$ on $E_2$ because $\tau$ supports a $d_1$. However, we do still have $\tau^2$. So all dots will now represent $\mathbb{F}_2[\rho, \tau^2]$-modules. Sometimes, this leads to more dots at some coordinate in the chart than there were before. That is because not every $\tau$-multiple on $E_1$ has to support a differential. For example at the coordinates $(1, 1)$ there are two dots, namely $h_1$ and $\tau h_1$. The first one represents the "even exponent $\tau$-multiples" of $h_1$ and the second one the "odd exponent $\tau$-multiples" of $h_1$. Of course one should be careful when saying this as there is no $\tau$-multiplication on $E_2$, only $\tau^2$-multiplication.

We will use the same colors as for $E_1$, but replace $\tau$ by $\tau^2$ everywhere. For example, on $E_1$ red was $\tau$-torsion, so on $E_2$ red now stands for $\tau^2$-torsion. Also, a magenta line on $E_2$ now indicates that the corresponding multiplication hits the $\tau^2$-multiple of the target dot.

\begin{longtable}{lcl}
\caption{Colors for the $\rho$-Bockstein $E_2$-page $\rho$-free quotient
\label{tab:Colors for rho-Bockstein E_2 rho-free quotient}
} \\
\toprule
Color & Example & Module\\
\midrule \endfirsthead
\caption[]{Colors for the $\rho$-Bockstein $E_2$-page $\rho$-free quotient} \\
\toprule
Color & Example & Module\\
\midrule \endhead
\bottomrule \endfoot
    gray & \textcolor{chartgray}{$\bullet$} & $\mathbb{F}_2[\rho, \tau^2]$\\
    red & \textcolor{red}{$\bullet$} & $\mathbb{F}_2[\rho, \tau^2]/\tau^2$\\
    blue & \textcolor{blue}{$\bullet$} & $\mathbb{F}_2[\rho, \tau^2]/(\tau^2)^2$\\
\end{longtable}

\subsubsection{\texorpdfstring{$E_3$}{E\_3} and \texorpdfstring{$E_4$}{E\_4}}\label{charts: rho-Bockstein E_3- and E_4-page rho-free quotient}

The smallest power of $\tau$ that exists on these pages is $\tau^4$. We adjust the meaning of the colors just as we did when changing from $E_1$ to $E_2$. A magenta line now indicates that the corresponding multiplication hits the $\tau^4$-multiple of the target dot.

\begin{longtable}{lcl}
\caption{Colors for the $\rho$-Bockstein $E_3$-page and $E_4$-page $\rho$-free quotient
\label{tab:Colors for rho-Bockstein E_3 and E_4 rho-free quotient}
} \\
\toprule
Color & Example & Module\\
\midrule \endfirsthead
\caption[]{Colors for the $\rho$-Bockstein $E_3$-page and $E_4$-page $\rho$-free quotient} \\
\toprule
Color & Example & Module\\
\midrule \endhead
\bottomrule \endfoot
    gray & \textcolor{chartgray}{$\bullet$} & $\mathbb{F}_2[\rho, \tau^4]$\\
    red & \textcolor{red}{$\bullet$} & $\mathbb{F}_2[\rho, \tau^4]/\tau^4$\\
\end{longtable}

\subsubsection{\texorpdfstring{$E_r$}{E\_r} for \texorpdfstring{$r \geq 5$}{r greater or equal to 5}}\label{charts: rho-Bockstein E_r-page rho-free quotient for r >= 5}

The smallest power of $\tau$ that exists on these pages is $\tau^8$.

\begin{longtable}{lcl}
\caption{Colors for the $\rho$-Bockstein $E_r$-page $\rho$-free quotient, $r\geq 5$
\label{tab:Colors for rho-Bockstein E_r rho-free quotient, r>=5}
} \\
\toprule
Color & Example & Module\\
\midrule \endfirsthead
\caption[]{Colors for the $\rho$-Bockstein $E_r$-page $\rho$-free quotient, $r\geq 5$} \\
\toprule
Color & Example & Module\\
\midrule \endhead
\bottomrule \endfoot
    gray & \textcolor{chartgray}{$\bullet$} & $\mathbb{F}_2[\rho, \tau^8]$\\
    red & \textcolor{red}{$\bullet$} & $\mathbb{F}_2[\rho, \tau^8]/\tau^8$\\
\end{longtable}

\subsection{\texorpdfstring{$\rho$}{rho}-Bockstein internal coweight charts}\label{charts: rho-Bockstein internal coweight charts}

The internal coweight is defined in Definition \ref{definition of internal coweight}. We provide charts for fixed values of the internal coweight on various $E_r$-page $\rho$-free quotients. These charts cover all values relevant to the computations of $\rho$-Bockstein differentials as carried out in section \ref{section: rho-Bockstein differentials}. It should be noted that these charts can be re-used on higher pages. For example on $E_3$ we will refer to internal coweight $7$, so then we can re-use the $E_2$-chart for internal coweight $7$.

The charts contain the following information:
\begin{itemize}
    \item A dot represents the module $\mathbb{F}_2[\rho]$,
    \item the color of a dot follows the coloring scheme of the respective $E_r$-page the chart is based on,\footnote{But it should be noted that a dot does not represent any $\tau^n$-multiples here, as $\tau^n$ does not have internal coweight $0$.}
    \item an arrow pointing to the left indicates that the given dot is not involved in any differentials, i.e. it represents a $\rho$-torsion free class on $E_\infty$,
    \item a cyan line\footnote{Here is an example of what a cyan line looks like. \textcolor{darkcyan}{\begin{tikzcd}[ampersand replacement = \&]{}\ar[r,thick,-]\&{}\end{tikzcd}}} connecting two dots indicates a $\rho$-Bockstein differential. The length of the line (or equivalently the relative position of the two dots it connects) implies which page the differential occurs on. More precisely, a line of slope $1/(r-1)$ represents a $d_r$-differential from the dot in lower Adams filtration to the $\rho^r$-multiple of the dot in higher Adams filtration.
\end{itemize}

Some charts show all differentials, e.g. figure \ref{fig:Rho-Bockstein E_2 rho-free part s+f-w=5 with differentials}, and some charts do not show differentials, e.g. figure \ref{fig:Rho-Bockstein E_2 rho-free part s+f-w=5 without differentials}. We never assume knowledge of a completed chart as in figure \ref{fig:Rho-Bockstein E_2 rho-free part s+f-w=5 with differentials}. Instead, if an internal coweight chart is needed in section \ref{section: rho-Bockstein differentials} we will explicitly determine the differentials for the relevant elements on that chart.

\subsection{Cohomology of \texorpdfstring{$\mathcal{A}(2)$}{A(2)} \texorpdfstring{$\rho$}{rho}-free quotient}\label{charts: cohomology of A(2) rho-free quotient}

We provide two charts which combine into the $\rho$-free quotient of the cohomology of $\mathbb{R}$-motivic $\mathcal{A}(2)$, i.e. $\Ext_{\mathcal{A}(2)}(\mathbb{M}_2,\mathbb{M}_2)/(\text{$\rho$-power-torsion})$. The first chart shows elements of coweight $s - w$ congruent to $0$, $1$, and $2$ mod $8$. The second chart shows elements of coweight congruent to $4$ mod $8$. In particular, there are no $\rho$-torsion free elements in other coweights\footnote{That is a direct consequence of Corollary \ref{representation of rho-localized target}.}. The main reason we use coweights here is because $h_1$- and $\rho$-multiplication respect coweights. We consider them modulo $8$ so that $\tau^8$-multiples are on the same chart. The first chart combines the three coweights $0$, $1$, and $2$ so that all $h_2$-multiples are on the same chart.

The charts contain the following information:
\begin{itemize}
    \item A dot represents the module $\mathbb{F}_2[\tau^8]$,
    \item lines of slope $1$ represent $h_1$-multiplication, lines of slope $1/3$ represent $h_2$-multiplication, horizontal lines represent $\rho$-multiplication,
    \item horizontal arrows indicate that a given dot supports infinitely many $\rho$-multiplications,
    \item dotted lines indicate that a multiplication comes from a hidden extension,
    \item magenta lines indicate that a multiplication hits the $\tau^8$-multiple of the target dot. All magenta lines are consequences of hidden $\tau^8$-extensions.
\end{itemize}

There are multiplications that are not hidden on the generator of a dot itself, but on the $\tau^8$-multiple of that generator. We have chosen not to introduce a separate symbol for this phenomenon, and simply draw a solid line. These types of hidden multiplications are always consequences of other hidden extensions. For example, the $h_1$-multiplication from $h_1^3$ to $h_1^4$ is not hidden. However, the $h_1$-multiplication on $\tau^8 \cdot h_1^3$ is hidden with value $\rho^4 \cdot \tau^4 P$. It is a consequence of the horizontal magenta line connecting the dot $\rho^3 \cdot \tau^4 P$ to the dot $h_1^4$, indicating a hidden $\tau^8$-extension from $h_1^4$ to $\rho^4 \cdot \tau^4 P$.

\subsection{\texorpdfstring{$\rho$}{rho}-Bockstein \texorpdfstring{$E_\infty$}{E\_infinity}-page \texorpdfstring{$\rho$}{rho}-power-torsion}\label{charts: rho-Bockstein E_infinity-page rho-power-torsion}

There are eight charts combining into the $\rho$-power-torsion subalgebra of the $E_\infty$-page of the $\rho$-Bockstein spectral sequence. The charts are divided into coweights $s - w$ mod $8$. The main reason we use coweights here is because $h_1$- and $\rho$-multiplication respect coweights. We consider them modulo $8$ so that $\tau^8$-multiples are on the same chart. These charts do not show any hidden extensions, although some are discussed in section \ref{section: hidden extensions on rho-torsion elements}.

The charts contain the following information:
\begin{itemize}
    \item A dot represents a quotient of $\mathbb{F}_2[\tau^8]$,
    \item the color of a dot determines the $\tau^8$-torsion of that quotient, see table \ref{tab:Colors for rho-Bockstein E_infinity rho-torsion},
    \item vertical lines represent $h_0$-multiplication, lines of slope $1$ represent $h_1$-multiplication, lines of slope $1/3$ represent $h_2$-multiplication, horizontal lines represent $\rho$-multiplication,
    \item the color of a line simply corresponds to the color of its target dot,
    \item arrows indicate that a given dot supports infinitely many multiplications of the type corresponding to the slope of the arrow,
    \item the color of an arrow indicates the $\tau^8$-torsion of the implied multiples.
\end{itemize}

\begin{longtable}{lcl}
\caption{Colors for the $\rho$-Bockstein $E_\infty$-page $\rho$-power-torsion
\label{tab:Colors for rho-Bockstein E_infinity rho-torsion}
} \\
\toprule
Color & Example & Module\\
\midrule \endfirsthead
\caption[]{Colors for the $\rho$-Bockstein $E_\infty$-page $\rho$-power-torsion} \\
\toprule
Color & Example & Module\\
\midrule \endhead
\bottomrule \endfoot
    gray & \textcolor{chartgray}{$\bullet$} & $\mathbb{F}_2[\tau^8]$\\
    red & \textcolor{red}{$\bullet$} & $\mathbb{F}_2[\tau^8]/\tau^8$\\
\end{longtable}

\section{\texorpdfstring{$\rho$}{rho}-Bockstein differentials}\label{section: rho-Bockstein differentials}

In this section, we compute all differentials in the $\rho$-Bockstein spectral sequence converging to the cohomology of $\mathbb{R}$-motivic $\mathcal{A}(2)$. There are charts provided at \cite{Charts} which show the $\rho$-free quotient $E_r/(\text{$\rho$-power-torsion})$ of each $E_r$-page of the $\rho$-Bockstein spectral sequence. For a guide on how to read these charts see section \ref{section: charts}, and specifically section \ref{charts: rho-Bockstein E_r-pages rho-free quotient}.

For the reader's convenience and for future reference, we begin with a table collecting all differentials in the $\rho$-Bockstein spectral sequence on indecomposable elements. From left to right, the columns of table \ref{tab:Bockstein d_r} describe
\begin{enumerate}[label=\rm{(\arabic*)}]
    \item the name of the indecomposable supporting a differential,
    \item the degree of that indecomposable,
    \item the name of the target of the differential,
    \item the index of the page where the differential occurs,
    \item where the proof of that differential can be found.
\end{enumerate}
All elements listed are indecomposables on the respective $E_r$-pages where they support a $d_r$-differential. Unlisted indecomposables do not support a differential. In particular, by using the Leibniz rule all other differentials in the spectral sequence can be deduced from the differentials in this table. The rest of this section will then be dedicated to proving the existence of these differentials.

\begin{longtable}{llllc}
\caption{All non-zero differentials on indecomposables in the $\rho$-Bockstein spectral sequence
\label{tab:Bockstein d_r}
} \\
\toprule
$x$ & $(s,f,w)$ & $d_r(x)$ & $r$ & Proof\\
\midrule \endfirsthead
\caption[]{All non-zero differentials on indecomposables in the $\rho$-Bockstein spectral sequence} \\
\toprule
$x$ & $(s,f,w)$ & $d_r(x)$ & $r$ & Proof\\
\midrule \endhead
\bottomrule \endfoot
    $\tau$ & $(0, 0, -1)$ & $\rho h_0$ & $1$ & \ref{Bockstein d_r on powers of tau}\\
    $\Dc$ & $(32, 7, 17)$ & $\rho h_0 a g$ & $1$ & \ref{proof of d_1-differentials}\\
    $\Du$ & $(35, 7, 19)$ & $\rho h_0 n g$ & $1$ & \ref{proof of d_1-differentials}\\
    $\tau^2$ & $(0, 0, -2)$ & $\rho^2 \tau h_1$ & $2$ & \ref{Bockstein d_r on powers of tau}\\
    $\Dh$ & $(25, 5, 13)$ & $\rho^2 \tau h_2^2 g$ & $2$ & \ref{differential on Dh}\\
    $\tau^3 d e$ & $(31, 8, 15)$ & $\rho^2 \tau^2 h_1 d e$ & $2$ & \ref{differential on t^3 d e}\\
    $\tau^3 h_2^2$ & $(6, 2, 1)$ & $\rho^3 \tau c$ & $3$ & \ref{proof of d_3 differentials}\\
    $P$ & $(8, 4, 4)$ & $\rho^3 h_1^2 c$ & $3$ & \ref{proof of d_3 differentials}\\
    $\tau^3 h_1 c$ & $(9, 4, 3)$ & $\rho^3 P h_2$ & $3$ & \ref{proof of d_3 differentials}\\
    $\tau^4 \Dh$ & $(25, 5, 9)$ & $\rho^3 \tau^2 a n$ & $3$ & \ref{proof of d_3 differentials}\\
    $\Dh h_1$ & $(26, 6, 14)$ & $\rho^3 c g$ & $3$ & \ref{differential on Dh h1}\\
    $\tau^3 n^2$ & $(30, 6, 13)$ & $\rho^3 (\tau \Dc + \tau^2 a g)$ & $3$ & \ref{differential on t^3 n^2}\\
    $\tau n^2$ & $(30, 6, 15)$ & $\rho^3 a g$ & $3$ & \ref{proof of d_3 differentials}\\
    $\tau^5 d e$ & $(31, 8, 13)$ & $\rho^3 \tau P \Dh$ & $3$ & \ref{proof of d_3 differentials}\\
    $\Dc + \tau a g$ & $(32, 7, 17)$ & $\rho^3 d g$ & $3$ & \ref{proof of d_3 differentials}\\
    $\tau^2 (\Du + \tau n g)$ & $(35, 7, 17)$ & $\rho^3 \tau^2 e g$ & $3$ & \ref{proof of d_3 differentials}\\
    $\Du + \tau n g$ & $(35, 7, 19)$ & $\rho^3 e g$ & $3$ & \ref{proof of d_3 differentials}\\
    $\tau^5 e g$ & $(37, 8, 17)$ & $\rho^3 \tau \Dh d$ & $3$ & \ref{proof of d_3 differentials}\\
    $\tau^2 \Dh d$ & $(39, 9, 19)$ & $\rho^3 a^2 e$ & $3$ & \ref{proof of d_3 differentials}\\
    $\tau^5 g^2$ & $(40, 8, 19)$ & $\rho^3 \tau \Dh e$ & $3$ & \ref{proof of d_3 differentials}\\
    $\tau^3 a^2 e$ & $(41, 10, 19)$ & $\rho^3 \tau^2 P n g$ & $3$ & \ref{proof of d_3 differentials}\\
    $\tau a^2 e$ & $(41, 10, 21)$ & $\rho^3 P n g$ & $3$ & \ref{proof of d_3 differentials}\\
    $\tau^2 (\Dc g + \tau a g^2)$ & $(52, 11, 27)$ & $\rho^3 \tau^2 d g^2$ & $3$ & \ref{proof of d_3 differentials}\\
    $\tau^4$ & $(0, 0, -4)$ & $\rho^4 \tau^2 h_2$ & $4$ & \ref{Bockstein d_r on powers of tau}\\
    $\tau P h_2^2$ & $(14, 6, 7)$ & $\rho^4 h_1^3 d$ & $4$ & \ref{proof of d_4 differentials}\\
    $\tau^2 n$ & $(15, 3, 6)$ & $\rho^4 h_2 n$ & $4$ & \ref{differential on t^2 n}\\
    $\tau h_0^2 e$ & $(17, 6, 9)$ & $\rho^4 h_1^3 e$ & $4$ & \ref{differential on t h0^2 e}\\
    $P a$ & $(20, 7, 10)$ & $\rho^4 h_1 c d$ & $4$ & \ref{proof of d_4 differentials}\\
    $P n$ & $(23, 7, 12)$ & $\rho^4 h_1 c e$ & $4$ & \ref{differential on P n}\\
    $\tau^4 \Dh h_1$ & $(26, 6, 10)$ & $\rho^4 \tau^2 a e$ & $4$ & \ref{differential on t^4 Dh h1}\\
    $\tau^4 \Dc h_1$ & $(33, 8, 14)$ & $\rho^4 a^3$ & $4$ & \ref{differential on t^4 Dc h1}\\
    $\tau^4 P \Dh h_1$ & $(34, 10, 14)$ & $\rho^4 \tau^2 P a e$ & $4$ & \ref{differential on t^4 P Dh h1}\\
    $\tau^4 P \Dc h_1$ & $(41, 12, 18)$ & $\rho^4 P a^3$ & $4$ & \ref{differential on t^4 P Dc h1}\\
    $\tau^4 d$ & $(14, 4, 4)$ & $\rho^5 \tau^2 h_1 e$ & $5$ & \ref{differential on t^4 d}\\
    $\tau^2 d$ & $(14, 4, 6)$ & $\rho^5 h_1 e$ & $5$ & \ref{differential on t^2 d}\\
    $\tau^6 e$ & $(17, 4, 4)$ & $\rho^5 \tau^4 h_1 g$ & $5$ & \ref{differential on t^6 e}\\
    $\tau^4 e$ & $(17, 4, 6)$ & $\rho^5 \tau^2 h_1 g$ & $5$ & \ref{differential on t^4 e}\\
    $\tau^2 P e$ & $(25, 8, 12)$ & $\rho^5 h_1 d^2$ & $5$ & \ref{differential on t^2 P e}\\
    $\tau^8 \Dh h_1^2$ & $(27, 7, 7)$ & $\rho^5 \tau^6 d e$ & $5$ & \ref{differential on t^8 Dh h1^2}\\
    $\tau^4 \Dh h_1^2$ & $(27, 7, 11)$ & $\rho^5 \tau^2 d e$ & $5$ & \ref{differential on t^4 Dh h1^2}\\
    $\tau^8 \Dc h_1^2$ & $(34, 9, 11)$ & $\rho^5 \tau^4 P n^2$ & $5$ & \ref{differential on t^8 Dc h1^2}\\
    $\tau^4 \Dc h_1^2$ & $(34, 9, 15)$ & $\rho^5 P n^2$ & $5$ & \ref{differential on t^4 Dc h1^2}\\
    $P a n$ & $(35, 10, 18)$ & $\rho^5 u d^2$ & $5$ & \ref{differential on P a n}\\
    $\tau^4 P \Dh h_1^2$ & $(35, 11, 15)$ & $\rho^5 \tau^2 P d e$ & $5$ & \ref{differential on t^4 P Dh h1^2}\\
    $\tau^4 P \Dc h_1^2$ & $(42, 13, 19)$ & $\rho^5 P^2 n^2$ & $5$ & \ref{differential on t^4 P Dc h1^2}\\
    $\tau^4 h_1$ & $(1, 1, -3)$ & $\rho^6 \tau h_2^2$ & $6$ & \ref{proof of d_6 differentials}\\
    $\tau^2 P h_2$ & $(11, 5, 4)$ & $\rho^6 h_1^2 d$ & $6$ & \ref{proof of d_6 differentials}\\
    $\tau^4 a$ & $(12, 3, 2)$ & $\rho^6 \tau^2 e$ & $6$ & \ref{differential on t^4 a}\\
    $\tau^2 a$ & $(12, 3, 4)$ & $\rho^6 e$ & $6$ & \ref{differential on t^2 a}\\
    $\tau^6 n$ & $(15, 3, 2)$ & $\rho^6 \tau^4 g$ & $6$ & \ref{differential on t^6 n}\\
    $\tau^3 h_0^2 e$ & $(17, 6, 7)$ & $\rho^6 c d$ & $6$ & \ref{proof of d_6 differentials}\\
    $\tau^8 P n$ & $(23, 7, 4)$ & $\rho^6 \tau^6 P g$ & $6$ & \ref{differential on t^8 P n}\\
    $\tau^2 P n$ & $(23, 7, 10)$ & $\rho^6 d^2$ & $6$ & \ref{differential on t^2 P n}\\
    $\tau^6 P h_1 d$ & $(23, 9, 7)$ & $\rho^6 \tau^3 P h_0^2 g$ & $6$ & \ref{differential on t^6 P h1 d}\\
    $\tau^2 n^2$ & $(30, 6, 14)$ & $\rho^6 n g$ & $6$ & \ref{differential on t^2 n^2}\\
    $\tau^{12} d e$ & $(31, 8, 6)$ & $\rho^6 \tau^6 a^3$ & $6$ & \ref{differential on t^12 d e}\\
    $\tau^4 P^2 n$ & $(31, 11, 12)$ & $\rho^6 \tau^2 P^2 g$ & $6$ & \ref{differential on t^4 P^2 n}\\
    $\tau^8 P d e$ & $(39, 12, 14)$ & $\rho^6 \tau^2 P a^3$ & $6$ & \ref{differential on t^8 P d e}\\
    $\tau^6 g^2$ & $(40, 8, 18)$ & $\rho^6 \tau \Dh g$ & $6$ & \ref{differential on t^6 g^2}\\
    $\tau^4 h_1^2$ & $(2, 2, -2)$ & $\rho^7 c$ & $7$ & \ref{differential on t^4 h1^2}\\
    $\tau^5 h_2^2$ & $(6, 2, -1)$ & $\rho^7 a$ & $7$ & \ref{differential on t^5 h2^2}\\
    $\tau^4 c$ & $(8, 3, 1)$ & $\rho^7 d$ & $7$ & \ref{differential on t^4 c}\\
    $\tau^{10} P h_2$ & $(11, 5, -4)$ & $\rho^7 \tau^7 h_0^2 e$ & $7$ & \ref{differential on t^10 P h2}\\
    $\tau^{11} h_0^2 e$ & $(17, 6, -1)$ & $\rho^7 \tau^6 P n$ & $7$ & \ref{differential on t^11 h0^2 e}\\
    $\tau^9 h_0^2 e$ & $(17, 6, 1)$ & $\rho^7 \tau^4 P n$ & $7$ & \ref{differential on t^9 h0^2 e}\\
    $\tau^6 P^2 h_2$ & $(19, 9, 4)$ & $\rho^7 \tau^3 P h_0^2 e$ & $7$ & \ref{proof of d_7 differentials}\\
    $\tau^7 P h_0^2 e$ & $(25, 10, 7)$ & $\rho^7 \tau^2 P^2 n$ & $7$ & \ref{differential on t^7 P h0^2 e}\\
    $\tau^5 P h_0^2 e$ & $(25, 10, 9)$ & $\rho^7 P^2 n$ & $7$ & \ref{differential on t^5 P h0^2 e}\\
    $\tau^6 h_2 n$ & $(18, 4, 4)$ & $\rho^8 \tau \Dh$ & $8$ & \ref{differential on t^6 h2 n}\\
    $\tau^6 h_2^2$ & $(6, 2, -2)$ & $\rho^{10} n$ & $10$ & \ref{differential on t^6 h2^2}\\
\end{longtable}

By definition of the $\rho$-Bockstein spectral sequence, the differentials $d_r$ increase the $\rho$-filtration by $r$. Because of Proposition \ref{only differentials between rho-free elements} all indecomposables in the $\rho$-Bockstein spectral sequence have $\rho$-filtration zero. As a consequence, the values of their $d_r$-differentials will always contain $\rho^r$. This is also apparent when comparing the third and fourth column of table \ref{tab:Bockstein d_r}. Based on this observation, we will make the following notational convention.
    
\begin{Not}\label{leaving out powers of rho from Bockstein differentials}
    Assume $x$ is an element in the $\rho$-Bockstein spectral sequence of $\rho$-filtration zero and $d_r(x) = \rho^r y$. In this section, we leave out the power of $\rho$ from our notation, i.e. we write $d_r(x) = y$ instead. The power of $\rho$ is implicitly given by the index of the differential.
\end{Not}

\begin{Not}\label{lemmas have (s, f, w) degree stated in them}
    Many statements in the following sections will be of the form $d_r(x) = y$. To help the reader navigate the computation we have included the degree \deg{s}{f}{w} of the source $x$ of the differential whenever we make such a statement. See for example Lemma \ref{differential on t^4 h1^2}.
\end{Not}

\subsection{Chart guide}\label{section: Chart guide for section on rho-Bockstein differentials}

The relevant charts from \cite{Charts} for this section can be found in the folders \texttt{Rho-Bockstein E\_r-pages rho-free quotient} and \texttt{Rho-Bockstein internal coweight charts}.

The charts in the first-mentioned folder depict the $\rho$-free quotient of the $\rho$-Bockstein spectral sequence $E_r$-page. They aid the reader throughout the different sections by providing an overview of the elements still under consideration for differentials.

The charts in the second folder will be crucial to many arguments in this section. How to use these charts in conjunction with the text is spelled out in detail in section \ref{section: The internal coweight method} for the case of internal coweight $5$. The internal coweight is defined in Definition \ref{definition of internal coweight}. Whenever a proof in this section mentions an internal coweight, we encourage the reader to open an associated internal coweight chart. Since each internal coweight chart is sparse, we suggest using the pdf version. To give an example of where to use a chart, in Lemma \ref{differential on t^4 h1^2} we want to look at internal coweight $6$. Since Lemma \ref{differential on t^4 h1^2} is located in the $d_2$-differentials section, we can use the chart titled \texttt{Rho-Bockstein E\_2 rho-free quotient s+f-w=6}. This chart will aid the reader in following the proof of Lemma \ref{differential on t^4 h1^2}. Note that sometimes we will re-use internal coweight charts from lower pages. For example in Lemma \ref{differential on t^4 d} we consider internal coweight $14$, and we can use the chart \texttt{Rho-Bockstein E\_4 rho-free quotient s+f-w=14} for that purpose. Since Lemma \ref{differential on t^4 d} is in the $d_5$-differentials section, we can ignore elements in the given chart that are involved in $d_4$-differentials. For every internal coweight mentioned in the following sections there is an associated chart.

\subsection{\texorpdfstring{$d_1$}{d\_1}-differentials}

We continue by giving proofs of the differentials one page at a time. We will freely make use of \cite[Theorem 4.13]{Isa09}, which provides generators and relations for $\Ext_{\mathcal{A}(2)^{\mathbb{C}}}(\mathbb{M}_2^{\mathbb{C}},\mathbb{M}_2^{\mathbb{C}})$, i.e. for the cohomology of $\mathbb{C}$-motivic $\mathcal{A}(2)^{\mathbb{C}}$. We also recalled these generators earlier in table \ref{tab:Generators of the cohomology of A(2)^C}.

We start with table \ref{tab:Bockstein d_1} collecting all indecomposables on the $E_1$-page, together with the values of their $d_1$-differentials. The $d_1$-differentials on all other elements are then consequences of the Leibniz rule. An empty entry in the $d_1$-column of table \ref{tab:Bockstein d_1} implies that $d_1$ applied to that indecomposable is zero. We leave out the indecomposable $\rho$ as we know it is a permanent cycle since it is a cycle in the cobar complex. Similarly, following Notation \ref{leaving out powers of rho from Bockstein differentials} we will also leave out the factor $\rho$ everywhere. As an example, consider the element $\tau$ in table \ref{tab:Bockstein d_1}. The $d_1$-column entry is $h_0$, and that tells us $d_1(\tau) = \rho h_0$.

\begin{longtable}{lllc}
\caption{Indecomposables on $E_1$ and their $d_1$-differentials
\label{tab:Bockstein d_1}
} \\
\toprule
indecomposable & $(s,f,w)$ & $d_1$ & Proof \\
\midrule \endfirsthead
\caption[]{Indecomposables on $E_1$ and their $d_1$-differentials} \\
\toprule
indecomposable & $(s,f,w)$ & $d_1$ & Proof \\
\midrule \endhead
\bottomrule \endfoot
    $\tau$ & $(0, 0, -1)$ & $h_0$ & \ref{Bockstein d_r on powers of tau}\\
    $h_0$ & $(0, 1, 0)$ & & \\
    $h_1$ & $(1, 1, 1)$ & & \\
    $h_2$ & $(3, 1, 2)$ & & \\
    $c$ & $(8, 3, 5)$ & & \\
    $P$ & $(8, 4, 4)$ & & \\
    $u$ & $(11, 3, 7)$ & & \\
    $a$ & $(12, 3, 6)$ & & \\
    $d$ & $(14, 4, 8)$ & & \\
    $n$ & $(15, 3, 8)$ & & \\
    $e$ & $(17, 4, 10)$ & & \\
    $g$ & $(20, 4, 12)$ & & \\
    $\Dh$ & $(25, 5, 13)$ & & \\
    $\Dc$ & $(32, 7, 17)$ & $h_0 a g$ & \ref{proof of d_1-differentials}\\
    $\Du$ & $(35, 7, 19)$ & $h_0 n g$ & \ref{proof of d_1-differentials}\\
    $\D^2$ & $(48, 8, 24)$ & & \\
\end{longtable}

\begin{Prop}\label{proof of d_1-differentials}
    Table \ref{tab:Bockstein d_1} describes the $d_1$-differentials in the $\rho$-Bockstein spectral sequence on all indecomposables on $E_1$.
\end{Prop}

\begin{proof}
    For every indecomposable whose $d_1$-differential is claimed to be zero this follows for degree reasons, i.e. the target degree of $d_1$ does not contain any elements of the form $\rho \cdot x$ for $x \in \Ext_{\mathcal{A}(2)^{\mathbb{C}}}(\mathbb{M}_2^{\mathbb{C}},\mathbb{M}_2^{\mathbb{C}})$. The proof of the differential on $\tau$ was given in Lemma \ref{Bockstein d_r on powers of tau}.

    The differentials on $\Dc$ and $\Du$ are consequences of the one on $\tau$. For $\Dc$ use the relation $h_0 \cdot \Dc = \tau \cdot h_0 \cdot a \cdot g$. Then by the Leibniz rule
    \[h_0 \cdot d_1(\Dc) = h_0^2 \cdot a \cdot g \neq 0.\]
    The only way this can be true is if $d_1(\Dc) = h_0 \cdot a \cdot g$. The proof for $\Du$ is similar, using the relation $h_0 \cdot \Du = \tau \cdot h_0 \cdot n \cdot g$.
\end{proof}

\subsection{The internal coweight method}\label{section: The internal coweight method}

Before moving on to the next page, we will discuss another extremely potent method for deducing differentials in the $\rho$-Bockstein spectral sequence. This method is essentially the same as \cite[Strategy 5.3]{BI}, we're just giving it a name and make separate charts for it. Before going into details, we will collect all the different steps and ideas involved in this method:
\begin{enumerate}[label=\rm{(\arabic*)}]
    \item By Proposition \ref{only differentials between rho-free elements} it suffices to consider the $\rho$-free quotient $E_r/(\text{$\rho$-power-torsion})$ of each $E_r$-page when computing $\rho$-Bockstein differentials.
    \item The hyperplanes given by fixing a constant value for $s + f - w$ on each $E_r$-page form additive sub-spectral sequences. This is because the degree of $d_r$ is $(-1, 1, 0)$, so it preserves $s + f - w$.
    \item We already have complete knowledge of the amount of $\mathbb{F}_2[\rho]$-modules on the $\rho$-free quotient of $E_\infty$. This follows from Corollary \ref{representation of rho-localized target} and the fact that the kernel of $\rho$-localization is precisely the $\rho$-power-torsion elements. So the $\rho$-free quotient injects into the $\rho$-localization and in particular the $\mathbb{F}_2[\rho]$-modules on the $\rho$-free quotient of $E_\infty$ correspond to the $\mathbb{F}_2[\rho^{\pm 1}]$-modules in the $\rho$-localization.
\end{enumerate}
We will make some additional remarks: For the first point, it is worth noting that the $\rho$-free quotient of the $E_r$-page is obtained from the $E_1$-page by removing elements that either supported or were the target of a differential. For sources of differentials that is just the usual way a spectral sequence works, for targets of differentials that is because they are now $\rho$-power-torsion. For the second point, we choose $s + f - w$ specifically so that $\rho$-multiplication also preserves this linear combination. This makes it easier to create charts, as we can have a dot represent an entire $\mathbb{F}_2[\rho]$-module. Note that multiplication by any other indecomposable does not preserve this linear combination. For the third point, we even know the $f$-values of the $\mathbb{F}_2[\rho]$-modules on the $\rho$-free quotient because $\rho$-multiplication preserves $f$.

As we will have to refer to the linear combination $s + f - w$ a lot, we give it a name.

\begin{Def}\label{definition of internal coweight}
    For an element $x$ in degree $(s, f, w)$ we call the integer $s + f - w$ its \textbf{internal coweight}.
\end{Def}

To explain the name, the quantity $s - w$ is typically referred to as coweight, and the quantity $s + f$ describes the internal degree if one were to compute $\Ext$ via a projective resolution.

Calculating the internal coweight of each element in table \ref{tab:Bockstein d_1} shows that every indecomposable of $E_1$ (and hence of each $E_r$-page) has a non-negative internal coweight. More precisely, only $\rho$ has internal coweight $0$ and all other indecomposables have a strictly positive internal coweight.

We will illustrate the internal coweight method with an example in internal coweight $5$.
\begin{figure}[htp]
    \centering
    \includegraphics[width={.75\textwidth}]{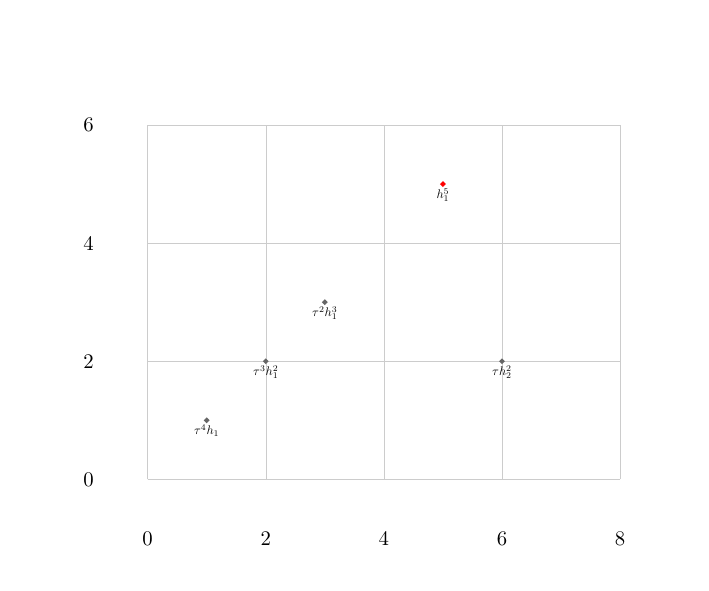}
    \caption{$\rho$-Bockstein $E_2$-page $\rho$-free quotient, internal coweight $5$}
    \label{fig:Rho-Bockstein E_2 rho-free part s+f-w=5 without differentials}
\end{figure}
In figure \ref{fig:Rho-Bockstein E_2 rho-free part s+f-w=5 without differentials} the horizontal axis represents the stem $s$ and the vertical axis represents the Adams filtration $f$. Each dot represents a free $\mathbb{F}_2[\rho]$-module with internal coweight $5$. The gray dots correspond to $\tau^2$-free classes whereas the red dot corresponds to a $\tau^2$-torsion class, but that is not relevant for the following arguments. By Corollary \ref{representation of rho-localized target} we have a presentation of the $\rho$-localized abutment of the spectral sequence. It is easy to check that the only element there with internal coweight $5$ is $h_1^5$ (and its $\rho$-multiples). We can also find a dot with that label in the provided figure. We see that $h_1^5$ cannot be involved in any differentials, as incoming differentials would have to come from filtration $4$ and outgoing differentials would end up in filtration $6$ and both of those are empty. That means $h_1^5$ survives the $\rho$-Bockstein spectral sequence and thus represents a $\rho$-torsion free class at $E_{\infty}$, namely the one that is called $h_1^5$ in Corollary \ref{representation of rho-localized target}. Next consider the other four gray dots. If one of them is not involved in any differentials in our spectral sequence, then it will represent a $\rho$-torsion free class on $E_\infty$. That $\rho$-torsion free class would map to a non-zero element in the $\rho$-localized cohomology of $\mathcal{A}(2)$. But as we have seen above, there is only one element with that property in internal coweight $5$, and that element is $h_1^5$. So all four gray dots must be involved in some differentials. There is now exactly two things that can happen to each dot: Either it supports a differential so that it disappears before $E_{\infty}$, or it gets hit by a differential and becomes $\rho$-power-torsion, hence zero in the $\rho$-localization. It is a quick and pleasant exercise to convince oneself that there is only one possible set of differentials. The solution is described in the next paragraph. The main point is that all $d_r$-differentials go one to the left and one up in the chart, but they also have to hit a $\rho^r$-multiple and $\rho$-multiplication moves one to the left.

For the readers' convenience we include another chart of internal coweight $5$ in figure \ref{fig:Rho-Bockstein E_2 rho-free part s+f-w=5 with differentials}.
\begin{figure}[htp]
    \centering
    \includegraphics[width=.75\textwidth]{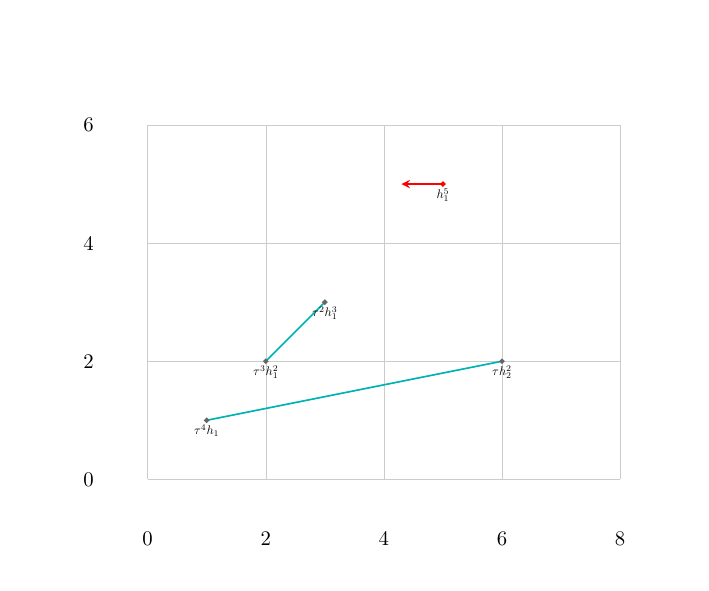}
    \caption{$\rho$-Bockstein $E_2$-page $\rho$-free quotient, internal coweight $5$, with differentials}
    \label{fig:Rho-Bockstein E_2 rho-free part s+f-w=5 with differentials}
\end{figure}
This chart shows each class in the $\rho$-free quotient of the $E_2$-page in internal coweight $5$, and describes their behavior in the $\rho$-Bockstein spectral sequence. The arrow on $h_1^5$ means that it survives the $\rho$-Bockstein spectral sequence, it is not hit by a differential and it does not support a differential. A line connecting two dots implies that there is a differential from the class that is lower on the $y$-axis to the other one. The length of the line determines the index of the differential and hence which power of $\rho$ is hit, but we will continue leaving out said power of $\rho$ from the notation, cf. Notation \ref{leaving out powers of rho from Bockstein differentials}. For example the line from $\tau^3 h_1^2$ to $\tau^2 h_1^3$ implies $d_2(\tau^3 h_1^2) = \tau^2 h_1^3$, where we have left out the factor $\rho^2$ on the right-hand side. To see that this differential has to occur, recall that $\tau^2 h_1^3$ has to be involved in some differential as there cannot be any $\rho$-torsion free classes on $E_\infty$ in internal coweight $5$ and Adams filtration $3$ by Corollary \ref{representation of rho-localized target}. Certainly $\tau^2 h_1^3$ cannot support a differential as the target would have to be in Adams filtration $4$. So some $\rho$-multiple of $\tau^2 h_1^3$ has to get hit by a differential. The source of that differential must be in Adams filtration $2$. A priori there are two elements there, $\tau^3 h_1^2$ and $\tau h_2^2$. But note that any $d_r$ on $\tau h_2^2$ would land at the coordinates $(5, 3)$ in this chart. Since the $\rho$-multiples of $\tau^2 h_1^3$ live to its own left in the chart, we see that $\tau h_2^2$ cannot hit them with a differential. Hence $\tau^3 h_1^2$ must support a differential hitting a $\rho$-multiple of $\tau^2 h_1^3$. Since $d_r(\tau^3 h_1^2)$ is at the coordinates $(1, 3)$ and $\tau^2 h_1^3$ is at the coordinates $(3, 3)$ and $3 - 1 = 2$, we must have $d_2(\tau^3 h_1^2) = \tau^2 h_1^3$. Now the other differential $d_6(\tau^4 h_1) = \tau h_2^2$ is forced.

In a completed internal coweight chart as in figure \ref{fig:Rho-Bockstein E_2 rho-free part s+f-w=5 with differentials}, every dot falls into exactly one of the two following categories:
\begin{itemize}
    \item There is a horizontal arrow pointing to its left, indicating that the class corresponding to this dot is not involved in any differentials in the $\rho$-Bockstein spectral sequence. Hence it is $\rho$-torsion free on $E_\infty$. The arrow points to where its $\rho$-multiples are.

    \item There is a line of height $1$ and positive slope connecting the given dot with exactly one other dot. This indicates that there is a differential from the class that is lower on the vertical axis to the other one. The index of the differential is implicitly given by the length of the line (or equivalently by the general position of the two dots connected by the line). More precisely, a line of slope $1/(r-1)$ from $x$ to $y$ implies $d_r(x) = \rho^r y$. Note that we only use the internal coweight method starting on $E_2$ and higher pages, so there will never be vertical lines which would indicate $d_1$'s.
\end{itemize}
There are charts for all relevant internal coweights at \cite{Charts}. For a guide on how to read them see section \ref{section: charts}, and specifically section \ref{charts: rho-Bockstein internal coweight charts}.

By playing the same game as we just did for other values of the internal coweight, we could already compute a lot of $d_r$-differentials for $r > 2$ on $E_2$. We will not do this. Instead, we will keep moving from page to page and only apply the internal coweight method when it is needed to prove existence or non-existence of some differential.

\begin{Conv}\label{Convention on targets of differentials being non-trivial}
    By using the internal coweight method we will be able to compute $d_r$-differentials on pages prior to $E_r$. For example in figure \ref{fig:Rho-Bockstein E_2 rho-free part s+f-w=5 with differentials} we were able to compute $d_6(\tau^4 h_1) = \tau h_2^2$ purely from information available on $E_2$. We will do this very often going forward. So we should be precise with interpreting our notation: When writing $d_6(\tau^4 h_1) = \tau h_2^2$ we are not only claiming that this differential exists, but also that it is a non-trivial differential on the $E_6$-page, i.e. the target $\tau h_2^2$ is not $\rho^6$-torsion on $E_6$. The running convention for all further sections computing $\rho$-Bockstein differentials is that whenever we write $d_r(x) = y$ and $y$ is not the symbol $0$, then we are implicitly claiming that the target $y$ is not $\rho^r$-torsion on $E_r$. In particular that means the claimed $d_r$-differential is non-trivial.
\end{Conv}

\subsection{\texorpdfstring{$d_2$}{d\_2}-differentials}

From this point onward we continue with describing the differentials page by page. Remember that we listed all non-zero differentials on indecomposables in table \ref{tab:Bockstein d_r}. Our goal is to prove that this list is both complete and correct.

Moving forward, we will not give tables of all indecomposables on each page as we did for $E_1$. Instead, inside of each Proposition we will only list those indecomposables that could possibly support a differential for degree reasons. Note that all other differentials on each page can then be deduced via the Leibniz rule. We argue separately for each indecomposable whether or not the given differential occurs.

The tables inside of each Proposition have columns describing
\begin{enumerate}[label=\rm{(\arabic*)}]
    \item the name of the indecomposable possibly supporting a differential,
    \item the degree of said indecomposable,
    \item the target of the differential (leaving out the factors of $\rho$ as in table \ref{tab:Bockstein d_1}),
    \item whether or not the given differential occurs,
    \item where the proof can be found.
\end{enumerate}

\begin{Prop}\label{proof of d_2 differentials}
    Table \ref{tab:Bockstein d_r} describes the non-zero $d_2$-differentials in the $\rho$-Bockstein spectral sequence on all indecomposables on $E_2$.
\end{Prop}

\begin{proof}
    For degree reasons, the only indecomposables that can support a $d_2$-differential are given by the following table.

    \begin{longtable}{llllc}
    \caption{Possible non-zero $d_2$-differentials on indecomposable elements
    \label{tab:Possible Bockstein d_2}
    } \\
    \toprule
    $x$ & $(s,f,w)$ & $d_2(x)$ & Occurs & Proof\\
    \midrule \endfirsthead
    \caption[]{Possible non-zero $d_2$-differentials on indecomposable elements} \\
    \toprule
    $x$ & $(s,f,w)$ & $d_2(x)$ & Occurs & Proof\\
    \midrule \endhead
    \bottomrule \endfoot
        $\tau^2$ & $(0, 0, -2)$ & $\tau h_1$ & Yes & \ref{Bockstein d_r on powers of tau}\\
        $\tau h_1$ & $(1, 1, 0)$ & $h_1^2$ & No & \ref{d_2 differentials on t h1, t c, t Dh, t Dc}\\
        $\tau c$ & $(8, 3, 4)$ & $h_1 c$ & No & \ref{d_2 differentials on t h1, t c, t Dh, t Dc}\\
        $\tau \Dh$ & $(25, 5, 12)$ & $\Dh h_1$ & No & \ref{d_2 differentials on t h1, t c, t Dh, t Dc}\\
        $\Dh$ & $(25, 5, 13)$ & $\tau h_2^2 g$ & Yes & \ref{differential on Dh}\\
        $\tau^3 d e$ & $(31, 8, 15)$ & $\tau^2 h_1 d e$ & Yes& \ref{differential on t^3 d e}\\
        $\tau \Dc$ & $(32, 7, 16)$ & $\Dc h_1$ & No & \ref{d_2 differentials on t h1, t c, t Dh, t Dc}\\
    \end{longtable}

    The differential on $\tau^2$ was already proven in Lemma \ref{Bockstein d_r on powers of tau}. For all other differentials see the following series of Lemmas.
\end{proof}

\begin{Lemma}\label{d_2 differentials on t h1, t c, t Dh, t Dc}
    The $d_2$-differentials on the elements $\tau h_1$, $\tau c$, $\tau \Dh$, and $\tau \Dc$ are all zero.
\end{Lemma}

\begin{proof}
    Note that $h_1$ is a $d_2$-cycle for degree reasons, so the $d_2$-differential is $h_1$-linear by the Leibniz rule. Thus, the $d_2$-differentials on $\tau h_1$, $\tau c$, $\tau \Dh$, and $\tau \Dc$ must be zero as these elements are all $h_1$-power-torsion but their possible targets are $h_1$-torsion free.
\end{proof}

It will be convenient now and at later stages to refer to the $\rho$-Bockstein differentials hitting $c$ and $a$, so we state them here. As it becomes relevant now, we remind the reader of Notation \ref{lemmas have (s, f, w) degree stated in them}. We also refer the reader back to the chart guide in section \ref{section: Chart guide for section on rho-Bockstein differentials}.

\begin{Lemma}\label{differential on t^4 h1^2}\label{differential hitting c}\deg{2}{2}{-2}
    $d_7(\tau^4 h_1^2) = c$.
\end{Lemma}

\begin{proof}
    In internal coweight $6$ the element $\tau^4 h_1^2$ must be involved in some differential by Corollary \ref{representation of rho-localized target} combined with degree reasons. But Lemma \ref{Bockstein d_r on powers of tau} implies $d_2(\tau^6) = d_2(\tau^2) \cdot \tau^4 = \tau^5 h_1$. Now the claimed $d_7$ is the only possible differential that $\tau^4 h_1^2$ can be involved in.
\end{proof}

\begin{Lemma}\label{differential on t^5 h2^2}\label{differential hitting a}\deg{6}{2}{-1}
    $d_7(\tau^5 h_2^2) = a$.
\end{Lemma}

\begin{proof}
   In internal coweight $9$ the element $\tau^5 h_2^2$ must be involved in some differential. It cannot get hit by any differentials since the only possible source would be $\tau^8 h_1 = \tau^8 \cdot h_1$ which is a product of permanent cycles by Corollary \ref{t^8, h1, h2, u, g are permanent cycles}. So $\tau^5 h_2^2$ must support a differential hitting either $\tau^3 c$ or $a$. But by Lemma \ref{d_2 differentials on t h1, t c, t Dh, t Dc} we know $d_2(\tau c) = 0$, and so $d_2(\tau^3 c) = d_2(\tau^2 \cdot \tau c) = d_2(\tau^2) \cdot \tau c = \tau^2 h_1 c$. In particular $\tau^3 c$ cannot get hit by any differentials. That means we must have $d_7(\tau^5 h_2^2) = a$.
\end{proof}

\begin{Lemma}\label{differential on Dh}\label{differential hitting t h2^2 g}\label{differential on t^3 d e}\label{differential hitting t^2 h1 d e}\label{differential on Dh h1}\label{differential hitting c g} \leavevmode
    \begin{enumerate}[label=\rm{(\arabic*)}]
        \item \deg{25}{5}{13} $d_2(\Dh) = \tau h_2^2 g$,
        \item \deg{31}{8}{15} $d_2(\tau^3 d e) = \tau^2 h_1 d e$,
        \item \deg{26}{6}{14} $d_3(\Dh h_1) = c g$.
    \end{enumerate}
\end{Lemma}

\begin{proof}
    Using the relation $\Dh \cdot a \cdot d = \tau^3 d e \cdot g$ and computing $d_2$ on both sides one quickly sees that the differentials on $\Dh$ and $\tau^3 d e$ are interdependent. So it suffices to show that one of the $d_2$-differentials occurs.
    
    We choose to compute the differential on $\Dh$. Consider internal coweight $18$ where $\Dh h_1$ lives. We will prove $d_3(\Dh h_1) = c g$. Then, since $c g$ is not $h_1$-divisible, this differential can only occur if $\Dh h_1$ is not $h_1$-divisible on $E_3$. That can only happen if $\Dh$ supports the given $d_2$.
    
    Now to see that $d_3(\Dh h_1) = c g$, recall from Lemma \ref{differential hitting c} that $d_7(\tau^4 h_1^2) = c$. By Corollary \ref{t^8, h1, h2, u, g are permanent cycles} $g$ is a permanent cycle, so it follows that $c g = c \cdot g$ has to get hit by a differential of length at most $7$. The possible sources are $a^2$ and $\Dh h_1$. But $a^2$ is a permanent cycle because $a$ is a permanent cycle by Lemma \ref{differential hitting a}.
\end{proof}

\subsection{\texorpdfstring{$d_3$}{d\_3}-differentials}

\begin{Prop}\label{proof of d_3 differentials}
    Table \ref{tab:Bockstein d_r} describes the non-zero $d_3$-differentials in the $\rho$-Bockstein spectral sequence on all indecomposables on $E_3$.
\end{Prop}

\begin{proof}
    For degree reasons, the only indecomposables that can support a $d_3$-differential are given by table \ref{tab:Possible Bockstein d_3}. If the proof column is empty, then the proof of that differential is contained in this Proposition.
    
    Note that $\tau^3 n^2$ appears twice in the first column of this table. That is because the target degree of its $d_3$ is $2$-dimensional as an $\mathbb{F}_2$-vector space generated by $\tau^2 a g$ and $\tau \Dc$. We will show that both of these generators are summands in $d_3(\tau^3 n^2)$.

    \begin{longtable}{llllc}
    \caption{Possible non-zero $d_3$-differentials on indecomposable elements
    \label{tab:Possible Bockstein d_3}
    } \\
    \toprule
    $x$ & $(s,f,w)$ & $d_3(x)$ & Occurs & Proof\\
    \midrule \endfirsthead
    \caption[]{Possible non-zero $d_3$-differentials on indecomposable elements} \\
    \toprule
    $x$ & $(s,f,w)$ & $d_3(x)$ & Occurs & Proof\\
    \midrule \endhead
    \bottomrule \endfoot
        $\tau^3 h_2^2$ & $(6, 2, 1)$ & $\tau c$ & Yes & \\
        $P$ & $(8, 4, 4)$ & $h_1^2 c$ & Yes & \\
        $\tau^3 h_1 c$ & $(9, 4, 3)$ & $P h_2$ & Yes & \\
        $\tau^4 \Dh$ & $(25, 5, 9)$ & $\tau^2 a n$ & Yes & \\
        $\Dh h_1$ & $(26, 6, 14)$ & $c g$ & Yes & \ref{differential on Dh h1}\\
        $\tau^3 n^2$ & $(30, 6, 13)$ & $\tau^2 a g$ & Yes & \ref{differential on t^3 n^2}\\
        $\tau^3 n^2$ & $(30, 6, 13)$ & $\tau \Dc$ & Yes & \ref{differential on t^3 n^2}\\
        $\tau n^2$ & $(30, 6, 15)$ & $a g$ & Yes & \\
        $\tau^5 d e$ & $(31, 8, 13)$ & $\tau P \Dh$ & Yes & \\
        $\Dc + \tau a g$ & $(32, 7, 17)$ & $d g$ & Yes & \\
        $\tau^2 (\Du + \tau n g)$ & $(35, 7, 17)$ & $\tau^2 e g$ & Yes & \\
        $\Du + \tau n g$ & $(35, 7, 19)$ & $e g$ & Yes & \\
        $\tau^5 e g$ & $(37, 8, 17)$ & $\tau \Dh d$ & Yes & \\
        $\tau^2 \Dh d$ & $(39, 9, 19)$ & $a^2 e$ & Yes & \\
        $\tau^5 g^2$ & $(40, 8, 19)$ & $\tau \Dh e$ & Yes & \\
        $\tau^3 a^2 e$ & $(41, 10, 19)$ & $\tau^2 P n g$ & Yes & \\
        $\tau a^2 e$ & $(41, 10, 21)$ & $P n g$ & Yes & \\
        $\tau^2 (\Dc g + \tau a g^2)$ & $(52, 11, 27)$ & $\tau^2 d g^2$ & Yes & \\
    \end{longtable}

    For the differential on $\tau^3 h_2^2$ consider internal coweight $7$. By Corollary \ref{representation of rho-localized target} and degree reasons $\tau^3 h_2^2$ must be involved in a differential. It cannot hit $u$ with a differential because then there is no other element left to represent what we called $u$ in Corollary \ref{representation of rho-localized target}. That only leaves $d_3(\tau^3 h_2^2) = \tau c$.
    
    The differentials on $P$ and $\tau^3 h_1 c$ follow similarly from looking at internal coweight $8$ and $10$, respectively.

    Almost every other differential can be deduced from multiplicative relations now. For example for $\Dc + \tau a g$ we can use the relation $(\Dc + \tau a g) \cdot c = \Dh h_1 \cdot d$. Then recall Lemma \ref{differential on Dh h1} stating $d_3(\Dh h_1) = c g$, and compute $d_3$ on both sides
    \[d_3(\Dc + \tau a g) \cdot c = d_3(\Dh h_1) \cdot d = c g \cdot d.\]
    But $c g \cdot d$ is $\rho$-torsion free on $E_3$, so $d_3(\Dc + \tau a g)$ must be non-zero and the given value is the only possible one.

    We provide a table indicating which relation and previous differential we could use to make an argument similar to the above.

    \begin{longtable}{llll}
    \caption{Relations to use to prove $d_3$-differentials
    \label{tab:Relations for d_3}
    } \\
    \toprule
    To prove $d_3$ on & Use $d_3$ on & Together with the relation\\
    \midrule \endfirsthead
    \caption[]{Relations to use to prove $d_3$-differentials} \\
    \toprule
    To prove $d_3$ on & Use $d_3$ on & Together with the relation\\
    \midrule \endhead
    \bottomrule \endfoot
        $\tau^4 \Dh$ & $\Dc + \tau a g$ & $(\Dc + \tau a g) \cdot \tau^4 \Dh = \tau^4 \cdot h_1 \cdot c \cdot \D^2 + \tau^8 \cdot e \cdot g^2$\\
        $\tau n^2$ & $\tau^4 \Dh$ & $\tau n^2 \cdot \tau^4 \Dh = \tau^8 \cdot n \cdot g^2$\\
        $\tau^2 (\Du + \tau n g)$ & $\tau^4 \Dh$ & $\tau^2 (\Du + \tau n g) \cdot n^2 = \tau^4 \Dh \cdot g^2$\\
        $\Du + \tau n g$ & $\tau^4 \Dh$ & $(\Du + \tau n g) \cdot \tau^2 n^2 = \tau^4 \Dh \cdot g^2$\\
        $\tau^5 e g$ & $\tau^4 \Dh$ & $\tau^2 a \cdot \tau^4 \Dh = \tau^4 \cdot \tau^5 e g$\\
        $\tau^2 \Dh d$ & $\tau^4 \Dh$ & $\tau^2 \Dh d \cdot \tau^4 = \tau^4 \Dh \cdot \tau^2 d$\\
        $\tau^5 g^2$ & $\tau^4 \Dh$ & $\tau^4 \Dh \cdot \tau^2 n = \tau^4 \cdot \tau^5 g^2$\\
        $\tau^3 a^2 e$ & $\tau^4 \Dh$ & $\tau^3 a^2 e \cdot n = \tau^4 \Dh \cdot d e$\\
        $\tau a^2 e$ & $\tau^4 \Dh$ & $\tau a^2 e \cdot \tau^2 n = \tau^4 \Dh \cdot d e$\\
        $\tau^2 (\Dc g + \tau a g^2)$ & $\tau^4 \Dh$ & $\tau^2 (\Dc g + \tau a g^2) \cdot \tau^4 = a n \cdot \tau^4 \Dh$\\
        $\tau^5 d e$ & $\tau^5 e g$ & $\tau^5 d e \cdot g = \tau^5 e g \cdot d$\\
    \end{longtable}
    
    This table covers all differentials except for the one on $\tau^3 n^2$. The proof of that differential is contained in Lemma \ref{differential on t^3 n^2}.
\end{proof}

The next few Lemmas build up to the proof of $d_3(\tau^3 n^2) = \tau^2 a g + \tau \Dc$. We should note that in the following Lemma \ref{differential on t^4 a n} we are making a one-time exception to Convention \ref{Convention on targets of differentials being non-trivial}. As of now, $\tau^2 a g$ could be hit by a $d_3$ which would make the claimed $d_6$ trivial. Note that this $d_3$ is the only shorter differential that $\tau^2 a g$ could be hit by. However, in the proof of Lemma \ref{differential on t^3 n^2} we will be able to lead the assumption $d_3(\tau^3 n^2) = \tau^2 a g$ to a contradiction by using Lemma \ref{differential on t^4 a n}.

\begin{Lemma}\label{differential on t^4 a n}\label{differential hitting t^2 a g}\deg{27}{6}{10}
    $d_6(\tau^4 a n) = \tau^2 a g$.
\end{Lemma}

\begin{proof}
    Consider $\tau^4 a$ in internal coweight $13$. Because of Corollary \ref{representation of rho-localized target} and degree reasons, $\tau^4 a$ must be involved in some differential. There are two options: $\tau^4 a$ either gets hit by a differential originating on $\tau^9 h_2^2$, or it supports a differential hitting $\tau^2 e$.
    
    However, $\tau^9 h_2^2$ itself gets hit by a differential from $\tau^{12} h_1$. To see this, one can for example consider internal coweight $5$ as we did in figure \ref{fig:Rho-Bockstein E_2 rho-free part s+f-w=5 with differentials} where $\tau^4 h_1$ has to hit $\tau h_2^2$. Since $\tau^8$ is a permanent cycle by Corollary \ref{representation of rho-localized target}, we then also have $\tau^{12} h_1 = \tau^8 \cdot \tau^4 h_1$ hitting $\tau^9 h_2^2 = \tau^8 \cdot \tau h_2^2$. So we must have $d_6(\tau^4 a) = \tau^2 e$.
    
    Lastly, we would like to multiply by $n$ and use the relation $\tau^2 e \cdot n = \tau^2 a g$. Note that $n$ does not support any differential for degree reasons. In particular, $n$ exists on $E_6$ and multiplying $\tau^4 a$ by $n$ yields
    \[d_6(\tau^4 a \cdot n) = d_6(\tau^4 a) \cdot n = \tau^2 e \cdot n = \tau^2 a g.\qedhere\]
\end{proof}

\begin{Lemma}\label{differential on t^4 h1^3}\label{differential hitting h1 c}\deg{3}{3}{-1}
    $d_7(\tau^4 h_1^3) = h_1 c$.
\end{Lemma}

\begin{proof}
    Consider internal coweight $7$. By Corollary \ref{representation of rho-localized target} and degree reasons the target $h_1 c$ must be involved in some differential and this is the only option.
\end{proof}

\begin{Lemma}\label{differential on t^6 a}\deg{12}{3}{0}
    $d_r(\tau^6 a) = 0$ for all $r \geq 3$.
\end{Lemma}

\begin{proof}
     By Corollary \ref{representation of rho-localized target} combined with degree reasons the element $\tau^6 a$ must represent $\tau^8 \cdot u$ on $E_\infty$, so $\tau^6 a$ cannot support any differentials. 
\end{proof}

\begin{Lemma}\label{differential on t^4 e}\label{differential hitting t^2 h1 g}\deg{17}{4}{6}
    $d_5(\tau^4 e) = \tau^2 h_1 g$.
\end{Lemma}

\begin{proof}
    Analyze internal coweight $15$. Note that $\tau^4 e$ must be involved in some differential by Corollary \ref{representation of rho-localized target} combined with degree reasons. We claim that it cannot get hit by any differentials. Possible sources are $\tau^6 a$, $\tau^9 c$, and $\tau^{12} h_1^3$. We can rule them out one by one. By Lemma \ref{differential on t^6 a} $\tau^6 a$ does not support any differentials. By Proposition \ref{proof of d_3 differentials} we already know $d_3(\tau^3 h_2^2) = \tau c$. The $\tau^8$-multiple of that is $d_3(\tau^{11} h_2^2) = \tau^9 c$. Lastly, Lemma \ref{differential on t^4 h1^3} and $\tau^8$-multiplication imply $d_7(\tau^{12} h_1^3) = \tau^8 h_1 c$.
    
    Altogether, we showed that $\tau^4 e$ cannot get hit by any differential, so it must support a differential and the claimed one is the only option.
\end{proof}

\begin{Lemma}\label{differential on t^4 h1 d}\label{differential hitting t h0^2 g}\deg{15}{5}{5}
    $d_6(\tau^4 h_1 d) = \tau h_0^2 g$.
\end{Lemma}

\begin{proof}
    Analyze internal coweight $15$. Note that $\tau^4 h_1 d$ must be involved in some differential by Corollary \ref{representation of rho-localized target} combined with degree reasons. It could be the target of a differential originating on $\tau^8 h_1 c$, but by combining Lemma \ref{differential on t^4 h1^3} with $\tau^8$-multiplication we know $d_7(\tau^{12} h_1^3) = \tau^8 h_1 c$. In particular $\tau^8 h_1 c$ cannot support a differential. The claimed differential on $\tau^4 h_1 d$ is then forced.
\end{proof}

\begin{Lemma}\label{differential on t^3 n^2}\label{differential hitting t^2 a g + t Dc}\deg{30}{6}{13}
    $d_3(\tau^3 n^2) = \tau^2 a g + \tau \Dc$.
\end{Lemma}

\begin{proof}
    To see that $\tau^2 a g$ is a summand in $d_3(\tau^3 n^2)$ use the relation $\tau^3 n^2 \cdot \tau^4 \Dh = \tau^8 \cdot \tau^2 n \cdot g^2$ and compute $d_3$ on both sides.

    For degree reasons we have the two remaining options $d_3(\tau^3 n^2) = \tau^2 a g$ and $d_3(\tau^3 n^2) = \tau^2 a g + \tau \Dc$. We will lead the assumption $d_3(\tau^3 n^2) = \tau^2 a g$ to a contradiction. By Lemma \ref{differential on t^4 a n} we have $d_6(\tau^4 a n) = \tau^2 a g$. Then if $d_3(\tau^3 n^2) = \tau^2 a g$ we get $d_6(\tau^4 a n) = \tau^2 a g = 0$. For degree reasons $\tau^4 a n$ cannot support any other differentials. We will show next that $\tau^4 a n$ does not get hit by any differentials. Altogether this will imply that $\tau^4 a n$ is a $\rho$-torsion free element on $E_\infty$. But that is in contradiction with Corollary \ref{representation of rho-localized target}, which implies that there are no $\rho$-torsion free elements in internal coweight $23$ and Adams filtration $6$ where $\tau^4 a n$ lives.

    To complete the proof we need to show that $\tau^4 a n$ cannot get hit by any differentials. There are two elements that could support a differential onto $\tau^4 a n$, namely $\tau^{10} h_1 g$ and $\tau^{12} h_1 d$. Both of these are $\tau^8$-multiples and by Lemma \ref{differential hitting t^2 h1 g} and Lemma \ref{differential on t^4 h1 d} they are involved in the differentials $d_5(\tau^{12} e) = \tau^{10} h_1 g$ and $d_6(\tau^{12} h_1 d) = \tau^9 h_0^2 g$.

    In summary, the assumption $d_3(\tau^3 n^2) = \tau^2 a g$ would imply that $\tau^4 a n$ is $\rho$-torsion free on $E_\infty$, which is in contradiction with Corollary \ref{representation of rho-localized target}. The only possibility left is $d_3(\tau^3 n^2) = \tau^2 a g + \tau \Dc$.
\end{proof}

\subsection{\texorpdfstring{$d_4$}{d\_4}-differentials}

\begin{Prop}\label{proof of d_4 differentials}
    Table \ref{tab:Bockstein d_r} describes the non-zero $d_4$-differentials in the $\rho$-Bockstein spectral sequence on all indecomposables on $E_4$.
\end{Prop}

\begin{proof}
    For degree reasons, the only indecomposables that can support a $d_4$-differential are given by table \ref{tab:Possible Bockstein d_4}. If the proof column is empty, then the proof of that differential is contained in this Proposition.
    
    \begin{longtable}{llllc}
    \caption{Possible non-zero $d_4$-differentials on indecomposable elements
    \label{tab:Possible Bockstein d_4}
    } \\
    \toprule
    $x$ & $(s,f,w)$ & $d_4(x)$ & Occurs & Proof\\
    \midrule \endfirsthead
    \caption[]{Possible non-zero $d_4$-differentials on indecomposable elements} \\
    \toprule
    $x$ & $(s,f,w)$ & $d_4(x)$ & Occurs & Proof\\
    \midrule \endhead
    \bottomrule \endfoot
        $\tau^4$ & $(0, 0, -4)$ & $\tau^2 h_2$ & Yes & \ref{Bockstein d_r on powers of tau}\\
        $\tau^2 h_2$ & $(3, 1, 0)$ & $h_2^2$ & No & \\
        $\tau^4 P$ & $(8, 4, 0)$ & $\tau^2 P h_2$ & No & \ref{differential on t^4 P}\\
        $\tau P h_2^2$ & $(14, 6, 7)$ & $h_1^3 d$ & Yes & \\
        $\tau^2 n$ & $(15, 3, 6)$ & $h_2 n$ & Yes & \ref{differential on t^2 n}\\
        $\tau h_0^2 e$ & $(17, 6, 9)$ & $h_1^3 e$ & Yes & \ref{differential on t h0^2 e}\\
        $P a$ & $(20, 7, 10)$ & $h_1 c d$ & Yes & \\
        $P n$ & $(23, 7, 12)$ & $h_1 c e$ & Yes & \ref{differential on P n}\\
        $\tau^4 \Dh h_1$ & $(26, 6, 10)$ & $\tau^2 a e$ & Yes & \ref{differential on t^4 Dh h1}\\
        $\tau^4 \Dc h_1$ & $(33, 8, 14)$ & $a^3$ & Yes & \ref{differential on t^4 Dc h1}\\
        $\tau^4 P \Dh h_1$ & $(34, 10, 14)$ & $\tau^2 P a e$ & Yes & \ref{differential on t^4 P Dh h1}\\
        $\tau^4 P \Dc h_1$ & $(41, 12, 18)$ & $P a^3$ & Yes & \ref{differential on t^4 P Dc h1}\\
    \end{longtable}
    
    The differential $d_4(\tau^4) = \tau^2 h_2$ was proven in Lemma \ref{Bockstein d_r on powers of tau}. Then the differential on $\tau^2 h_2$ does not occur as $\tau^2 h_2$ gets hit by a differential.

    The differentials on $\tau P h_2^2$ and $P a$ follow immediately by looking at internal coweights $13$ and $17$. Both times the claimed target of the differential must be involved in some differential by Corollary \ref{representation of rho-localized target} combined with degree reasons, and the differentials stated in the table are the only options.    

    For all other differentials see the following series of Lemmas.
\end{proof}

It will be convenient to refer to the $d_6$-differential hitting $e$ many times, so we state it here first.

\begin{Lemma}\label{differential on t^2 a}\label{differential hitting e}\deg{12}{3}{4}
    $d_6(\tau^2 a) = e$.
\end{Lemma}

\begin{proof}
    This is immediate from internal coweight $11$ because both $\tau^2 a$ and $e$ must be involved in some differentials and this is the only option.
\end{proof}

The next Lemma will show that $\tau^4 P$ is a permanent cycle. Along the way we will analyze internal coweight $12$ and prove the differential on $\tau^2 n$. For future reference, we include another finding in the statement of the Lemma as well.

\newpage

\begin{Lemma}\label{differential on t^4 P}\label{differential on t^2 n}\label{differential hitting h2 n}\label{differential on t^2 d}\label{differential hitting h1 e}
\leavevmode
    \begin{enumerate}[label=\rm{(\arabic*)}]
        \item \deg{8}{4}{0} $d_r(\tau^4 P) = 0$ for all $r \geq 4$,
        \item \deg{15}{3}{6} $d_4(\tau^2 n) = h_2 n$,
        \item \deg{14}{4}{6} $d_5(\tau^2 d) = h_1 e$.
    \end{enumerate}
\end{Lemma}

\begin{proof}
    The element $\tau^4 P$ has internal coweight $12$ and filtration $4$. According to Corollary \ref{representation of rho-localized target} there are two $\rho$-torsion free elements with internal coweight $12$ and filtration $4$ in the $\rho$-periodicized cohomology of $\mathcal{A}(2)$, namely $\tau^8 \cdot h_1^4$ and $g$. There must be a hidden extension for $\tau^8 \cdot h_1^4$ on $E_\infty$ as the element $h_1^4$ in our spectral sequence is $\tau$-torsion.
    
    We claim that $\rho^4 \tau^4 P$ has to be the target of the hidden $\tau^8$-extension on $h_1^4$. To prove this, we can show that all other elements with filtration $4$ on the given chart must be involved in differentials or represent a different element. The element named $g$ on $E_4$ must represent the element we called $g$ in Corollary \ref{representation of rho-localized target} because there is no other potential $\rho$-torsion free element in degree $(20, 4, 12)$. In particular, $g$ does not get hit by a differential. Then since $\tau^2 n$ cannot survive for degree reasons, it has to hit $h_2 n$ with a $d_4$. Next, by Lemma \ref{differential hitting e} $e$ is hit by a $d_6$, so $h_1 e$ must get hit by a differential of length at most $6$. The only possible source for such a differential is $\tau^2 d$. This leaves $\tau^4 P$ as the only element that can represent $\tau^8 \cdot h_1^4$. In particular it cannot support any differentials.
\end{proof}

\begin{Lemma}\label{differential on t h0^2 e}\label{differential hitting h1^3 e}\deg{17}{6}{9}
    $d_4(\tau h_0^2 e) = h_1^3 e$.
\end{Lemma}

\begin{proof}
    By Lemma \ref{differential hitting e} we know that $h_1^3 e$ must get hit by a differential of length at most $6$. The claimed differential is the only option.
\end{proof}

\begin{Lemma}\label{differential on P n}\label{differential hitting h1 c e}\deg{23}{7}{12}
    $d_4(P n) = h_1 c e$.
\end{Lemma}

\begin{proof}
    By Lemma \ref{differential hitting c} and Lemma \ref{differential hitting e} we know that $h_1 c e = h_1 \cdot c \cdot e$ must get hit by a differential of length at most $6$ and the claimed one is the only option.
\end{proof}

To determine the differential on $\tau^4 \Dh h_1$ and for many other future differentials it will be convenient to refer to the differential that $n$ is involved in, so we will state it here.

\begin{Lemma}\label{differential on t^6 h2^2}\label{differential hitting n}\deg{6}{2}{-2}
    $d_{10}(\tau^6 h_2^2) = n$.
\end{Lemma}

\begin{proof}
    For degree reasons $n$ must be involved in some differential. Looking at internal coweight $10$ we see that $n$ can only be the target of a differential and the possible sources are $\tau^8 h_1^2$ and $\tau^6 h_2^2$. But the first one of these is a product of permanent cycles.
\end{proof}

\begin{Lemma}\label{differential on t^4 Dh h1}\label{differential hitting t^2 a e}\label{differential on t^9 h0^2 e}\label{differential hitting t^4 P n}
\leavevmode
    \begin{enumerate}[label=\rm{(\arabic*)}]
        \item \deg{26}{6}{10} $d_4(\tau^4 \Dh h_1) = \tau^2 a e$,
        \item \deg{17}{6}{1} $d_7(\tau^9 h_0^2 e) = \tau^4 P n$.
    \end{enumerate}
\end{Lemma}

\begin{proof}
    Consider internal coweight $22$, which is where $\tau^4 \Dh h_1$ and $\tau^2 a e$ live. Then observe that $\tau^2 a e$ cannot be $\rho$-torsion free on $E_\infty$ and there are no possible differentials it can support. Hence, it must get hit by some differential. The possible sources are $\tau^{12} P h_1^2$, $\tau^9 h_0^2 e$, $\tau^4 a^2$ and $\tau^4 \Dh h_1$.
    
    Certainly $\tau^{12} P h_1^2 = \tau^8 \cdot \tau^4 P \cdot h_1^2$ cannot support a differential because by Corollary \ref{t^8, h1, h2, u, g are permanent cycles} and Lemma \ref{differential on t^4 P} it is a product of permanent cycles.
    
    Then we claim $\tau^9 h_0^2 e$ has to hit $\tau^4 P \cdot n$. That is because $\tau^4 P$ is a permanent cycle by Lemma \ref{differential on t^4 P} and $n$ gets hit by a $d_{10}$ by Lemma \ref{differential hitting n}, so $\tau^4 P \cdot n$ must get hit by a differential of length at most $10$ and the claimed one is the only option.
    
    Next, $\tau^4 a^2$ cannot support a $d_6$ hitting $\tau^2 a e$ because on $E_6$ it is a square $\tau^4 a^2 = (\tau^2 a)^2$. That $\tau^2 a$ exists on $E_6$ can easily be deduced from internal coweight $11$.
    
    So the only option left is $d_4(\tau^4 \Dh h_1) = \tau^2 a e$.
\end{proof}

Along the way of the proof of the next differential $d_4(\tau^4 \Dc h_1) =  a^3$, we will analyze internal coweight $16$. For future reference, we include some of the findings in the statement of the Lemma.

\begin{Lemma}\label{differential on t^4 Dc h1}\label{differential hitting a^3}\label{differential on t^6 n}\label{differential hitting t^4 g}\label{differential on t^6 d}
\leavevmode
    \begin{enumerate}[label=\rm{(\arabic*)}]
        \item \deg{33}{8}{14} $d_4(\tau^4 \Dc h_1) =  a^3$,
        \item \deg{15}{3}{2} $d_6(\tau^6 n) = \tau^4 g$,
        \item \deg{14}{4}{2} $d_r(\tau^6 d) = 0$ for all $r \geq 4$.
    \end{enumerate}
\end{Lemma}

\begin{proof}
    Analyze internal coweight $27$, which is where $\tau^4 \Dc h_1$ and $a^3$ live. The element $a^3$ must get hit by a differential of length at most $7$ because $a$ gets hit by a $d_7$ by Lemma \ref{differential hitting a}. There are two possible sources for such a differential onto $a^3$, namely $\tau^6 d e$ and $\tau^4 \Dc h_1$.
    
    We claim that $\tau^6 d e = \tau^6 d \cdot e$ is a product of permanent cycles. For $e$ we already saw that it gets hit by a differential in Lemma \ref{differential hitting e} so it cannot support any differentials. To see that $\tau^6 d$ is a permanent cycle one can check internal coweight $16$: Either $\tau^6 d$, $\tau^4 g$ or $\tau^4 h_2 n$ must represent $\tau^8 \cdot h_1 \cdot u$ on $E_\infty$. But $\tau^4 h_2 n = \tau^4 \cdot h_2 \cdot n$ supports a $d_4$. Also $\tau^6 n$ cannot be a permanent cycle for degree reasons so it must hit $\tau^4 g$ with a $d_6$. Consequentially $\tau^6 d$ has to represent $\tau^8 \cdot h_1 \cdot u$ and therefore must be a permanent cycle.

    Now the only option left for $a^3$ to get hit by a differential is the claimed one.
\end{proof}

\begin{Rem}\label{d_6 hitting tau^8 g^2}
    An interesting consequence of $d_6(\tau^6 n) = \tau^4 g$ is the $\tau^4 g$-multiple $d_6(\tau^{10} n g) = \tau^8 g^2$. This seems to be in contradiction with Corollary \ref{representation of rho-localized target} since the $\rho$-localized cohomology should have a non-zero element named $\tau^8 \cdot g^2$. This is not a contradiction, it simply implies the existence of a hidden extension that turns the $E_\infty$ class $\tau^8 g^2$ into a $\rho$-torsion free class. This hidden extension must involve $\D^2$ for degree reasons.
\end{Rem}

\begin{Lemma}\label{differential on t^4 P Dh h1}\label{differential hitting t^2 P a e}\label{differential on t^4 P Dc h1}\label{differential hitting P a^3}
\leavevmode
    \begin{enumerate}[label=\rm{(\arabic*)}]
        \item \deg{34}{10}{14} $d_4(\tau^4 P \Dh h_1) = \tau^2 P a e$,
        \item \deg{41}{12}{18} $d_4(\tau^4 P \Dc h_1) = P a^3$.
    \end{enumerate}
\end{Lemma}

\begin{proof}
    For $\tau^4 P \Dh h_1$ use the differentials on $\tau^4 \Dh h_1$ and $\tau^4 P$ proven in Lemma \ref{differential on t^4 Dh h1} and Lemma \ref{differential on t^4 P} together with the relation $\tau^4 P \cdot \tau^4 \Dh h_1 = \tau^4 \cdot \tau^4 P \Dh h_1$.

    For $\tau^4 P \Dc h_1$ use the differentials on $\tau^4 \Dc h_1$ and $\tau^4 P$ proven in Lemma \ref{differential on t^4 Dc h1} and Lemma \ref{differential on t^4 P} together with the relation $\tau^4 P \cdot \tau^4 \Dc h_1 = \tau^4 \cdot \tau^4 P \Dc h_1$.    
\end{proof}

\subsection{\texorpdfstring{$d_5$}{d\_5}-differentials}

\begin{Prop}\label{proof of d_5 differentials}
    Table \ref{tab:Bockstein d_r} describes the non-zero $d_5$-differentials in the $\rho$-Bockstein spectral sequence on all indecomposables on $E_5$.
\end{Prop}

\begin{proof}
    For degree reasons, the only indecomposables that can support a $d_5$-differential are given by table \ref{tab:Possible Bockstein d_5}.

    \begin{longtable}{llllc}
    \caption{Possible non-zero $d_5$-differentials on indecomposable elements
    \label{tab:Possible Bockstein d_5}
    } \\
    \toprule
    $x$ & $(s,f,w)$ & $d_5(x)$ & Occurs & Proof\\
    \midrule \endfirsthead
    \caption[]{Possible non-zero $d_5$-differentials on indecomposable elements} \\
    \toprule
    $x$ & $(s,f,w)$ & $d_5(x)$ & Occurs & Proof\\
    \midrule \endhead
    \bottomrule \endfoot
        $\tau^6 d$ & $(14, 4, 2)$ & $\tau^4 h_1 e$ & No & \ref{differential on t^6 d}\\
        $\tau^4 d$ & $(14, 4, 4)$ & $\tau^2 h_1 e$ & Yes & \ref{differential on t^4 d}\\
        $\tau^2 d$ & $(14, 4, 6)$ & $h_1 e$ & Yes & \ref{differential on t^2 d}\\
        $\tau^6 e$ & $(17, 4, 4)$ & $\tau^4 h_1 g$ & Yes & \ref{differential on t^6 e}\\
        $\tau^4 e$ & $(17, 4, 6)$ & $\tau^2 h_1 g$ & Yes & \ref{differential on t^4 e}\\
        $\tau^2 e$ & $(17, 4, 8)$ & $h_1 g$ & No & \ref{differential on t^2 e}\\
        $\tau^2 P e$ & $(25, 8, 12)$ & $h_1 d^2$ & Yes & \ref{differential on t^2 P e}\\
        $\tau^8 \Dh h_1^2$ & $(27, 7, 7)$ & $\tau^6 d e$ & Yes & \ref{differential on t^8 Dh h1^2}\\
        $\tau^4 \Dh h_1^2$ & $(27, 7, 11)$ & $\tau^2 d e$ & Yes & \ref{differential on t^4 Dh h1^2}\\
        $\tau^8 \Dc h_1^2$ & $(34, 9, 11)$ & $\tau^4 P n^2$ & Yes & \ref{differential on t^8 Dc h1^2}\\
        $\tau^4 \Dc h_1^2$ & $(34, 9, 15)$ & $P n^2$ & Yes & \ref{differential on t^4 Dc h1^2}\\
        $P a n$ & $(35, 10, 18)$ & $u d^2$ & Yes & \ref{differential on P a n}\\
        $\tau^4 P \Dh h_1^2$ & $(35, 11, 15)$ & $\tau^2 P d e$ & Yes & \ref{differential on t^4 P Dh h1^2}\\
        $\tau^4 P \Dc h_1^2$ & $(42, 13, 19)$ & $P^2 n^2$ & Yes & \ref{differential on t^4 P Dc h1^2}\\
    \end{longtable}

    For $\tau^6 d$ we saw in Lemma \ref{differential on t^6 d} that it is a permanent cycle. Similarly we already proved the differentials on $\tau^2 d$ and $\tau^4 e$ in Lemma \ref{differential on t^2 d} and Lemma \ref{differential on t^4 e}, respectively.

    For all other differentials see the following series of Lemmas.
\end{proof}

\begin{Lemma}\label{differential on t^4 d}\label{differential hitting t^2 h1 e}\deg{14}{4}{4}
    $d_5(\tau^4 d) = \tau^2 h_1 e$.
\end{Lemma}

\begin{proof}
    Consider internal coweight $14$ where $\tau^2 h_1 e$ must be involved in some differential. We claim that $\tau^2 h_1 e$ cannot support a differential. The only possible target is $h_1^2 g$, but if $h_1^2 g$ gets hit by a differential then there is no element left to represent what we called $h_1^2 \cdot g$ in Corollary \ref{representation of rho-localized target}.
    
    The claimed $d_5$ on $\tau^4 d$ is the only remaining differential that $\tau^2 h_1 e$ can be involved in. 
\end{proof}

\begin{Lemma}\label{differential on t^6 e}\label{differential hitting t^4 h1 g}\deg{17}{4}{4}
    $d_5(\tau^6 e) = \tau^4 h_1 g$.
\end{Lemma}

\begin{proof}
    Recall that in Lemma \ref{differential on t^6 n} we saw $d_6(\tau^6 n) = \tau^4 g$. Hence $\tau^4 h_1 g = \tau^4 g \cdot h_1$ has to get hit by a differential of length at most $6$ and in internal coweight $17$ the claimed differential is the only option.
\end{proof}

\begin{Lemma}\label{differential on t^4 a}\label{differential hitting t^2 e}\label{differential on t^2 e}\deg{12}{3}{2}
    $d_6(\tau^4 a) = \tau^2 e$, so in particular $d_r(\tau^2 e) = 0$ for all $r \geq 5$.
\end{Lemma}

\begin{proof}
    In internal coweight $13$ the element $\tau^4 a$ must be involved in some differential. We claim that it cannot get hit by a differential. The only possible source for such a differential is $\tau^9 h_2^2$. From the discussion of internal coweight $5$ surrounding figure \ref{fig:Rho-Bockstein E_2 rho-free part s+f-w=5 with differentials} in section \ref{section: The internal coweight method} we already know $d_6(\tau^{4} h_1) = \tau h_2^2$. Via $\tau^8$-multiplication we get $d_6(\tau^{12} h_1) = \tau^9 h_2^2$.

    Then $\tau^4 a$ must support a differential and the claimed one is the only option. In particular this means that $\tau^2 e$ cannot support any differentials.
\end{proof}

For the differential on $\tau^4 c$ we want to refer to a $d_7$ hitting $d$, so we state it first.

\begin{Lemma}\label{differential on t^4 c}\label{differential hitting d}\deg{8}{3}{1}
    $d_7(\tau^4 c) = d$.
\end{Lemma}

\begin{proof}
    Immediate from internal coweight $10$ and the fact that $d$ cannot be $\rho$-torsion free on $E_\infty$ for degree reasons.
\end{proof}

\begin{Lemma}\label{differential on t^2 P e}\label{differential hitting h1 d^2}\deg{25}{8}{12}
    $d_5(\tau^2 P e) = h_1 d^2$.
\end{Lemma}

\begin{proof}
    Recall from Lemma \ref{differential hitting d} that $d$ gets hit by a $d_7$. This implies that in internal coweight $21$ the element $h_1 d^2$ has to get hit by a differential of length at most $7$. The only option is the given differential.
\end{proof}

\begin{Lemma}\label{differential on t^8 Dh h1^2}\label{differential hitting t^6 d e}\deg{27}{7}{7}
    $d_5(\tau^8 \Dh h_1^2) = \tau^6 d e$.
\end{Lemma}

\begin{proof}
    Note that the target is decomposable $\tau^6 d e = \tau^6 d \cdot e$. We saw in Lemma \ref{differential on t^6 d} that $\tau^6 d$ is a permanent cycle. By Lemma \ref{differential hitting e} $e$ gets hit by a $d_6$. So $\tau^6 d e$ must get hit by a differential of length at most $6$. Looking at internal coweight $27$, there are two possible sources, namely $\tau^8 a d$ and $\tau^8 \Dh h_1^2$. But $\tau^8 a d = \tau^8 \cdot a \cdot d$ is a permanent cycle by Lemma \ref{differential hitting a} and Lemma \ref{differential hitting d}.
\end{proof}

\begin{Lemma}\label{differential on t^4 Dh h1^2}\label{differential hittingt^2 d e}\deg{27}{7}{11}
    $d_5(\tau^4 \Dh h_1^2) = \tau^2 d e$.
\end{Lemma}

\begin{proof}
    Analyze internal coweight $23$ where $\tau^4 \Dh h_1^2$ and $\tau^2 d e$ live. Note that $\tau^2 d e = d \cdot \tau^2 e$ is a permanent cycle because both factors are permanent cycles by Lemma \ref{differential hitting d} and Lemma \ref{differential hitting t^2 e}. But $\tau^2 d e$ must also be involved in some differential for degree reasons. Since $\tau^2 e$ gets hit by a $d_6$, $\tau^2 d e$ must get hit by a differential of length at most $6$. The possible sources are $\tau^4 a d$ and $\tau^4 \Dh h_1^2$. By Lemma \ref{differential on t^6 a} and Lemma \ref{differential on t^6 d} both $\tau^6 a$ and $\tau^6 d$ are permanent cycles. In particular, this means that $\tau^6 a \cdot \tau^6 d = \tau^8 \cdot \tau^4 a d$ is also a permanent cycle. So $\tau^4 a d$ can only support a differential hitting a $\tau^8$-torsion class, which $\tau^2 d e$ is not (and $\tau^2 d e$ also cannot become $\tau^8$-torsion until a hypothetical $d_6$ on $\tau^4 a d$ would hit it).
\end{proof}

\begin{Lemma}\label{differential on t^8 Dc h1^2}\label{differential hitting t^4 P n^2}\deg{34}{9}{11}
    $d_5(\tau^8 \Dc h_1^2) = \tau^4 P n^2$.
\end{Lemma}

\begin{proof}
    Consider internal coweight $32$ where $\tau^8 \Dc h_1^2$ and $\tau^4 P n^2$ live. Note that the purported target $\tau^4 P n^2 = \tau^4 P \cdot n^2$ must get hit by some differential since Lemma \ref{differential hitting n} says that $n$ gets hit by a $d_{10}$ and $\tau^4 P$ is a permanent cycle. The only option is the claimed $d_5$.
\end{proof}

\begin{Lemma}\label{differential on t^4 Dc h1^2}\label{differential hitting P n^2}\deg{34}{9}{15}
    $d_5(\tau^4 \Dc h_1^2) = P n^2$.
\end{Lemma}

\begin{proof}
    Look at internal coweight $28$, where $P n^2$ must get hit by a differential. There are two possible sources, namely $\tau^8 P h_1 e$ and $\tau^4 \Dc h_1^2$. Then note that $\tau^2 P a^2$ must get hit by a differential and the only possibility for that is a differential originating on $\tau^8 P h_1 e$. So $\tau^4 \Dc h_1^2$ must hit $P n^2$.
\end{proof}

\begin{Lemma}\label{differential on P a n}\label{differential hitting u d^2}\deg{35}{10}{18}
    $d_5(P a n) = u d^2$.
\end{Lemma}

\begin{proof}
    Consider internal coweight $27$ where $u d^2 = u \cdot d^2$ lives. It has to get hit by a differential of length at most $7$ because $d$ is the target of a $d_7$ by Lemma \ref{differential hitting d} and $u$ is a permanent cycle by Corollary \ref{t^8, h1, h2, u, g are permanent cycles}. The only option is the claimed differential.
\end{proof}

\begin{Lemma}\label{differential on t^4 P Dh h1^2}\label{differential hitting t^2 P d e}\label{differential on t^4 P Dc h1^2}\label{differential hitting P^2 n^2}
\leavevmode
    \begin{enumerate}[label=\rm{(\arabic*)}]
        \item \deg{35}{11}{15} $d_5(\tau^4 P \Dh h_1^2) = \tau^2 P d e$,
        \item \deg{42}{13}{19} $d_5(\tau^4 P \Dc h_1^2) = P^2 n^2$.
    \end{enumerate}
\end{Lemma}

\begin{proof}
    For $\tau^4 P \Dh h_1^2$ we can use the relation $\tau^4 P \cdot \tau^8 \Dh h_1^2 = \tau^8 \cdot \tau^4 P \Dh h_1^2$ and the $d_5$-differential on $\tau^8 \Dh h_1^2$ that we already know from Lemma \ref{differential on t^8 Dh h1^2}.
    
    For $\tau^4 P \Dc h_1^2$ we can use the relation $\tau^4 P \cdot \tau^8 \Dc h_1^2 = \tau^8 \cdot \tau^4 P \Dc h_1^2$ and the $d_5$-differential on $\tau^8 \Dc h_1^2$ that we already know from Lemma \ref{differential on t^8 Dc h1^2}.
\end{proof}

\subsection{\texorpdfstring{$d_6$}{d\_6}-differentials}

\begin{Prop}\label{proof of d_6 differentials}
    Table \ref{tab:Bockstein d_r} describes the non-zero $d_6$-differentials in the $\rho$-Bockstein spectral sequence on all indecomposables on $E_6$.
\end{Prop}

\begin{proof}
    For degree reasons, the only indecomposables that can support a $d_6$-differential are given by table \ref{tab:Possible Bockstein d_6}. If the proof column is empty, then the proof of that differential is contained in this Proposition.

    \begin{longtable}{llllc}
    \caption{Possible non-zero $d_6$-differentials on indecomposable elements
    \label{tab:Possible Bockstein d_6}
    } \\
    \toprule
    $x$ & $(s,f,w)$ & $d_6(x)$ & Occurs & Proof\\
    \midrule \endfirsthead
    \caption[]{Possible non-zero $d_6$-differentials on indecomposable elements} \\
    \toprule
    $x$ & $(s,f,w)$ & $d_6(x)$ & Occurs & Proof\\
    \midrule \endhead
    \bottomrule \endfoot
        $\tau^4 h_1$ & $(1, 1, -3)$ & $\tau h_2^2$ & Yes & \\
        $\tau^2 P h_2$ & $(11, 5, 4)$ & $h_1^2 d$ & Yes & \\
        $\tau^4 a$ & $(12, 3, 2)$ & $\tau^2 e$ & Yes & \ref{differential on t^4 a}\\
        $\tau^2 a$ & $(12, 3, 4)$ & $e$ & Yes & \ref{differential on t^2 a}\\
        $\tau^6 n$ & $(15, 3, 2)$ & $\tau^4 g$ & Yes & \ref{differential on t^6 n}\\
        $\tau^3 h_0^2 e$ & $(17, 6, 7)$ & $c d$ & Yes & \\
        $\tau^8 P n$ & $(23, 7, 4)$ & $\tau^6 P g$ & Yes & \ref{differential on t^8 P n}\\
        $\tau^6 P n$ & $(23, 7, 6)$ & $\tau^4 d^2$ & No & \ref{differential on t^6 P n}\\
        $\tau^2 P n$ & $(23, 7, 10)$ & $d^2$ & Yes & \ref{differential on t^2 P n}\\
        $\tau^6 P h_1 d$ & $(23, 9, 7)$ & $\tau^3 P h_0^2 g$ & Yes & \ref{differential on t^6 P h1 d}\\
        $\tau^5 \Dh$ & $(25, 5, 8)$ & $\tau^2 n^2$ & No & \ref{differential on t^5 Dh}\\
        $\tau^2 n^2$ & $(30, 6, 14)$ & $n g$ & Yes & \ref{differential on t^2 n^2}\\
        $\tau^{12} d e$ & $(31, 8, 6)$ & $\tau^6 a^3$ & Yes & \ref{differential on t^12 d e}\\
        $\tau^4 P^2 n$ & $(31, 11, 12)$ & $\tau^2 P^2 g$ & Yes & \ref{differential on t^4 P^2 n}\\
        $\tau^2 P^2 n$ & $(31, 11, 14)$ & $P^2 g$ & No & \ref{differential on t^2 P^2 n}\\
        $\tau^8 P d e$ & $(39, 12, 14)$ & $\tau^2 P a^3$ & Yes & \ref{differential on t^8 P d e}\\
        $\tau^6 g^2$ & $(40, 8, 18)$ & $\tau \Dh g$ & Yes & \ref{differential on t^6 g^2}\\
    \end{longtable}

    The differentials on $\tau^4 a$, $\tau^2 a$, and $\tau^6 n$ were already recorded in Lemma \ref{differential on t^4 a}, Lemma \ref{differential hitting e}, and Lemma \ref{differential on t^6 n}.

    The differential on $\tau^4 h_1$ was already proven when we discussed internal coweight $5$ in section \ref{section: The internal coweight method}, see also figure \ref{fig:Rho-Bockstein E_2 rho-free part s+f-w=5 with differentials}. The differentials on $\tau^2 P h_2$ and $\tau^3 h_0^2 e$ are also immediate from their respective internal coweights $12$ and $16$.

    For all other differentials see the following series of Lemmas.
\end{proof}

\begin{Lemma}\label{differential on t^8 P n}\label{differential hitting t^6 P g}\deg{23}{7}{4}
    $d_6(\tau^8 P n) = \tau^6 P g$.
\end{Lemma}

\begin{proof}
    Consider internal coweight $26$ and note that $\tau^6 P g$ must be involved in some differential and the only option is that it gets hit by one. The elements that can possibly support a differential onto $\tau^6 P g$ are $\tau^{12} P c$ and $\tau^8 P n$. But $\tau^{12} P c = \tau^8 \cdot \tau^4 P \cdot c$ is a product of permanent cycles by Corollary \ref{t^8, h1, h2, u, g are permanent cycles}, Lemma \ref{differential on t^4 P}, and Lemma \ref{differential hitting c} so it cannot support any differentials.
\end{proof}

\begin{Lemma}\label{differential on t^11 h0^2 e}\label{differential hitting t^6 P n}\label{differential on t^6 P n}\deg{17}{6}{-1}
    $d_7(\tau^{11} h_0^2 e) = \tau^6 P n$, so in particular $d_r(\tau^6 P n) = 0$ for all $r \geq 6$.
\end{Lemma}

\begin{proof}
    In internal coweight $24$ the element $\tau^{11} h_0^2 e$ must be involved in a differential for degree reasons. It cannot get hit so it must support a differential. The two possible targets are $\tau^6 P n$ and $\tau^4 a e$. But we already know $d_6(\tau^4 a) = \tau^2 e$ by Lemma \ref{differential on t^4 a}, so $d_6(\tau^4 a e) = d_6(\tau^4 a) \cdot e = \tau^2 e \cdot e = \tau^2 d g \neq 0$. In particular $\tau^{11} h_0^2 e$ cannot hit $\tau^4 a e$ so it must hit $\tau^6 P n$. In turn this means $\tau^6 P n$ cannot support any differentials.
\end{proof}

\begin{Lemma}\label{differential on t^2 P n}\label{differential hitting d^2}\deg{23}{7}{10}
    $d_6(\tau^2 P n) = d^2$.
\end{Lemma}

\begin{proof}
    Consider internal coweight $20$, where $\tau^2 P n$ and $d^2$ live. By Lemma \ref{differential hitting d} $d$ gets hit by a $d_7$, so $d^2$ must get hit by a differential of length at most $7$. The claimed one is the only option.
\end{proof}

\begin{Lemma}\label{differential on t^6 P h1 d}\label{differential hitting t^3 P h0^2 g}\deg{23}{9}{7}
    $d_6(\tau^6 P h_1 d) = \tau^3 P h_0^2 g$.
\end{Lemma}

\begin{proof}
    Recall $d_6(\tau^4 h_1) = \tau h_2^2$ from Proposition \ref{proof of d_6 differentials}. Then use the relation $\tau^4 h_1 \cdot \tau^4 P \cdot \tau^6 d = \tau^8 \cdot \tau^6 P h_1 d$ together with the already known facts that $\tau^6 d$ and $\tau^4 P$ are permanent cycles by Lemma \ref{differential on t^4 Dc h1} and Lemma \ref{differential on t^4 P}.
\end{proof}

\begin{Lemma}\label{differential on t^5 Dh}\deg{25}{5}{8}
    $d_r(\tau^5 \Dh) = 0$ for all $r \geq 6$.
\end{Lemma}

\begin{proof}
    Consider internal coweight $22$ and observe that $\tau^5 \Dh$ must represent $\tau^8 \cdot h_2 \cdot g$ in the $\rho$-localized cohomology of $\mathcal{A}(2)$. In particular it must be a permanent cycle.
\end{proof}

\begin{Lemma}\label{differential on t^2 n^2}\label{differential hitting n g}\deg{30}{6}{14}
    $d_6(\tau^2 n^2) = n g$.
\end{Lemma}

\begin{proof}
    Recall from Lemma \ref{differential hitting n} that $n$ gets hit by a $d_{10}$. So $n g = n \cdot g$ must get hit by a differential of length at most $10$ and the only possible source for that differential is $\tau^2 n^2$.
\end{proof}

\begin{Lemma}\label{differential on t^12 d e}\label{differential hitting t^6 a^3}\deg{31}{8}{6}
    $d_6(\tau^{12} d e) = \tau^6 a^3$.
\end{Lemma}

\begin{proof}
    Recall from Lemma \ref{differential hitting a} that $a$ gets hit by a $d_7$. By Lemma \ref{differential on t^6 a} $\tau^6 a$ is a permanent cycle. So in particular the product $\tau^6 a^3 = \tau^6 a \cdot a \cdot a$ must get hit by a differential of length at most $7$. The claimed $d_6$ on $\tau^{12} d e$ is the only option.
\end{proof}

\begin{Lemma}\label{differential on t^4 P^2 n}\label{differential hitting t^2 P^2 g}\deg{31}{11}{12}
    $d_6(\tau^4 P^2 n) = \tau^2 P^2 g$.
\end{Lemma}

\begin{proof}
    Combine the relation $\tau^4 P \cdot \tau^8 P n = \tau^8 \cdot \tau^4 P^2 n$ with Lemma \ref{differential on t^8 P n} and compute $d_6$ on both sides.
\end{proof}

\begin{Lemma}\label{differential on t^7 P h0^2 e}\label{differential hitting t^2 P^2 n}\label{differential on t^2 P^2 n}\deg{25}{10}{7}
    $d_7(\tau^7 P h_0^2 e) = \tau^2 P^2 n$, so in particular $d_r(\tau^2 P^2 n) = 0$ for all $r \geq 6$.
\end{Lemma}

\begin{proof}
    Observe in internal coweight $28$ that $\tau^7 P h_0^2 e$ must disappear and the only option is that it hits $\tau^2 P^2 n$. In particular $\tau^2 P^2 n$ cannot support any non-trivial differential.
\end{proof}

\begin{Lemma}\label{differential on t^8 P d e}\label{differential hitting t^2 P a^3}\deg{39}{12}{14}
    $d_6(\tau^8 P d e) = \tau^2 P a^3$.
\end{Lemma}

\begin{proof}
    Combine the relation $\tau^4 P \cdot \tau^{12} d e = \tau^8 \cdot \tau^8 P d e$ with Lemma \ref{differential on t^12 d e} and compute $d_6$ on both sides.
\end{proof}

\begin{Lemma}\label{differential on t^6 g^2}\label{differential hitting t Dh g}\label{differential on t^6 h2 n}\label{differential hitting t Dh}
\leavevmode
    \begin{enumerate}[label=\rm{(\arabic*)}]
        \item \deg{40}{8}{18} $d_6(\tau^6 g^2) = \tau \Dh g$,
        \item \deg{18}{4}{4} $d_8(\tau^6 h_2 n) = \tau \Dh$.
    \end{enumerate}
\end{Lemma}

\begin{proof}
    Consider first internal coweight $18$ where $\tau \Dh$ lives. For degree reasons $\tau \Dh$ must get hit by a differential and the two possible sources are $\tau^8 d$ and $\tau^6 h_2 n$. But from Lemma \ref{differential hitting d} we know $d_7(\tau^4 c) = d$, so $\tau^8 d = \tau^8 \cdot d$ is a permanent cycle. In particular this means $\tau \Dh$ is hit via $d_8(\tau^6 h_2 n) = \tau \Dh$. Now $\tau \Dh g$ must also get hit by a differential of length at most $8$ and the claimed $d_6$ on $\tau^6 g^2$ is the only option left.
\end{proof}

\subsection{\texorpdfstring{$d_7$}{d\_7}-differentials}

\begin{Prop}\label{proof of d_7 differentials}
    Table \ref{tab:Bockstein d_r} describes the non-zero $d_7$-differentials in the $\rho$-Bockstein spectral sequence on all indecomposables on $E_7$.
\end{Prop}

\begin{proof}
    For degree reasons, the only indecomposables that can support a $d_7$-differential are given by table \ref{tab:Possible Bockstein d_7}. If the proof column is empty, then the proof of that differential is contained in this Proposition.
    
    \begin{longtable}{llllc}
    \caption{Possible non-zero $d_7$-differentials on indecomposable elements
    \label{tab:Possible Bockstein d_7}
    } \\
    \toprule
    $x$ & $(s,f,w)$ & $d_7(x)$ & Occurs & Proof\\
    \midrule \endfirsthead
    \caption[]{Possible non-zero $d_7$-differentials on indecomposable elements} \\
    \toprule
    $x$ & $(s,f,w)$ & $d_7(x)$ & Occurs & Proof\\
    \midrule \endhead
    \bottomrule \endfoot
        $\tau^4 h_1^2$ & $(2, 2, -2)$ & $c$ & Yes & \ref{differential on t^4 h1^2}\\
        $\tau^5 h_2^2$ & $(6, 2, -1)$ & $a$ & Yes & \ref{differential on t^5 h2^2}\\
        $\tau^4 c$ & $(8, 3, 1)$ & $d$ & Yes & \ref{differential on t^4 c}\\
        $\tau^{10} P h_2$ & $(11, 5, -4)$ & $\tau^7 h_0^2 e$ & Yes & \ref{differential on t^10 P h2}\\
        $\tau^{11} h_0^2 e$ & $(17, 6, -1)$ & $\tau^6 P n$ & Yes & \ref{differential on t^11 h0^2 e}\\
        $\tau^9 h_0^2 e$ & $(17, 6, 1)$ & $\tau^4 P n$ & Yes & \ref{differential on t^9 h0^2 e}\\
        $\tau^6 P^2 h_2$ & $(19, 9, 4)$ & $\tau^3 P h_0^2 e$ & Yes & \\
        $\tau^7 P h_0^2 e$ & $(25, 10, 7)$ & $\tau^2 P^2 n$ & Yes & \ref{differential on t^7 P h0^2 e}\\
        $\tau^5 P h_0^2 e$ & $(25, 10, 9)$ & $P^2 n$ & Yes & \ref{differential on t^5 P h0^2 e}\\
    \end{longtable}
    
    The differentials on $\tau^4 h_1^2$, $\tau^5 h_2^2$, $\tau^4 c$, $\tau^{11} h_0^2 e$, and $\tau^7 P h_0^2 e$ were already proven in Lemmas \ref{differential hitting c}, \ref{differential hitting a}, \ref{differential hitting d}, \ref{differential on t^11 h0^2 e}, and \ref{differential on t^7 P h0^2 e} respectively.

    The differential on $\tau^6 P^2 h_2$ is immediate from internal coweight $24$.

    For all other differentials see the following series of Lemmas.
\end{proof}

\begin{Lemma}\label{differential on t^10 P h2}\label{differential hitting t^7 h0^2 e}\deg{11}{5}{-4}
    $d_7(\tau^{10} P h_2) = \tau^7 h_0^2 e$.
\end{Lemma}

\begin{proof}
    Look at internal coweight $20$. Checking degrees shows that $\tau^{10} P h_2$ must be involved in some differential. There are no possible differentials that can hit $\tau^{10} P h_2$, so it must be the source of a differential. The only possible targets are $\tau^7 h_0^2 e$ and $n^2$. Recall from Lemma \ref{differential hitting n} that $n$ gets hit by a $d_{10}$, so $n^2$ must get hit by a differential of length at most $10$. This implies that $\tau^{10} P h_2$ cannot hit $n^2$, so it must hit $\tau^7 h_0^2 e$.
\end{proof}

\begin{Lemma}\label{differential on t^5 P h0^2 e}\label{differential hitting P^2 n}\deg{25}{10}{9}
    $d_7(\tau^5 P h_0^2 e) = P^2 n$.
\end{Lemma}

\begin{proof}
    Consider internal coweight $26$, where $P^2 n$ must be involved in some differential. Then note that there are no targets for differentials on $P^2 n$, so it must get hit by a differential. There are only two possible sources, $\tau^8 P^2 h_1^2$ and $\tau^5 P h_0^2 e$. But $\tau^8 P^2 h_1^2 = (\tau^4 P)^2 \cdot h_1^2$ is a permanent cycle.
\end{proof}

\subsection{\texorpdfstring{$d_8$}{d\_8}-differentials}

\begin{Prop}\label{proof of d_8 differentials}
    Table \ref{tab:Bockstein d_r} describes the non-zero $d_8$-differentials in the $\rho$-Bockstein spectral sequence on all indecomposables on $E_8$.
\end{Prop}

\begin{proof}
    For degree reasons, the only possible $d_8$-differential on an indecomposable is $d_8(\tau^6 h_2 n) = \tau \Dh$. This differential occurs, as was proven in Lemma \ref{differential on t^6 h2 n}.
\end{proof}

\subsection{\texorpdfstring{$d_9$}{d\_9}-differentials}

\begin{Prop}\label{proof of d_9 differentials}
    There are no non-zero $d_9$-differentials in the $\rho$-Bockstein spectral sequence.
\end{Prop}

\begin{proof}
    For degree reasons, there are no possible non-zero $d_9$-differentials on indecomposable elements.
\end{proof}

\subsection{\texorpdfstring{$d_{10}$}{d\_10}-differentials}

\begin{Prop}\label{proof of d_10 differentials}
    Table \ref{tab:Bockstein d_r} describes the non-zero $d_{10}$-differentials in the $\rho$-Bockstein spectral sequence on all indecomposables on $E_{10}$.
\end{Prop}

\begin{proof}
    For degree reasons, the only possible differential on an indecomposable on $E_{10}$ is $d_{10}(\tau^6 h_2^2) = n$. This differential occurs, as was proven in Lemma \ref{differential on t^6 h2^2}.
\end{proof}

\subsection{\texorpdfstring{$d_r$}{d\_r}-differentials for \texorpdfstring{$r \geq 11$}{r greater or equal to 11}}

\begin{Prop}\label{E_11 is E_infty}
    There are no non-zero $d_r$-differentials for $r \geq 11$ in the $\rho$-Bockstein spectral sequence. Hence, $E_{11} \cong E_{\infty}$.
\end{Prop}

\begin{proof}
    Every indecomposable on the $\rho$-free quotient of $E_{11}$ represents some non-zero element in the $\rho$-localization described in Corollary \ref{representation of rho-localized target}. For example $\tau^4 P$ has to represent $\tau^8 \cdot h_1^4$ and then $P^2$ has to represent $\tau^8 \cdot h_1^8$, or $\tau^6 a$ has to represent $\tau^8 \cdot u$ and then $\tau^6 d$ has to represent $\tau^8 \cdot h_1 \cdot u$, and so on.
    
    In particular, no indecomposable can support a $d_r$-differential for any $r \geq 11$.
\end{proof}

\subsection{Indecomposables on \texorpdfstring{$E_\infty$}{E\_infinity}}

By inspection, there are $56$ indecomposables on the $\rho$-Bockstein $E_\infty$-page. We list them here in table \ref{tab:Indecomposables on E_infty}. In some cases, a choice has to be made regarding the name of the indecomposable. For example in degree $(32, 7, 12)$ we have an $\mathbb{F}_2$-basis consisting of $\tau^5 \Dc$ and $\tau^6 a g$. The element $\tau^6 a g$ is decomposable, but the other two elements $\tau^5 \Dc$ and $\tau^5 \Dc + \tau^6 a g$ are indecomposable. We choose to include $\tau^5 \Dc + \tau^6 a g$ in our table. The reason for this is that we believe that this element better represents the structure of the $E_\infty$-page because $\tau^5 \Dc + \tau^6 a g$ is the unique $\rho$-power-torsion element in degree $(32, 7, 12)$.

Based on the following remark, table \ref{tab:Indecomposables on E_infty} also describes the indecomposables in the cohomology of $\mathcal{A}(2)$.

\begin{Rem}\label{remark on hidden extensions and indecomposables}
    Note that in general the indecomposables on the $E_\infty$-page of a spectral sequence do not have to be the same as the indecomposables in the abutment of a spectral sequence. That is because an indecomposable element on $E_\infty$ can be the target of a hidden extension, making it decomposable in the abutment. However, this does not occur in the $\rho$-Bockstein spectral sequence because the target of a hidden extension is always decomposable: Consider three elements $x$, $y$, $z$ and assume there is a hidden $x$-extension from $y$ to $z$. This implies that the $\rho$-filtration of $z$ is strictly larger than the sum of the $\rho$-filtration of $x$ and $y$. In particular, the $\rho$-filtration of $z$ is at least $1$, meaning it is $\rho$-divisible and hence decomposable.
\end{Rem}

\begin{longtable}{ll}
    \caption{Indecomposables on the $\rho$-Bockstein $E_\infty$-page or equivalently in $\Ext_{\mathcal{A}(2)}(\mathbb{M}_2, \mathbb{M}_2)$
    \label{tab:Indecomposables on E_infty}
    } \\
    \toprule
    Indecomposable & $(s,f,w)$\\
    \midrule \endfirsthead
    \caption[]{Indecomposables on the $\rho$-Bockstein $E_\infty$-page or equivalently in $\Ext_{\mathcal{A}(2)}(\mathbb{M}_2, \mathbb{M}_2)$} \\
    \toprule
    Indecomposable & $(s,f,w)$\\
    \midrule \endhead
    \bottomrule \endfoot
        $\rho$ & $(-1, 0, -1)$\\
        $\tau^8$ & $(0, 0, -8)$\\
        $\tau^6 h_0$ & $(0, 1, -6)$\\
        $\tau^4 h_0$ & $(0, 1, -4)$\\
        $\tau^2 h_0$ & $(0, 1, -2)$\\
        $h_0$ & $(0, 1, 0)$\\
        $\tau^5 h_1$ & $(1, 1, -4)$\\
        $\tau h_1$ & $(1, 1, 0)$\\
        $h_1$ & $(1, 1, 1)$\\
        $\tau^2 h_2$ & $(3, 1, 0)$\\
        $h_2$ & $(3, 1, 2)$\\
        $\tau h_2^2$ & $(6, 2, 3)$\\
        $\tau^5 c$ & $(8, 3, 0)$\\
        $\tau c$ & $(8, 3, 4)$\\
        $c$ & $(8, 3, 5)$\\
        $\tau^4 P$ & $(8, 4, 0)$\\
        $\tau^2 P h_0$ & $(8, 5, 2)$\\
        $P h_0$ & $(8, 5, 4)$\\
        $\tau P h_1$ & $(9, 5, 4)$\\
        $u$ & $(11, 3, 7)$\\
        $P h_2$ & $(11, 5, 6)$\\
        $\tau^6 a$ & $(12, 3, 0)$\\
        $a$ & $(12, 3, 6)$\\
        $\tau^6 d$ & $(14, 4, 2)$\\
        $d$ & $(14, 4, 8)$\\
        $n$ & $(15, 3, 8)$\\
        $\tau P c$ & $(16, 7, 8)$\\
        $P^2$ & $(16, 8, 8)$\\
        $\tau^2 e$ & $(17, 4, 8)$\\
        $e$ & $(17, 4, 10)$\\
        $\tau^7 h_0^2 e$ & $(17, 6, 3)$\\
        $\tau^5 h_0^2 e$ & $(17, 6, 5)$\\
        $\tau^4 g$ & $(20, 4, 8)$\\
        $g$ & $(20, 4, 12)$\\
        $\tau^2 P a$ & $(20, 7, 8)$\\
        $\tau^2 P d$ & $(22, 8, 10)$\\
        $\tau^6 P n$ & $(23, 7, 6)$\\
        $\tau^4 a^2$ & $(24, 6, 8)$\\
        $\tau^5 \Dh$ & $(25, 5, 8)$\\
        $\tau \Dh$ & $(25, 5, 12)$\\
        $\tau^3 P h_0^2 e$ & $(25, 10, 11)$\\
        $\tau P h_0^2 e$ & $(25, 12, 13)$\\
        $\tau^4 a d$ & $(26, 7, 10)$\\
        $\tau^2 a n$ & $(27, 6, 12)$\\
        $\tau^4 d^2$ & $(28, 8, 12)$\\
        $\tau^4 n^2$ & $(30, 6, 12)$\\
        $\tau^2 P^2 n$ & $(31, 11, 14)$\\
        $\tau^5 \Dc + \tau^6 a g$ & $(32, 7, 12)$\\
        $\tau \Dc + \tau^2 a g$ & $(32, 7, 16)$\\
        $P a^2$ & $(32, 10, 16)$\\
        $\tau P \Dh$ & $(33, 9, 16)$\\
        $P a d$ & $(34, 11, 18)$\\
        $\tau^2 a^3$ & $(36, 9, 16)$\\
        $\tau^2 P n^2$ & $(38, 10, 18)$\\
        $\tau P \Dc + \tau^2 P a g$ & $(40, 11, 20)$\\
        $\D^2$ & $(48, 8, 24)$\\
\end{longtable}

\section{Hidden extensions on the \texorpdfstring{$\rho$}{rho}-Bockstein \texorpdfstring{$E_\infty$}{E\_infinity}-page}\label{section: hidden extensions}

For any multiplicative spectral sequence, the multiplicative structure of the abutment of the spectral sequence may differ from that of $E_\infty$. On the $E_\infty$-page, the filtration degree of a product is the sum of the filtration degrees of the factors. In the abutment, the filtration degree of a product can be larger than the sum of the filtration degrees of the factors. Roughly speaking, if this happens one typically says that there is a hidden extension. For a more precise definition of the term hidden extension see \cite[Section 4.1.1]{Isa19}. We will use the cited definition here.

We will consider some but not all hidden extensions. The reason being that $E_\infty$ has too many indecomposables. For now, we will only consider hidden extensions by $\tau^8$, $h_0$, $h_1$, and $h_2$. There is a slight amount of care needed when talking about $h_0$ because the $E_\infty$ class of the same name represents two elements in $\Ext$ due to the presence of $\rho h_1$ in the same degree. On $E_\infty$, $\rho h_1$ is $\rho$-torsion free and $h_0$ is $\rho$-torsion, so one of the elements represented by $h_0$ is $\rho$-torsion free and the other one is $\rho$-torsion. When we talk about hidden $h_0$-extensions we specifically mean hidden extensions by the $\rho$-torsion element represented by $h_0$ on $E_\infty$. The other three $E_\infty$ classes $\tau^8$, $h_1$, and $h_2$ each only represent one element of $\Ext$ which we denote by the same symbol.

Hidden extensions can be separated into three categories:
\begin{enumerate}[label=\rm{(\arabic*)}]
    \item The source of the hidden extension is $\rho$-torsion free and the target is $\rho$-torsion free,

    \item the source is $\rho$-torsion free and the target is $\rho$-power-torsion,

    \item both source and target are $\rho$-power-torsion.
\end{enumerate}
The fourth case where the source is $\rho$-power-torsion and the target is $\rho$-torsion free cannot occur for multiplicative reasons. We study each of these three types of hidden extensions separately, one subsection at a time.

\subsection{Chart guide}\label{section: Chart guide for section on hidden extensions}

The relevant charts from \cite{Charts} vary for each subsection. For section \ref{section: Hidden extensions on the rho-free quotient} the reader should first open the chart \texttt{Rho-Bockstein E\_11 rho-free quotient} from the folder \texttt{Rho-Bockstein E\_r-pages rho-free quotient}. By Proposition \ref{E_11 is E_infty} this chart depicts the $\rho$-free quotient of the $E_\infty$-page of the $\rho$-Bockstein spectral sequence. So the hidden extensions considered in section \ref{section: Hidden extensions on the rho-free quotient} will be between elements on this chart. Once we are done with determining hidden extensions, the two charts in the folder \texttt{Cohomology of A(2) rho-free quotient} show all of these hidden extensions. They were separated into two charts for visual clarity.

In section \ref{section: Hidden extensions from rho-torsion free elements to rho-power-torsion elements} we should again consider the chart \texttt{Rho-Bockstein E\_11 rho-free quotient} for the sources of the hidden extensions. However, since the targets of the hidden extensions are $\rho$-power-torsion they will be located on charts in the folder \texttt{Rho Bockstein E\_infinity rho-torsion}.

For section \ref{section: hidden extensions on rho-torsion elements} we only consider $\rho$-power-torsion elements of coweights congruent to $1$ mod $8$, and hidden extensions by elements preserving the coweight. So the relevant chart is \texttt{Rho-Bockstein E\_infinity rho-torsion part with s-w == 1 mod 8} from the folder \texttt{Rho Bockstein E\_infinity rho-torsion}. Whenever we invoke the map $p$ from Lemma \ref{the map p} in section \ref{section: hidden extensions on rho-torsion elements} it is helpful to consider the chart \texttt{Rho-Bockstein E\_1 rho-free quotient} from the folder \texttt{Rho-Bockstein E\_r-pages rho-free quotient}. That is because it coincides as a chart with the cohomology of $\mathbb{C}$-motivic $\mathcal{A}(2)$ since $E_1$ is the same as the cohomology of $\mathbb{C}$-motivic $\mathcal{A}(2)$ adjoined $\rho$, but we don't plot $\rho$-multiples in the given chart.

\subsection{Hidden extensions on the \texorpdfstring{$\rho$}{rho}-free quotient}\label{section: Hidden extensions on the rho-free quotient}

By Corollary \ref{representation of rho-localized target} we know the multiplicative structure of the cohomology of $\mathcal{A}(2)$ after $\rho$-localization. This completely determines the multiplication on the $\rho$-free quotient.

\begin{Lemma}\label{multiplication on rho-free quotient is determined by rho-localization}
    The multiplication in the $\rho$-free quotient $\text{\rm Ext}_{\mathcal{A}(2)}(\mathbb{M}_2,\mathbb{M}_2)/(\text{\rm $\rho$-power-torsion})$ is uniquely determined by the multiplication in the $\rho$-localization $\text{\rm Ext}_{\mathcal{A}(2)}(\mathbb{M}_2,\mathbb{M}_2)[\rho^{-1}]$.
\end{Lemma}

\begin{proof}
    The kernel of the localization map $\ell: \text{\rm Ext}_{\mathcal{A}(2)}(\mathbb{M}_2,\mathbb{M}_2) \to \text{\rm Ext}_{\mathcal{A}(2)}(\mathbb{M}_2,\mathbb{M}_2)[\rho^{-1}]$ is the $\rho$-power-torsion subalgebra of the domain. Therefore, $\ell$ induces an injective map of algebras from the $\rho$-free quotient to the $\rho$-localization. The value of a product $x \cdot y$ in the $\rho$-free quotient is then given by $\ell^{-1}(\ell(x) \cdot \ell(y))$.
\end{proof}

In order to make use of Lemma \ref{multiplication on rho-free quotient is determined by rho-localization} we need to understand the localization map. To do so, we need to determine some hidden extensions on the $\rho$-free quotient of the $\rho$-Bockstein $E_\infty$-page. Using our knowledge of the $\rho$-localization of the cohomology of $\mathcal{A}(2)$, we can find all hidden extensions by $\tau^8$, $h_1$, and $h_2$ on the $\rho$-free quotient. There are two charts of the $\rho$-free quotient $\text{\rm Ext}_{\mathcal{A}(2)}(\mathbb{M}_2,\mathbb{M}_2)/(\text{$\rho$-power-torsion})$ with the aforementioned hidden extensions at \cite{Charts}. For a guide on how to read them see section \ref{section: charts}, and specifically section \ref{charts: cohomology of A(2) rho-free quotient}. Our task now is to prove the existence of these hidden extensions.

The following table \ref{tab:Hidden h_1-, h_2- and t^8-extensions on the rho-free quotient} describes all hidden extensions by $\tau^8$, $h_1$, and $h_2$ on the $\rho$-free quotient of the $\rho$-Bockstein $E_\infty$-page. We have inserted dashed horizontal lines to divide the table into different families. Hidden extensions in one family are closely connected, for example by multiplication with $h_1$ or $\tau^4 P$. Inside each family, the $h_1$-extensions are roughly located on a line of slope $1/2$, corresponding to the slope of $\tau^4 P$-multiplication. The columns of table \ref{tab:Hidden h_1-, h_2- and t^8-extensions on the rho-free quotient} describe:
\begin{enumerate}[label=\rm{(\arabic*)}]
    \item the name of the element supporting a hidden extension,
    \item the degree of the element supporting a hidden extension,
    \item which element we are multiplying with,
    \item the name of the target of the hidden extension,
    \item the periodicity of the hidden extension in terms of $E_\infty$-multiplication on the $\rho$-free quotient. Meaning if we multiply both the source and target on the $\rho$-free quotient of $E_\infty$ arbitrarily often by the given periodicity operator then we will see the same hidden extension.
\end{enumerate}
We choose $E_\infty$-multiplication in the periodicity column to avoid talking about hidden extensions other than $\tau^8$, $h_1$, and $h_2$ at this point. For example, by Corollary \ref{representation of rho-localized target} every element in the $\rho$-free quotient supports infinitely many $g$-multiplications. However, some hidden extensions in the given table are not $g$-periodic with respect to $E_\infty$-multiplication. For example the $\tau^8$-extension on $h_1^4$ is not $g$-periodic with respect to $E_\infty$-multiplication because the target of the hidden $\tau^8$-extension on $h_1^4 g^2$ is $\rho^8 \D^2 h_1^4$, which is not the $g^2$-multiple of $\rho^4 \tau^4 P$ on $E_\infty$.

Also, the periodicity column only concerns multiplication on the $\rho$-free quotient. We are not making any statement on possible $\rho$-power-torsion summands of any of these hidden extensions.

\begin{longtable}{llclc}
    \caption{Hidden $\tau^8$-, $h_1$-, and $h_2$-extensions on the $\rho$-free quotient of the $\rho$-Bockstein $E_\infty$-page
    \label{tab:Hidden h_1-, h_2- and t^8-extensions on the rho-free quotient}
    } \\
    \toprule
    Source & $(s,f,w)$ & Extension by & Target & $E_\infty$-periodicity\\
    \midrule \endfirsthead
    \caption[]{Hidden $\tau^8$-, $h_1$-, and $h_2$-extensions on the $\rho$-free quotient of the $\rho$-Bockstein $E_\infty$-page} \\
    \toprule
    Source & $(s,f,w)$ & Extension by & Target & $E_\infty$-periodicity\\
    \midrule \endhead
    \bottomrule \endfoot
        $h_1^4$ & $(4, 4, 4)$ & $\tau^8$ & $\rho^4 \tau^4 P$ & $P^2$, $\D^2$\\
        $h_1^5$ & $(5, 5, 5)$ & $\tau^8$ & $\rho^4 \tau^4 P h_1$ & $P^2$, $\D^2$\\
        $h_1^6$ & $(6, 6, 6)$ & $\tau^8$ & $\rho^4 \tau^4 P h_1^2$ & $P^2$, $\D^2$\\
        $h_1^7$ & $(7, 7, 7)$ & $\tau^8$ & $\rho^4 \tau^4 P h_1^3$ & $P^2$, $\D^2$\\
        $h_1^8$ & $(8, 8, 8)$ & $\tau^8$ & $\rho^8 P^2$ & $h_1$, $P^2$, $\D^2$\\
        $\tau^4 P h_1^3$ & $(11, 7, 3)$ & $h_1$ & $\rho^4 P^2$ & $P^2$, $\D^2$\\\hdashline[1pt/1pt]
        $u$ & $(11, 3, 7)$ & $\tau^8$ & $\rho \tau^6 a$ & $P^2$, $\D^2$\\
        $\tau^6 a$ & $(12, 3, 0)$ & $h_1$ & $\rho \tau^6 d$ & $P^2$, $\D^2$\\
        $h_1 u$ & $(12, 4, 8)$ & $\tau^8$ & $\rho^2 \tau^6 d$ & $P^2$, $\D^2$\\
        $h_1^2 u$ & $(13, 5, 9)$ & $\tau^8$ & $\rho^2 \tau^6 h_1 d$ & $P^2$, $\D^2$\\
        $\tau^6 h_1 d$ & $(15, 5, 3)$ & $h_1$ & $\rho \tau^5 h_0^2 e$ & $P^2$, $\D^2$\\
        $h_1^3 u$ & $(14, 6, 10)$ & $\tau^8$ & $\rho^3 \tau^5 h_0^2 e$ & $P^2$, $\D^2$\\
        $\tau^5 h_0^2 e$ & $(17, 6, 5)$ & $h_1$ & $\rho^2 \tau^2 P a$ & $P^2$, $\D^2$\\
        $h_1^4 u$ & $(15, 7, 11)$ & $\tau^8$ & $\rho^5 \tau^2 P a$ & $P^2$, $\D^2$\\
        $\tau^2 P a$ & $(20, 7, 8)$ & $h_1$ & $\rho \tau^2 P d$ & $P^2$, $\D^2$\\
        $h_1^5 u$ & $(16, 8, 12)$ & $\tau^8$ & $\rho^6 \tau^2 P d$ & $P^2$, $\D^2$\\
        $h_1^6 u$ & $(17, 9, 13)$ & $\tau^8$ & $\rho^6 \tau^2 P h_1 d$ & $P^2$, $\D^2$\\
        $\tau^2 P h_1 d$ & $(23, 9, 11)$ & $h_1$ & $\rho \tau P h_0^2 e$ & $P^2$, $\D^2$\\
        $h_1^7 u$ & $(18, 10, 14)$ & $\tau^8$ & $\rho^7 \tau P h_0^2 e$ & $P^2$, $\D^2$\\
        $\tau P h_0^2 e$ & $(25, 10, 13)$ & $h_1$ & $\rho P^2 u$ & $P^2$, $\D^2$\\
        $h_1^8 u$ & $(19, 11, 15)$ & $\tau^8$ & $\rho^8 P^2 u$ & $h_1$, $P^2$, $\D^2$\\\hdashline[1pt/1pt]
        $h_1^2 g$ & $(22, 6, 14)$ & $\tau^8$ & $\rho^2 \tau^4 a^2$ & $P^2$, $\D^2$\\
        $\tau^4 a^2$ & $(24, 6, 8)$ & $h_1$ & $\rho \tau^4 a d$ & $P^2$, $\D^2$\\
        $h_1^3 g$ & $(23, 7, 15)$ & $\tau^8$ & $\rho^3 \tau^4 a d$ & $P^2$, $\D^2$\\
        $\tau^4 a d$ & $(26, 7, 10)$ & $h_1$ & $\rho \tau^4 d^2$ & $P^2$, $\D^2$\\
        $h_1^4 g$ & $(24, 8, 16)$ & $\tau^8$ & $\rho^4 \tau^4 d^2$ & $P^2$, $\D^2$\\
        $h_1^5 g$ & $(25, 9, 17)$ & $\tau^8$ & $\rho^4 \tau^4 h_1 d^2$ & $P^2$, $\D^2$\\
        $\tau^4 h_1 d^2$ & $(29, 9, 13)$ & $h_1$ & $\rho^2 P a^2$ & $P^2$, $\D^2$\\
        $h_1^6 g$ & $(26, 10, 18)$ & $\tau^8$ & $\rho^6 P a^2$ & $P^2$, $\D^2$\\
        $P a^2$ & $(32, 10, 16)$ & $h_1$ & $\rho P a d$ & $P^2$, $\D^2$\\
        $h_1^7 g$ & $(27, 11, 19)$ & $\tau^8$ & $\rho^7 P a d$ & $P^2$, $\D^2$\\
        $P a d$ & $(34, 11, 18)$ & $h_1$ & $\rho P^2 g$ & $P^2$, $\D^2$\\
        $h_1^8 g$ & $(28, 12, 20)$ & $\tau^8$ & $\rho^8 P^2 g$ & $P^2$, $\D^2$\\\hdashline[1pt/1pt]
        $h_2 g$ & $(23, 5, 14)$ & $\tau^8$ & $\rho^2 \tau^5 \Dh$ & $\D^2$\\
        $\tau^5 \Dh$ & $(25, 5, 8)$ & $h_2$ & $\rho^2 \tau^4 n^2$ & $\D^2$\\
        $h_2^2 g$ & $(26, 6, 16)$ & $\tau^8$ & $\rho^4 \tau^4 n^2$ & $\D^2$\\\hdashline[1pt/1pt]
        $u g$ & $(31, 7, 19)$ & $\tau^8$ & $\rho \tau^6 a g$ & $P^2$, $\D^2$\\
        $\tau^6 a g$ & $(32, 7, 12)$ & $h_1$ & $\rho \tau^6 d g$ & $P^2$, $\D^2$\\
        $h_1 u g$ & $(32, 8, 20)$ & $\tau^8$ & $\rho^2 \tau^6 d g$ & $P^2$, $\D^2$\\
        $\tau^6 d g$ & $(34, 8, 14)$ & $h_1$ & $\rho \tau^2 a^3$ & $P^2$, $\D^2$\\
        $h_1^2 u g$ & $(33, 9, 21)$ & $\tau^8$ & $\rho^3 \tau^2 a^3$ & $P^2$, $\D^2$\\
        $\tau^2 a^3$ & $(36, 9, 16)$ & $h_1$ & $\rho \tau^2 P n^2$ & $P^2$, $\D^2$\\
        $h_1^3 u g$ & $(34, 10, 22)$ & $\tau^8$ & $\rho^4 \tau^2 P n^2$ & $P^2$, $\D^2$\\
        $\tau^2 P n^2$ & $(38, 10, 18)$ & $h_1$ & $\rho \tau^2 P a g$ & $P^2$, $\D^2$\\
        $h_1^4 u g$ & $(35, 11, 23)$ & $\tau^8$ & $\rho^5 \tau^2 P a g$ & $P^2$, $\D^2$\\
        $\tau^2 P a g$ & $(40, 11, 20)$ & $h_1$ & $\rho \tau^2 P d g$ & $P^2$, $\D^2$\\
        $h_1^5 u g$ & $(36, 12, 24)$ & $\tau^8$ & $\rho^6 \tau^2 P d g$ & $P^2$, $\D^2$\\
        $h_1^6 u g$ & $(37, 13, 25)$ & $\tau^8$ & $\rho^6 \tau^2 P h_1 d g$ & $P^2$, $\D^2$\\
        $\tau^2 P h_1 d g$ & $(43, 13, 23)$ & $\tau^8$ & $\rho \tau^6 P a^3$ & $P^2$, $\D^2$\\
        $\tau^2 P h_1 d g$ & $(43, 13, 23)$ & $h_1$ & $\rho \tau P h_0^2 e g$ & $P^2$, $\D^2$\\
        $\tau^6 P a^3$ & $(44, 13, 16)$ & $h_1$ & $\rho \tau^6 P^2 n^2$ & $P^2$, $\D^2$\\
        $h_1^7 u g$ & $(38, 14, 26)$ & $\tau^8$ & $\rho^7 \tau P h_0^2 e g$ & $P^2$, $\D^2$\\
        $\tau P h_0^2 e g$ & $(45, 14, 25)$ & $\tau^8$ & $\rho \tau^6 P^2 n^2$ & $P^2$, $\D^2$\\
        $\tau P h_0^2 e g$ & $(45, 14, 25)$ & $h_1$ & $\rho P^2 u g$ & $P^2$, $\D^2$\\
        $\tau^6 P^2 n^2$ & $(46, 14, 18)$ & $h_1$ & $\rho \tau^6 P^2 a g$ & $P^2$, $\D^2$\\
        $h_1^8 u g$ & $(39, 15, 27)$ & $\tau^8$ & $\rho^8 P^2 u g$ & $h_1$, $P^2$, $\D^2$\\\hdashline[1pt/1pt]
        $g^2$ & $(40, 8, 24)$ & $\tau^8$ & $\rho^8 \D^2$ & $h_1$, $P^2$, $g$, $\D^2$\\
        $h_2 g^2$ & $(43, 9, 26)$ & $\tau^8$ & $\rho^8 \D^2 h_2$ & $g$, $\D^2$\\
        $h_2^2 g^2$ & $(46, 10, 28)$ & $\tau^8$ & $\rho^8 \D^2 h_2^2$ & $g$, $\D^2$\\\hdashline[1pt/1pt]
        $u g^2$ & $(51, 11, 31)$ & $\tau^8$ & $\rho^8 \D^2 u$ & $h_1$, $P^2$, $g$, $\D^2$\\
\end{longtable}

\begin{Prop}
    Table \ref{tab:Hidden h_1-, h_2- and t^8-extensions on the rho-free quotient} describes all hidden $\tau^8$-, $h_1$-, and $h_2$-extensions on $\rho$-torsion free elements on the $\rho$-Bockstein $E_\infty$-page.
\end{Prop}

\begin{proof}
By Lemma \ref{multiplication on rho-free quotient is determined by rho-localization} all hidden extensions by $\tau^8$, $h_1$, and $h_2$ are detected by the $\rho$-localized cohomology of $\mathcal{A}(2)$. We provide some examples:

Consider $h_1^4$, which is $\tau^8$-torsion on $E_\infty$. By Corollary \ref{representation of rho-localized target}, $h_1^4$ is not $\tau^8$-torsion after $\rho$-localization. Hence, there must be a hidden extension $\tau^8 \cdot h_1^4 = z$ for some element $z$ in higher $\rho$-filtration. There are two possible targets, namely $\rho^4 \tau^4 P$ and $\rho^{16} g$. For degree reasons the $E_\infty$ element $g$ has to represent what we called $g$ in Corollary \ref{representation of rho-localized target}, so we must have $\tau^8 \cdot h_1^4 = \rho^4 \tau^4 P$. By $h_1$-multiplication this also implies $\tau^8 \cdot h_1^5 = \rho^4 \tau^4 P h_1$, $\tau^8 \cdot h_1^6 = \rho^4 \tau^4 P h_1^2$, and $\tau^8 \cdot h_1^7 = \rho^4 \tau^4 P h_1^3$. Then note that $\tau^8 \cdot h_1^8$ also is not zero on $E_\infty$, but $\rho^4 \tau^4 P h_1^4$ is. This implies that there is a hidden extension, and for degree reasons it must be of the form $h_1 \cdot \rho^4 \tau^4 P h_1^3 = \rho^8 P^2$. This extension then also implies $h_1 \cdot \tau^4 P h_1^3 = \rho^4 P^2$. That this family has the claimed periodicities follows from inspection of the $E_\infty$-page.

A different kind of pattern emerges with $u$. Again, $u$ is $\tau^8$-torsion on $E_\infty$ but not after $\rho$-localization. This forces the hidden extension $\tau^8 \cdot u = \rho \tau^6 a$. Then $h_1 u$ must also become $\rho$-torsion free and we must have the hidden extension $\tau^8 \cdot h_1 u = \rho^2 \tau^6 d$. Then $h_1 \cdot \rho \tau^6 a = \rho^2 \tau^6 d$ must be true, so $h_1 \cdot \tau^6 a = \rho \tau^6 d$. Much like before this leads to a periodic family of hidden extensions. We mention that this family can also be derived from the hidden extension $\tau^8 \cdot h_1^4 = \rho^4 \tau^4 P$. For example we have $\tau^6 a \cdot \tau^4 P = \tau^8 \cdot \tau^2 P a$ on $E_\infty$ and $\rho^4 \cdot \tau^8 \cdot \tau^2 P a$ is not trivial. This implies that $\tau^6 a$ must support at least four $h_1$-multiplications and the given pattern is the only possibility.

The families of hidden extensions starting with $h_1^2 g$ and $u g$, can be derived in similar fashion from the $\rho$-localization or by appealing to the hidden $\tau^8$ on $h_1^4$ as we just did. 

The family of hidden extensions starting with $h_2 g$ is also immediate from the $\rho$-localization and degree reasons.

A more interesting pattern emerges for $g^2$. Notice how $g^2$ and all of its $h_1$- and $h_2$-multiples are $\tau^8$-torsion on $E_\infty$\footnote{At least in the $\rho$-free quotient. Technically only $\rho^6 g^2$ is $\tau^8$-torsion, so to get the precise relations in $\Ext$ one should multiply everything by $\rho^6$.}. This means that there must be hidden $\tau^8$-extensions on all of these elements since $\tau^8 \cdot g^2$ is non-zero after $\rho$-localization. The only possibility is the relation $\tau^8 \cdot g^2 = \rho^8 \D^2$. This extends to $h_1$- and $h_2$-multiples, e.g. $\tau^8 \cdot h_1 g^2 = \rho^8 \D^2 h_1$, and $\tau^8 \cdot h_2 g^2 = \rho^8 \D^2 h_2$. Regarding the $P^2$-periodicity of the hidden extensions on $g^2$, note that $\rho^3 d^4 = \rho^3 (P^2 g^2 + \D^2 h_1^8)$ is the target of $d_3(P (\Dc d + \tau a d g))$. Therefore $P^2 g^2 = \D^2 h_1^8$ on the $\rho$-free quotient of the $\rho$-Bockstein $E_\infty$-page, and the $P^2$-multiple of the hidden $\tau^8$-extension on $g^2$ is a hidden $\tau^8$-extension from $\D^2 h_1^8$ to $\rho^8 \D^2 P^2$.

The hidden extension on $u g^2$ follows similarly. For the $P^2$-periodicity there is again a $\rho$-Bockstein differential $d_3(P \Dh h_1 d e)$ hitting $\rho^3 u d^4 = \rho^3 (P^2 u g^2 + \D^2 h_1^8 u)$, so that $P^2 u g^2 = \D^2 h_1^8 u$ on the $\rho$-free quotient of $E_\infty$.

Lastly, we should mention the interplay between the different periodicities for hidden extensions on the $\rho$-free quotient: First, consider $h_1^4 g^2$. By $h_1$-periodicity of the hidden $\tau^8$-extension on $g^2$, we have $\tau^8 \cdot h_1^4 g^2 = \rho^8 \D^2 h_1^4$. By $\D^2$-periodicity of the hidden $\tau^8$-extension on $h_1^4$, we have $\tau^8 \cdot \D^2 h_1^4 = \rho^4 \tau^4 \D^2 P$. These combine into $\tau^{16} \cdot h_1^4 g^2 = \tau^8 \cdot \rho^8 \D^2 h_1^4 = \rho^{12}\tau^4 \D^2 P$. As a second example, consider the hidden $\tau^8$-extension on $g^2$. This extension is $g$-periodic, so in particular we get $\tau^8 \cdot g^4 = \rho^8 \D^2 g^2$. It is also $\D^2$-periodic, so we get $\tau^8 \cdot \D^2 g^2 = \rho^8 \D^4$. Combining these we see $\tau^{16} \cdot g^4 = \rho^{16} \D^4$. This pattern generalizes to $(\tau^8 \cdot g^2)^n = (\rho^8 \D^2)^n$. In the charts, this is reflected by the longer and longer horizontal lines to the right of powers of $g^2$.
\end{proof}

\begin{Rem}
    Now we can give an example of how to use Lemma \ref{multiplication on rho-free quotient is determined by rho-localization} to multiply any two elements in the $\rho$-free quotient of the cohomology of $\mathcal{A}(2)$: Say we want to multiply $\tau^5 h_0^2 e$ and $\tau^4 d^2$. First of all, we need to be precise with our notation. There are no notational issues with $\tau^5 h_0^2 e$ as it represents exactly one element in $\Ext$ because there are no classes in higher $\rho$-filtration in its degree on $E_\infty$. On the other hand, the class $\tau^4 d^2$ represents multiple elements in $\Ext$ due to the presence of $\rho^{12} g^2$. Since the map from the $\rho$-free quotient to the $\rho$-localization is injective, we can define $\tau^4 d^2$ by its image in the $\rho$-localization, which we choose to be $\rho^{-4} \cdot \tau^8 \cdot h_1^4 \cdot g$.
    
    Now to multiply our two $\Ext$ elements, note that we have $\rho^3 \cdot \tau^5 h_0^2 e = \tau^8 \cdot h_1^3 \cdot u$ and $\rho^4 \cdot \tau^4 d^2 = \tau^8 \cdot h_1^4 \cdot g$. So the $\rho$-localization map sends the product $\tau^5 h_0^2 e \cdot \tau^4 d^2$ to $\rho^{-7} \cdot \tau^{16} \cdot h_1^7 \cdot u \cdot g$. The element $\tau^{16} \cdot h_1^7 \cdot u \cdot g$ is represented by $\rho^{8} \cdot \tau^6 P^2 n^2$ on $E_\infty$. So if we again abuse notation and identify $\tau^6 P^2 n^2$ with the appropriate $\Ext$ element, then we get $\tau^5 h_0^2 e \cdot \tau^4 d^2 = \rho \cdot \tau^6 P^2 n^2$.
\end{Rem}

\begin{Rem}\label{multiplication on rho-torsion free elements can have rho-torsion summands}
    As we saw in the previous remark, we need to be careful not to forget that $E_\infty$-page elements can represent multiple elements in the target of the $\rho$-Bockstein spectral sequence. This is especially important when multiplying by $g^2$: Many $g^2$-multiples are $\rho$-power-torsion on $E_\infty$. However, by Corollary \ref{representation of rho-localized target} we know that the $\rho$-localization is $g$-torsion free. In particular, multiplying any $\rho$-torsion free element with $g^2$ must yield another $\rho$-torsion free element. We will consider two examples:

    First, we have the $E_\infty$ class $\tau^8 g^2$ representing two elements in the cohomology of $\mathcal{A}(2)$ due to the presence of $\rho^8 \cdot \Delta^2$. One of these two elements is $\rho$-torsion free, and the other is $\rho$-power-torsion. By Corollary \ref{representation of rho-localized target}, $\tau^8 \cdot g^2$ must be the $\rho$-torsion free element. Conversely, since $\tau^4 g$ is $\rho$-power-torsion, the element $\tau^4 g \cdot \tau^4 g$ must be the $\rho$-power-torsion element. This leads to the $\Ext$ relation
    \[\tau^8 \cdot g^2 + \tau^4 g \cdot \tau^4 g = \rho^8 \cdot \Delta^2.\]

    As another example, consider $\tau^6 a g^2$. This $E_\infty$ class is $\rho^3$-torsion. Again, it represents two elements in the cohomology of $\mathcal{A}(2)$ due to the presence of $\rho^7 \cdot \D^2 u$. Because $\D^2 u$ is $\rho$-torsion free, one of these two elements is $\rho$-torsion free and the other is $\rho$-power-torsion. Since $\tau^6 a$ is $\rho$-torsion free, the $\Ext$ element $\tau^6 a \cdot g^2$ must be the $\rho$-torsion free one.
\end{Rem}

\subsection{Hidden extensions from \texorpdfstring{$\rho$}{rho}-torsion free elements to \texorpdfstring{$\rho$}{rho}-power-torsion elements}\label{section: Hidden extensions from rho-torsion free elements to rho-power-torsion elements}

As discussed in the introduction of this section, there can also be hidden extensions from $\rho$-torsion free elements to $\rho$-power-torsion elements on the $\rho$-Bockstein $E_\infty$-page. As an example, we will consider hidden $h_0$-extensions. Recall from the introduction that our choice of the $\Ext$ element $h_0$ is $\rho$-torsion. It follows that all targets of hidden $h_0$-extensions will also be $\rho$-torsion.

\begin{longtable}{lllc}
    \caption{Hidden $h_0$-extensions on $\rho$-torsion free elements on the $\rho$-Bockstein $E_\infty$-page
    \label{tab:Hidden h_0-extensions on rho-torsion free elements}
    } \\
    \toprule
    Source & $(s,f,w)$ & Target & $E_\infty$-Periodicity\\
    \midrule \endfirsthead
    \caption[]{Hidden $h_0$-extensions on $\rho$-torsion free elements on the $\rho$-Bockstein $E_\infty$-page} \\
    \toprule
    Source & $(s,f,w)$ & Target & $E_\infty$-Periodicity\\
    \midrule \endhead
    \bottomrule \endfoot
        $\tau^5 \Dh$ & $(25, 5, 8)$ & $\rho \tau^5 \Dh h_1$ & $P^2$, $\D^2$\\
        $\tau^4 n^2$ & $(30, 6, 12)$ & $\rho^2 \tau^4 a g$ & $P^2$, $g$, $\D^2$\\\hdashline[1pt/1pt]
        $\tau^2 P n^2$ & $(38, 10, 18)$ & $\rho^2 (\tau P \Dc + \tau^2 P a g)$ & $P^2$, $g$, $\D^2$\\
        $\tau^6 P^2 n^2$ & $(46, 14, 18)$ & $\rho^2 \tau^4 P^2 (\tau \Dc + \tau^2 a g)$ & $P^2$, $g$, $\D^2$\\
\end{longtable}

For the first two hidden $h_0$-extensions we mention that they are $P^2$-periodic but the $P^2$-multiples of the sources are $\rho$-power-torsion. Similarly, the $g$-multiple of $\tau^4 n^2$ is $\rho$-power-torsion. There is also a hidden $g$-extension on $\tau^5 \Dh h_1$ making the $h_0$-extension on $\tau^5 \Dh$ $g$-periodic with respect to multiplication in $\Ext$.

To deduce these hidden extensions, we will use an analogue of \cite[Remark 3.3]{BI} relying on the map $p$ from the following Lemma. For brevity, we will write $\Ext_{\mathcal{A}(2)^k}$ for $\Ext_{\mathcal{A}(2)^k}(\mathbb{M}_2^k, \mathbb{M}_2^k)$, where $k$ is $\mathbb{R}$ or $\mathbb{C}$.

\begin{Lemma}\label{the map p}
    There is a short exact sequence of $\text{\rm Ext}_{\mathcal{A}(2)^{\mathbb{R}}}$-modules
    \[0 \xrightarrow[]{} \text{\rm coker}(\rho) \xrightarrow[]{i} \text{\rm Ext}_{\mathcal{A}(2)^\mathbb{C}} \xrightarrow[]{p} \ker(\rho) \xrightarrow[]{} 0,\]
    where $i$ has degree $(0, 0, 0)$ and $p$ has degree $(0, 1, 1)$. Here, $\text{\rm coker}(\rho)$ and $\ker(\rho)$ are taken with respect to $\rho$-multiplication in $\text{\rm Ext}_{\mathcal{A}(2)^{\mathbb{R}}}$.
\end{Lemma}

\begin{proof}
Consider the short exact sequence of chain complexes
\[0 \to \mathcal{A}(2)^{\mathbb{R}} \xrightarrow[]{\rho} \mathcal{A}(2)^{\mathbb{R}} \to \mathcal{A}(2)^{\mathbb{R}}/\rho \to 0.\]
As observed in section \ref{section: notation}, $\mathcal{A}(2)^{\mathbb{R}}/\rho$ is isomorphic to $\mathcal{A}(2)^{\mathbb{C}}$. In particular, $\mathcal{A}(2)^{\mathbb{C}}$ has an $\mathcal{A}(2)^{\mathbb{R}}$-module structure where $\rho$ acts trivially. The short exact sequence of chain complexes induces a long exact sequence of $\Ext_{\mathcal{A}(2)^\mathbb{R}}$-modules
\[\dotsc \xrightarrow[]{} \Ext_{\mathcal{A}(2)^\mathbb{R}} \xrightarrow[]{\rho} \Ext_{\mathcal{A}(2)^\mathbb{R}} \xrightarrow[]{i} \Ext_{\mathcal{A}(2)^\mathbb{C}} \xrightarrow[]{p} \Ext_{\mathcal{A}(2)^\mathbb{R}} \xrightarrow[]{\rho} \Ext_{\mathcal{A}(2)^\mathbb{R}} \xrightarrow[]{} \dotsc\]
where the degree of $i$ is $(0, 0, 0)$ and the degree of $p$ is $(0, 1, 1)$. Splitting up this long exact sequence at $\Ext_{\mathcal{A}(2)^\mathbb{C}}$ yields the claim.
\end{proof}

Whenever we invoke Lemma \ref{the map p}, we are using the structure of $\Ext_{\mathcal{A}(2)^\mathbb{C}}$. It is worthwhile noting that we do have a chart of $\Ext_{\mathcal{A}(2)^\mathbb{C}}$. It is the chart depicting the $\rho$-free quotient of the $\rho$-Bockstein $E_1$-page. This chart agrees with $\Ext_{\mathcal{A}(2)^\mathbb{C}}$ because the $E_1$-page is just $\Ext_{\mathcal{A}(2)^\mathbb{C}}[\rho]$, and we do not plot $\rho$-multiples in the given chart.

Hidden extensions by $h_0$ are relatively easy to deduce because of the somewhat tautological observation that a hidden $h_0$-extension from $x$ to $y$ is the same as a hidden $x$-extension from $h_0$ to $y$, under some mild hypotheses.

\begin{Prop}\label{how to find hidden h_0-extensions}
    Let $x$, $y$ be non-trivial elements on the $\rho$-Bockstein $E_\infty$-page such that:
    \begin{itemize}
        \item $x$ is in $\rho$-filtration $0$,
        \item $h_0 \cdot x = 0$ on $E_\infty$, 
        \item the degree of $y$ is the degree of $x$ plus $(0, 1, 0)$,
        \item $y$ is in $\rho$-filtration greater than $0$ and in $\ker(\rho)$,
        \item there is no $h_0$-multiple or other $\rho$-multiple in the degree of $y$ in $\ker(\rho)$.
    \end{itemize}
    Then there exists a hidden $h_0$-extension from $x$ to $y$ if and only if $p(\tau \cdot x) = y$. Here, $p$ is the map from Lemma \ref{the map p}.
\end{Prop}

\begin{proof}
    By the assumptions, it makes sense to speak of a hidden $h_0$-extension from $x$ to $y$. This hidden extension is equivalent to $h_0 \cdot x = y$ in $\text{\rm Ext}_{\mathcal{A}(2)^{\mathbb{R}}}$. For degree reasons we must have $h_0 = p(\tau)$. So by linearity of $p$ we get $h_0 \cdot x = p(\tau) \cdot x = p(\tau \cdot x)$.
\end{proof}

\begin{Cor}\label{Corollary for finding hidden h_0-extensions}
    Under the assumptions of Proposition \ref{how to find hidden h_0-extensions}, if $\tau \cdot x \in \text{\rm Ext}_{\mathcal{A}(2)^{\mathbb{C}}}$ is the only non-trivial element in its degree, then there is a hidden $h_0$-extension from $x$ to $y$.
\end{Cor}

\begin{proof}
    The map $p$ is surjective which implies that the degree where $p(\tau \cdot x)$ lives is at most $1$-dimensional over $\mathbb{F}_2$. Since $y$ is assumed non-trivial, $p$ is an isomorphism and we must have $p(\tau \cdot x) = y$.
\end{proof}

\begin{Prop}
    Table \ref{tab:Hidden h_0-extensions on rho-torsion free elements} describes all hidden $h_0$-extensions on $\rho$-torsion free elements on the $\rho$-Bockstein $E_\infty$-page.
\end{Prop}

\begin{proof}
    All elements not listed in table \ref{tab:Hidden h_0-extensions on rho-torsion free elements} either already have $h_0$-multiples or they cannot support hidden extensions for degree or multiplicative reasons.

    Every claimed hidden $h_0$-extension follows from Corollary \ref{Corollary for finding hidden h_0-extensions}.
\end{proof}

\subsection{Hidden extensions on \texorpdfstring{$\rho$}{rho}-power-torsion elements}\label{section: hidden extensions on rho-torsion elements}

The third and last possible type of hidden extension is where both the source and the target are $\rho$-power-torsion. We provide some examples of such hidden extensions by $\tau^8$, $h_0$, and $h_1$.

At \cite{Charts} eight charts are provided which combine into the $\rho$-power-torsion subalgebra of the $\rho$-Bockstein $E_\infty$-page. Each chart consists of elements whose coweight $s - w$ is congruent to some fixed value modulo $8$. For a guide on how to read the charts see section \ref{section: charts}, and specifically section \ref{charts: rho-Bockstein E_infinity-page rho-power-torsion}. In this section, we will simply pick one coweight (or equivalently one chart) and illustrate how to compute hidden extensions there. For the rest of this section we will only consider elements of coweight $1$ mod $8$. Note that all of $\tau^8$, $h_0$, and $h_1$ have coweight $0$ mod $8$, meaning that source and target of every hidden extension we consider are in the same coweight, and hence on the same chart.

There are two general methods we will use to deduce hidden extensions: The first method involves using hidden extensions from the $\rho$-free quotient together with the multiplication on $E_{\infty}$. The second method uses the map $p$ from Lemma \ref{the map p}.

What follows is a table describing all hidden $\tau^8$- and $h_1$-extensions between $\rho$-power-torsion elements in coweight $1$ mod $8$. We have inserted dashed horizontal lines to divide the table into different families. Hidden extensions in the same family are closely connected, for example by multiplication with $h_1$ or $\tau^8$ or $\tau^4 P$. Inside each family, the $h_1$-extensions are roughly located on a line of slope $1/2$, corresponding to the slope of $\tau^4 P$-multiplication. The columns of table \ref{tab:Hidden h_1- and t^8-extensions on rho-power-torsion elements in coweight 1 mod 8} describe:
\begin{enumerate}[label=\rm{(\arabic*)}]
    \item the name of the element supporting a hidden extension,
    \item the degree of the element supporting a hidden extension,
    \item which element we are multiplying with,
    \item the name of the target of the hidden extension,
    \item the periodicity of the hidden extension in terms of $E_\infty$-multiplication.
\end{enumerate}
We choose $E_\infty$-multiplication in the periodicity column to avoid talking about hidden extensions other than $\tau^8$ and $h_1$. For example $\tau^6 P h_1 e$ supports a hidden $g$-extension to $\rho \tau^3 P \Dh d$, so the hidden $h_1$-extension on $\tau^6 P h_1 e$ is $g$-periodic in Ext while not being $g$-periodic with respect to $E_\infty$-multiplication.

\begin{longtable}{llclc}
    \caption{Hidden $\tau^8$- and $h_1$-extensions on $\rho$-power-torsion elements in coweight $1$ mod $8$ on the $\rho$-Bockstein $E_\infty$-page
    \label{tab:Hidden h_1- and t^8-extensions on rho-power-torsion elements in coweight 1 mod 8}
    } \\
    \toprule
    Source & $(s,f,w)$ & Extension by & Target & $E_\infty$-Periodicity\\
    \midrule \endfirsthead
    \caption[]{Hidden $\tau^8$- and $h_1$-extensions on $\rho$-power-torsion elements in coweight $1$ mod $8$ on the $\rho$-Bockstein $E_\infty$-page} \\
    \toprule
    Source & $(s,f,w)$ & Extension by & Target & $E_\infty$-Periodicity\\
    \midrule \endhead
    \bottomrule \endfoot
        $\tau^6 h_1 c$ & $(9, 4, 0)$ & $h_1$ & $\rho \tau^4 P h_2$ & $P^2$, $\D^2$\\
        $\tau^2 P h_1 c$ & $(17, 8, 8)$ & $h_1$ & $\rho P^2 h_2$ & $P^2$, $\D^2$\\\hdashline[1pt/1pt]
        $\tau^2 h_1 e$ & $(18, 5, 9)$ & $h_1$ & $\rho \tau h_0^2 g$ & $P^2$, $g$, $\D^2$\\
        $\tau h_0^2 g$ & $(20, 6, 11)$ & $h_1$ & $\rho c d$ & $P^2$, $g$, $\D^2$\\
        $c d$ & $(22, 7, 13)$ & $\tau^8$ & $\rho \tau^6 P n$ & $P^2$, $g$, $\D^2$\\
        $\tau^6 P n$ & $(23, 7, 6)$ & $h_1$ & $\rho \tau^6 P e$ & $P^2$, $g$, $\D^2$\\
        $h_1 c d$ & $(23, 8, 14)$ & $\tau^8$ & $\rho^2 \tau^6 P e$ & $P^2$, $g$, $\D^2$\\
        $h_1^2 c d$ & $(24, 9, 15)$ & $\tau^8$ & $\rho^2 \tau^6 P h_1 e$ & $P^2$, $\D^2$\\
        $\tau^6 P h_1 e$ & $(26, 9, 9)$ & $h_1$ & $\rho \tau^5 P h_0^2 g$ & $P^2$, $\D^2$\\
        $h_1^3 c d$ & $(25, 10, 16)$ & $\tau^8$ & $\rho^3 \tau^5 P h_0^2 g$ & $P^2$, $\D^2$\\
        $\tau^5 P h_0^2 g$ & $(28, 10, 11)$ & $h_1$ & $\rho^2 \tau^2 P^2 n$ & $P^2$, $\D^2$\\
        $h_1^4 c d$ & $(26, 11, 17)$ & $\tau^8$ & $\rho^5 \tau^2 P^2 n$ & $P^2$, $\D^2$\\
        $\tau^2 P^2 n$ & $(31, 11, 14)$ & $h_1$ & $\rho \tau^2 P^2 e$ & $P^2$, $g$, $\D^2$\\\hdashline[1pt/1pt]
        $P^2 h_2 g$ & $(39, 13, 22)$ & $\tau^8$ & $\rho^2 \tau^5 P^2 \Dh$ & $P^2$, $g$, $\D^2$\\\hdashline[1pt/1pt]
        $\tau^2 h_1 e g$ & $(38, 9, 21)$ & $\tau^8$ & $\rho \tau^7 \Dh d$ & $P^2$, $g$, $\D^2$\\
        $\tau^7 \Dh d$ & $(39, 9, 14)$ & $h_1$ & $\rho \tau^6 a^2 e$ & $P^2$, $g$, $\D^2$\\
        $\tau h_0^2 g^2$ & $(40, 10, 23)$ & $\tau^8$ & $\rho \tau^6 a^2 e$ & $P^2$, $g$, $\D^2$\\
        $\tau^6 a^2 e$ & $(41, 10, 16)$ & $h_1$ & $\rho \tau^6 P n g$ & $P^2$, $g$, $\D^2$\\
        $\tau^6 P e g$ & $(45, 12, 20)$ & $h_1$ & $\rho \tau^3 P \Dh d$ & $P^2$, $g$, $\D^2$\\
        $\tau^3 P \Dh d$ & $(47, 13, 22)$ & $h_1$ & $\rho \tau^2 P a^2 e$ & $P^2$, $g$, $\D^2$\\
        $\tau^2 P a^2 e$ & $(49, 14, 24)$ & $h_1$ & $\rho \tau^2 P^2 n g$ & $P^2$, $g$, $\D^2$\\\hdashline[1pt/1pt]
        $\tau h_1 g^2$ & $(41, 9, 24)$ & $\tau^8$ & $\rho^4 \tau^5 \Dh g$ & $g$, $\D^2$\\\hdashline[1pt/1pt]
        $\tau^5 \Dh g$ & $(45, 9, 20)$ & $h_1$ & $\rho \tau^4 a n g$ & $P^2$, $g$, $\D^2$\\
        $\tau^4 a n g$ & $(47, 10, 22)$ & $h_1$ & $\rho \tau^4 a e g$ & $P^2$, $g$, $\D^2$\\
        $\tau^4 a e g$ & $(49, 11, 24)$ & $h_1$ & $\rho \tau^4 d e g$ & $P^2$, $g$, $\D^2$\\
        $\tau^4 d e g$ & $(51, 12, 26)$ & $h_1$ & $\rho \tau P \Dh g$ & $P^2$, $g$, $\D^2$\\
        $\tau P \Dh g$ & $(53, 13, 28)$ & $h_1$ & $\rho P a n g$ & $P^2$, $g$, $\D^2$\\
        $P a n g$ & $(55, 14, 30)$ & $h_1$ & $\rho P a e g$ & $P^2$, $g$, $\D^2$\\
        $P a e g$ & $(57, 15, 32)$ & $h_1$ & $\rho d^3 e$ & $P^2$, $g$, $\D^2$\\
        $h_1 d^3 e$ & $(60, 17, 35)$ & $\tau^8$ & $\rho \tau^5 P^2 \Dh g$ & $P^2$, $g$, $\D^2$\\
        $h_1^2 d^3 e$ & $(61, 18, 36)$ & $\tau^8$ & $\rho^2 \tau^4 P^2 a n g$ & $P^2$, $g$, $\D^2$\\
\end{longtable}

\begin{Prop}
    Table \ref{tab:Hidden h_1- and t^8-extensions on rho-power-torsion elements in coweight 1 mod 8} describes all hidden $\tau^8$- and $h_1$-extensions on $\rho$-power-torsion elements on the $\rho$-Bockstein $E_\infty$-page in coweight $1$ mod $8$.
\end{Prop}

\begin{proof}
    All elements not listed in table \ref{tab:Hidden h_1- and t^8-extensions on rho-power-torsion elements in coweight 1 mod 8} either already have $\tau^8$- and $h_1$-multiples or they cannot support hidden extensions for degree or multiplicative reasons.
    
    To deduce the hidden $h_1$-extension on $\tau^6 h_1 c$ we can use the map $p$ from Lemma \ref{the map p}. Multiplying both source and target of the claimed hidden extension by $\rho$ we get a possible hidden extension in $\text{ker}(\rho)$. For degree reasons, we must have $p(\tau^7 c) = \rho \tau^6 h_1 c$ and $p(\tau^7 h_1 c) = \rho^2 \tau^4 P h_2$. Certainly the $h_1$-multiple of $\tau^7 c$ in $\Ext_{\mathcal{A}(2)^{\mathbb{C}}}$ is $\tau^7 h_1 c$, so $p(\tau^7 h_1 c) = h_1 \cdot p(\tau^7 c)$. Therefore there must be a hidden $h_1$-extension from $\rho \tau^6 h_1 c$ to $\rho^2 \tau^4 P h_2$. This also implies our claimed hidden $h_1$-extension from $\tau^6 h_1 c$ to $\rho \tau^4 P h_2$. The hidden $h_1$-extension on $\tau^2 P h_1 c$ immediately follows by multiplying the one on $\tau^6 h_1 c$ by $\tau^4 P$.

    Similar arguments can be used to find the hidden $\tau^8$-extensions on $P^2 h_2 g$ and $\tau h_1 g^2$.

    For the family of hidden extensions starting with $\tau^2 h_1 e$, recall the hidden extension $\tau^8 \cdot h_1^4 = \rho^4 \tau^4 P$ from the $\rho$-free quotient. Then note that $\tau^2 h_1 e$ has a non-trivial $\rho^4 \tau^4 P$-multiple, namely $\rho^4 \tau^6 P h_1 e$. This implies that $\tau^2 h_1 e$ must support at least four $h_1$-multiplications, and all of those must support at least one $\tau^8$-multiplication. The first few of the hidden extensions in this family are now immediate. For the remaining ones, repeat the same argument starting with $\tau^6 P h_1 e$ instead of $\tau^2 h_1 e$.

    For the family of hidden extensions starting with $\tau^2 e g$, we first note that the same argument involving $\rho^4 \tau^4 P$-multiplication does not work because $\tau^6 P e g$ is $\rho^3$-torsion. But from the previous paragraph we can deduce the hidden extension $h_1^3 \cdot \tau^2 e = \rho^2 c d$. Note that $\rho^2 c d g$ is not zero, so $\tau^2 e g$ must support at least three $h_1$-multiplications and all of those must support at least one $\tau^8$-multiplication. Much like before, the actual targets of these hidden extensions are uniquely determined for degree reasons. This determines the first few hidden extensions in this family, and the rest are consequences of $\tau^4 P$-multiplication.

    The hidden extensions in the last family starting with $\tau^5 \Dh g = \tau^5 \Dh \cdot g$ are most easily proven by referring to hidden $g$-extensions or other coweights. For example using the map $p$ from Lemma \ref{the map p} one can easily deduce a hidden $g$-extension from $\tau^5 \Dh h_1^2$ to $\rho^2 \tau^4 a e g$. That implies the first two $h_1$-extensions in this family. Then the hidden $h_1$ from $\tau^4 a e g$ to $\rho \tau^4 d e g$ follows from coweight $3$ mod $8$, where $h_1 \cdot a = \rho d$ is an immediate consequence of $\tau^8 \cdot h_1^4 = \rho^4 \tau^4 P$. Next, to get the hidden $h_1$ on $\tau^4 d e g$ we can consider coweight $4$ mod $8$ where $\tau^4 g \cdot h_1^2 = \rho^2 a^2$ is immediate from $\rho^4 \tau^4 P$-multiplication since $\rho^4 \tau^4 P \cdot \tau^4 g = \rho^4 \tau^8 d^2$. Using some $\mathbb{C}$-motivic relations we can deduce $\tau^4 d e g \cdot h_1^2 = \rho^2 P a n g$. All other hidden extensions in this family are now consequences of $\tau^4 P$-multiplication.
\end{proof}

\begin{Rem}
    The reader might wonder if the map $p$ from Lemma \ref{the map p} corresponds to "the $\rho$-Bockstein spectral sequence differential divided by $\rho$". After all, in the above we saw $p(\tau^7 c) = \rho \tau^6 h_1 c$ and $p(\tau^7 h_1 c) = \rho^2 \tau^4 P h_2$, and in the $\rho$-Bockstein spectral sequence we have $d_2(\tau^7 c) = \tau^4 \cdot \tau c \cdot d_2(\tau^2) = \rho^2 \tau^6 h_1 c$ and $d_3(\tau^7 h_1 c) = \tau^4 \cdot d_3(\tau^3 h_1 c) = \rho^3 \tau^4 P h_2$. In general, we have to be slightly careful when comparing $p$ and the Bockstein differentials $d_r$ because the former is defined on $\Ext$ and the latter is not. So we should rather speak of $p$ being {\em detected} by $d_r$. Roughly speaking, this is then true because $p$ is by definition the cobar differential divided by $\rho$, and $d_r$ is the cobar differential modulo higher $\rho$-filtration.

    When using this comparison between $p$ and $d_r$, we need to remember that a Bockstein spectral sequence element can detect multiple elements in $\Ext$: As an example consider $\tau^8 g^2$. We have the differential $d_6(\tau^{10} n g) = \rho^6 \tau^8 g^2$. However, as we saw when talking about hidden extensions on the $\rho$-free quotient, the element $\tau^8 \cdot g^2$ is $\rho$-torsion free in $\Ext$ via the relation $\tau^8 \cdot \rho^6 g^2 = \rho^{14} \D^2$. As $p$ maps onto $\text{ker}(\rho)$, we must have $p(\tau^{10} n g) = \rho^5 \cdot \tau^8 \cdot g^2 + \rho^{13} \cdot \D^2$.
\end{Rem}

Another interesting observation is the following: In the $\rho$-free quotient we saw that there is a hidden $\tau^8$-extension from $g^2$ to $\rho^8 \D^2$. In the $\rho$-power-torsion part we do not see any connection between $g^2$- and $\D^2$-multiples. That is because the hidden $\tau^8$-extension in $\Ext$ goes from $\rho^6 g^2$ to $\rho^{14} \D^2$, but $\rho^{14}$-multiplication is trivial in the $\rho$-power-torsion part. Therefore, the patterns of hidden extensions for $g^2$- and $\D^2$-multiples of $\rho$-power-torsion elements do not have to look the same. For example the pattern of hidden $h_1$- and $\tau^8$-extensions on $\tau^2 e g^2$ resembles the pattern on $\tau^2 e g$, whereas the pattern on $\tau^2 \D^2 e$ resembles the one on $\tau^2 e$.

The next table collects all hidden $h_0$-extensions on $\rho$-power-torsion elements in the $\rho$-Bockstein spectral sequence in coweight $1$ mod $8$. The columns are the same as those in table \ref{tab:Hidden h_1- and t^8-extensions on rho-power-torsion elements in coweight 1 mod 8}. Again, we have chosen to separate families of hidden extensions by dashed horizontal lines. Hidden extensions inside each family are typically connected via $\tau^4 P$- or $\tau^5 c$-multiplication. They are roughly located on a line of slope $1/2$, corresponding to the slope of $\tau^4 P$-multiplication and almost the slope of $\tau^5 c$-multiplication. We note that $\tau^5 \Dh$ is also in coweight $1$ mod $8$ and supports a hidden $h_0$-extension, but it is $\rho$-torsion free so it should not be part of this table.

\begin{longtable}{lllc}
    \caption{Hidden $h_0$-extensions on $\rho$-power-torsion elements in coweight $1$ mod $8$ on the $\rho$-Bockstein $E_\infty$-page
    \label{tab:Hidden h_0-extensions on rho-power-torsion elements in coweight 1 mod 8}
    } \\
    \toprule
    Source & $(s,f,w)$ & Target & $E_\infty$-Periodicity\\
    \midrule \endfirsthead
    \caption[]{Hidden $h_0$-extensions on $\rho$-power-torsion elements in coweight $1$ mod $8$ on the $\rho$-Bockstein $E_\infty$-page} \\
    \toprule
    Source & $(s,f,w)$ & Target & $E_\infty$-Periodicity\\
    \midrule \endhead
    \bottomrule \endfoot
        $\tau h_1$ & $(1, 1, 0)$ & $\rho \tau h_1^2$ & $P^2$, $\D^2$\\
        $\tau^6 h_1 c$ & $(9, 4, 0)$ & $\rho^2 \tau^4 P h_2$ & $P^2$, $\D^2$\\
        $\tau^5 P h_1$ & $(9, 5, 0)$ & $\rho \tau^5 P h_1^2$ & $P^2$, $\D^2$\\
        $\tau^2 P h_1 c$ & $(17, 8, 8)$ & $\rho^2 P^2 h_2$ & $P^2$, $\D^2$\\\hdashline[1pt/1pt]
        $\tau^3 h_1 d$ & $(15, 5, 6)$ & $\rho^5 \tau h_0^2 g$ & $P^2$, $\D^2$\\
        $\tau^2 h_0^2 e$ & $(17, 6, 8)$ & $\rho^5 c d$ & $P^2$, $\D^2$\\
        $\tau^{10} h_0^2 e$ & $(17, 6, 0)$ & $\rho^6 \tau^6 P n$ & $P^2$, $\D^2$\\
        $\tau^7 P h_1 d$ & $(23, 9, 6)$ & $\rho^5 \tau^5 P h_0^2 g$ & $P^2$, $\D^2$\\
        $\tau^6 P h_0^2 e$ & $(25, 10, 8)$ & $\rho^6 \tau^2 P^2 n$ & $P^2$, $\D^2$\\\hdashline[1pt/1pt]
        $\tau^2 \Dc h_1$ & $(33, 8, 16)$ & $\rho \tau^2 \Dc h_1^2$ & $P^2$, $\D^2$\\
        $\tau P \Dh$ & $(33, 9, 16)$ & $\rho \tau P \Dh h_1$ & $P^2$, $\D^2$\\
        $\tau^6 P \Dc h_1$ & $(41, 12, 16)$ & $\rho \tau^6 P \Dc h_1^2$ & $P^2$, $\D^2$\\
        $\tau^5 P^2 \Dh$ & $(41, 13, 16)$ & $\rho \tau^5 P^2 \Dh h_1$ & $P^2$, $\D^2$\\\hdashline[1pt/1pt]
        $\tau^2 e g$ & $(37, 8, 20)$ & $\rho \tau^2 h_1 e g$ & $P^2$, $g$, $\D^2$\\
        $\tau^{10} e g$ & $(37, 8, 12)$ & $\rho^2 \tau^7 \Dh d$ & $P^2$, $g$, $\D^2$\\
        $\tau^7 \Dh d$ & $(39, 9, 14)$ & $\rho^2 \tau^6 a^2 e$ & $P^2$, $g$, $\D^2$\\
        $\tau^6 a^2 e$ & $(41, 10, 16)$ & $\rho^2 \tau^6 P n g$ & $P^2$, $g$, $\D^2$\\
        $\tau^6 P n g$ & $(43, 11, 18)$ & $\rho^2 \tau^6 P e g$ & $P^2$, $g$, $\D^2$\\
        $\tau^6 P e g$ & $(45, 12, 20)$ & $\rho^2 \tau^3 P \Dh d$ & $P^2$, $g$, $\D^2$\\
        $\tau^3 P \Dh d$ & $(47, 13, 22)$ & $\rho^2 \tau^2 P a^2 e$ & $P^2$, $g$, $\D^2$\\
        $\tau^2 P a^2 e$ & $(49, 14, 24)$ & $\rho^2 \tau^2 P^2 n g$ & $P^2$, $g$, $\D^2$\\
        $\tau^2 P^2 n g$ & $(51, 15, 26)$ & $\rho^2 \tau^2 P^2 e g$ & $P^2$, $g$, $\D^2$\\\hdashline[1pt/1pt]
        $\tau^5 \Dh g$ & $(45, 9, 20)$ & $\rho^2 \tau^4 a n g$ & $P^2$, $g$, $\D^2$\\
        $\tau^4 a n g$ & $(47, 10, 22)$ & $\rho^2 \tau^4 a e g$ & $P^2$, $g$, $\D^2$\\
        $\tau^4 a e g$ & $(49, 11, 24)$ & $\rho^2 \tau^4 d e g$ & $P^2$, $g$, $\D^2$\\
        $\tau^4 d e g$ & $(51, 12, 26)$ & $\rho^2 \tau P \Dh g$ & $P^2$, $g$, $\D^2$\\
        $\tau P \Dh g$ & $(53, 13, 28)$ & $\rho^2 P a n g$ & $P^2$, $g$, $\D^2$\\
        $P a n g$ & $(55, 14, 30)$ & $\rho^2 P a e g$ & $P^2$, $g$, $\D^2$\\
        $\tau^8 P a e g$ & $(57, 15, 24)$ & $\rho^2 \tau^8 d^3 e$ & $P^2$, $g$, $\D^2$\\
        $\tau^8 d^3 e$ & $(59, 16, 26)$ & $\rho^2 \tau^5 P^2 \Dh g$ & $P^2$, $g$, $\D^2$\\
\end{longtable}

\begin{Prop}
    Table \ref{tab:Hidden h_0-extensions on rho-power-torsion elements in coweight 1 mod 8} describes all hidden $h_0$-extensions on $\rho$-power-torsion elements on the $\rho$-Bockstein $E_\infty$-page in coweight $1$ mod $8$.
\end{Prop}

\begin{proof}
    All elements not listed in table \ref{tab:Hidden h_0-extensions on rho-power-torsion elements in coweight 1 mod 8} either already have $h_0$-multiples or they cannot support hidden extensions for degree or multiplicative reasons.

    Every claimed hidden $h_0$-extension follows from Corollary \ref{Corollary for finding hidden h_0-extensions}.
\end{proof}

\section{An Adams spectral sequence?}\label{section: An Adams spectral sequence?}

In the classical setting $\Ext_{\mathcal{A}(2)^{\cl}}(\mathbb{F}_2, \mathbb{F}_2)$ is the $E_2$-page of the Adams spectral sequence computing the homotopy groups of the topological modular forms spectrum $\tmf$. Similarly, in the $\mathbb{C}$-motivic setting $\Ext_{\mathcal{A}(2)^\mathbb{C}}(\mathbb{M}_2^\mathbb{C}, \mathbb{M}_2^\mathbb{C})$ is the $E_2$-page of the $\mathbb{C}$-motivic Adams spectral sequence computing the homotopy groups of the $\mathbb{C}$-motivic modular forms spectrum $\mmf^{\,\,\mathbb{C}}$. Speculatively, one could assume that there is an $\mathbb{R}$-motivic Adams spectral sequence with $E_2$-page $\Ext_{\mathcal{A}(2)}(\mathbb{M}_2, \mathbb{M}_2)$, associated to a speculative $\mathbb{R}$-motivic modular forms spectrum $\mmf^{\,\,\mathbb{R}}$. Under this assumption, we will describe the $d_2$-differentials of that Adams spectral sequence on all indecomposables.

We will not attempt to compute the entire Adams spectral sequence. The issue is that we do not have all hidden extensions that are necessary to pass from the $\rho$-Bockstein $E_\infty$-page to the cohomology of $\mathcal{A}(2)$, as there are too many of them. While we do have decompositions for all elements of $\Ext_{\mathcal{A}(2)}(\mathbb{M}_2, \mathbb{M}_2)$, the Leibniz rule for Adams differentials $d_r(x \cdot y) = x \cdot d_r(y) + d_r(x) \cdot y$ may involve hidden extensions, i.e. both $x \cdot d_r(y)$ and $d_r(x) \cdot y$ may be zero on the Bockstein $E_\infty$-page but non-zero in \Ext. Hence, knowledge of the differentials on indecomposables does not immediately translate to knowledge of all differentials.

A notational issue that we also need to consider is the passage from the $\rho$-Bockstein $E_\infty$-page to $\Ext$. If a given class on the $E_\infty$-page has other classes in the same tridegree that are of higher $\rho$-filtration, then that class represents multiple elements in $\Ext$. Oftentimes we can distinguish these elements by their multiplicative decompositions. For example, we already saw in Remark \ref{multiplication on rho-torsion free elements can have rho-torsion summands} that the $E_\infty$ class $\tau^8 g^2$ represents two elements in $\Ext$ which can be distinguished by their decompositions $\tau^8 \cdot g^2$ and $\tau^4 g \cdot \tau^4 g$. We will use such multiplicative decompositions whenever there are multiple elements represented by a single class on $E_\infty$. In particular, we will make the following notational convention in this section.

\begin{Not}
    In this section, the symbol $\cdot$ will denote multiplication in $\text{\rm Ext}_{\mathcal{A}(2)^{\mathbb{R}}}$ unless stated otherwise. I.e. it denotes the multiplication coming from the $\rho$-Bockstein $E_\infty$-page after resolving possible hidden extensions.
\end{Not}

Using multiplicative decompositions to distinguish between different elements in $\Ext$ that are represented by the same class on $E_\infty$ obviously does not work for indecomposables. 

Let us first consider $\rho$-torsion free indecomposables. Since the map from the $\rho$-free quotient to the $\rho$-localization is injective, we can define $\rho$-free classes (up to $\rho$-power-torsion summands) by their images in the $\rho$-localization. We will assign names as they appear on our $\rho$-Bockstein $E_\infty$-page. For example, we let $\tau^4 P$ be the element whose image in the $\rho$-localization is $\rho^{-4} \cdot \tau^8 \cdot h_1^4$, or we let $\tau^6 a$ be the element which localizes to $\rho^{-1} \cdot \tau^8 \cdot u$. This kind of assignment works as long as there are no $\rho$-power-torsion elements in higher $\rho$-filtration. That almost never occurs for the elements we care about. Exceptions are $\tau^2 P a$ which also has $\rho^6 a d$ in higher $\rho$-filtration, and $P a^2$ which has $\rho^5 \tau h_0^2 e g$. For the first exception the choice of element representing $\tau^2 P a$ makes no difference. While $\rho^6 a d = \rho^6 \cdot a \cdot d$ contains $a$ which supports an Adams $d_2$, the target is $\rho^3$-torsion so $\rho^6 a d$ is a $d_2$-cycle. For $P a^2$ we have to be more careful a priori. The element $\rho^5 \tau h_0^2 e g$ has the decomposition $\rho^5 \cdot c \cdot a \cdot e$, and as we will see $d_2(e) = h_1^2 d$ supports a non-trivial $\rho^5$-multiplication. However, $h_1^2 d \cdot c$ is $\rho^3$-torsion, so that $\rho^5 \tau h_0^2 e g$ is again a $d_2$-cycle. So the existence of $\rho$-power-torsion elements in higher $\rho$-filtration is also immaterial for $P a^2$, at least for the purpose of computing Adams $d_2$-differentials.

When an indecomposable class on $E_\infty$ is $\rho$-power-torsion we can typically abuse notation and use the same letters to denote the corresponding element of $\Ext$ whose $\rho$-power-torsion exactly matches that of the class on $E_\infty$. For example, $\tau h_1$ has the class $\rho^2 h_2$ in higher $\rho$-filtration, but only one of the elements it represents is $\rho^2$-torsion because $h_2$ is $\rho$-torsion free. So we let $\tau h_1$ denote the unique $\rho^2$-torsion element in $\Ext$ represented by the element on $E_\infty$ with the same name. Again, there are two indecomposables where this approach does not define them uniquely, namely $\tau P h_1$ with $\rho^2 P h_2$ in higher $\rho$-filtration and $\tau P \Dh$ with $\rho \tau^2 \Dc h_1^2$. Much like for the $\rho$-torsion free exceptions, one can argue that the choices of representatives here make no difference.

\subsection{Chart guide}\label{section: Chart guide for Adams differentials}

For section \ref{section: Adams d_2-differentials} the relevant charts from \cite{Charts} are in the two folders \texttt{Cohomology of A(2) rho-free quotient} and \texttt{Rho Bockstein E\_infinity rho-torsion}. The first is relevant for differentials whose sources are $\rho$-torsion free, the second for differentials where either the target or the source is $\rho$-power-torsion. Note that the Adams $d_2$-differential has degree $(-1, 2, 0)$, so in particular coweight $-1$. Since the charts in the aforementioned folders each depict a certain set of coweights, the source and the target of an Adams $d_2$-differential are likely to be on separate charts.

\subsection{Adams \texorpdfstring{$d_2$}{d\_2}-differentials}\label{section: Adams d_2-differentials}

What follows is a table of Adams $E_2$-page indecomposables, their degrees and the values of their $d_2$-differentials. If the $d_2$-column is empty, that means that $d_2$ on that indecomposable is zero. As mentioned in Remark \ref{remark on hidden extensions and indecomposables}, this table contains exactly the same indecomposables as table \ref{tab:Indecomposables on E_infty}.

\begin{longtable}{lll}
    \caption{Differentials on indecomposables on the Adams $E_2$-page for $\mmf^{\,\,\mathbb{R}}$
    \label{tab:Indecomposables on Adams E_2 and differentials}
    } \\
    \toprule
    Indecomposable & $(s,f,w)$ & $d_2$\\
    \midrule \endfirsthead
    \caption[]{Differentials on indecomposables on the Adams $E_2$-page for $\mmf^{\,\,\mathbb{R}}$} \\
    \toprule
    Indecomposable & $(s,f,w)$ & $d_2$\\
    \midrule \endhead
    \bottomrule \endfoot
        $\rho$ & $(-1, 0, -1)$ & \\
        $\tau^8$ & $(0, 0, -8)$ & \\
        $\tau^6 h_0$ & $(0, 1, -6)$ & \\
        $\tau^4 h_0$ & $(0, 1, -4)$ & \\
        $\tau^2 h_0$ & $(0, 1, -2)$ & \\
        $h_0$ & $(0, 1, 0)$ & \\
        $\tau^5 h_1$ & $(1, 1, -4)$ & \\
        $\tau h_1$ & $(1, 1, 0)$ & \\
        $h_1$ & $(1, 1, 1)$ & \\
        $\tau^2 h_2$ & $(3, 1, 0)$ & \\
        $h_2$ & $(3, 1, 2)$ & \\
        $\tau h_2^2$ & $(6, 2, 3)$ & \\
        $\tau^5 c$ & $(8, 3, 0)$ & \\
        $\tau c$ & $(8, 3, 4)$ & \\
        $c$ & $(8, 3, 5)$ & \\
        $\tau^4 P$ & $(8, 4, 0)$ & \\
        $\tau^2 P h_0$ & $(8, 5, 2)$ & \\
        $P h_0$ & $(8, 5, 4)$ & \\
        $\tau P h_1$ & $(9, 5, 4)$ & \\
        $u$ & $(11, 3, 7)$ & $h_1^2 c$\\
        $P h_2$ & $(11, 5, 6)$ & \\
        $\tau^6 a$ & $(12, 3, 0)$ & $\tau^6 P h_2$\\
        $a$ & $(12, 3, 6)$ & $P h_2$\\
        $\tau^6 d$ & $(14, 4, 2)$ & $\rho \tau^5 P h_2^2$\\
        $d$ & $(14, 4, 8)$ & \\
        $n$ & $(15, 3, 8)$ & $h_0 \cdot d + \rho \cdot h_1 d$\\
        $\tau P c$ & $(16, 7, 8)$ & \\
        $P^2$ & $(16, 8, 8)$ & \\
        $\tau^2 e$ & $(17, 4, 8)$ & \\
        $e$ & $(17, 4, 10)$ & $h_1^2 d$\\
        $\tau^7 h_0^2 e$ & $(17, 6, 3)$ & $\rho \tau^6 P h_1 c$\\
        $\tau^5 h_0^2 e$ & $(17, 6, 5)$ & $\rho \tau^4 P h_1 c$\\
        $\tau^4 g$ & $(20, 4, 8)$ & \\
        $g$ & $(20, 4, 12)$ & \\
        $\tau^2 P a$ & $(20, 7, 8)$ & $\tau^2 P^2 h_2$\\
        $\tau^2 P d$ & $(22, 8, 10)$ & $\rho \tau P^2 h_2^2$\\
        $\tau^6 P n$ & $(23, 7, 6)$ & $\tau^6 d \cdot P h_0$\\
        $\tau^4 a^2$ & $(24, 6, 8)$ & \\
        $\tau^5 \Dh$ & $(25, 5, 8)$ & \\
        $\tau \Dh$ & $(25, 5, 12)$ & \\
        $\tau^3 P h_0^2 e$ & $(25, 10, 11)$ & $\rho \tau^2 P^2 h_1 c$\\
        $\tau P h_0^2 e$ & $(25, 12, 13)$ & $\rho P^2 h_1 c$\\
        $\tau^4 a d$ & $(26, 7, 10)$ & $\tau^4 P e \cdot h_0$\\
        $\tau^2 a n$ & $(27, 6, 12)$ & \\
        $\tau^4 d^2$ & $(28, 8, 12)$ & \\
        $\tau^4 n^2$ & $(30, 6, 12)$ & \\
        $\tau^2 P^2 n$ & $(31, 11, 14)$ & $\tau^2 h_0 \cdot P^2 \cdot d$\\
        $\tau^5 \Dc + \tau^6 a g$ & $(32, 7, 12)$ & \\
        $\tau \Dc + \tau^2 a g$ & $(32, 7, 16)$ & \\
        $P a^2$ & $(32, 10, 16)$ & \\
        $\tau P \Dh$ & $(33, 9, 16)$ & \\
        $P a d$ & $(34, 11, 18)$ & $P^2 e \cdot h_0$\\
        $\tau^2 a^3$ & $(36, 9, 16)$ & $\tau^2 P a n \cdot h_0 + \rho^2 \cdot \tau^2 P a e$\\
        $\tau^2 P n^2$ & $(38, 10, 18)$ & $\rho^2 \cdot \tau^2 P d e$\\
        $\tau P \Dc + \tau^2 P a g$ & $(40, 11, 20)$ & \\
        $\D^2$ & $(48, 8, 24)$ & $\tau^2 a n g$\\
\end{longtable}

We proceed by giving proofs of the claimed $d_2$-differentials. The following is an analogue of \cite[Section 3]{BI}.

\begin{Lemma}\label{Comparison to C-motivic differentials}
    Recall the short exact sequence of $\text{\rm Ext}_{\mathcal{A}(2)^{\mathbb{R}}}$-modules from Lemma \ref{the map p}
    \[0 \xrightarrow[]{} \text{\rm coker}(\rho) \xrightarrow[]{i} \text{\rm Ext}_{\mathcal{A}(2)^\mathbb{C}} \xrightarrow[]{p} \ker(\rho) \xrightarrow[]{} 0.\]
    If $x$ is a permanent cycle in the $\rho$-Bockstein spectral sequence, then the map $i$ takes $x$ in $\text{\rm coker}(\rho) \subset \text{\rm Ext}_{\mathcal{A}(2)^\mathbb{R}}$ to the element of $\text{\rm Ext}_{\mathcal{A}(2)^\mathbb{C}}$ with the same name. Also, $i$ commutes with Adams differentials. Therefore $d_2^{\mathbb{R}}(x) \equiv d_2^{\mathbb{C}}(x)$ mod $\rho$ for all $x \in \text{\rm im}(i)$.
\end{Lemma}

\begin{proof}
    Assuming the existence of an $\mathbb{R}$-motivic modular forms spectrum $\mmf^{\,\,\mathbb{R}}$, all claims follow from the cofiber sequence
    \[\mmf^{\,\,\mathbb{R}} \xrightarrow[]{\rho} \mmf^{\,\,\mathbb{R}} \to \mmf^{\,\,\mathbb{R}}/\rho\]
    with arguments similar to \cite[Section 3]{BI}. The map $i$ is induced by $\mmf^{\,\,\mathbb{R}} \to \mmf^{\,\,\mathbb{R}}/\rho$ so it commutes with Adams differentials.
\end{proof}

Our approach to the Adams spectral sequence will be similar to our approach to the $\rho$-Bockstein spectral sequence in the following sense: We provide a table of all Adams $d_2$-differentials on indecomposables that could exist solely for degree reasons. Then, we consider whether each differential does or does not occur individually. Note that the degree of the Adams $d_2$-differential is $(-1, 2, 0)$. In particular its coweight is $-1$, meaning the source of a differential and its target are likely to be on different coweight charts.

Sometimes we will make reference to hidden extensions on the $\rho$-free quotient of the $\rho$-Bockstein $E_\infty$-page as proven in section \ref{section: Hidden extensions on the rho-free quotient}. In every case, these hidden extensions lift uniquely to the non-quotiented $E_\infty$-page because there are no $\rho$-power-torsion elements in the relevant degrees.

\begin{Not}
    We will write $d_2$ for the $\mathbb{R}$-motivic Adams differential. If we want to appeal to a $\mathbb{C}$-motivic Adams differential we will write $d_2^{\mathbb{C}}$.
\end{Not}

\begin{Prop}\label{proof of Adams d_2 differentials}
    Table \ref{tab:Indecomposables on Adams E_2 and differentials} describes the non-zero $d_2$-differentials in the Adams spectral sequence on all indecomposables on $E_2$.
\end{Prop}

\begin{proof}
    For degree reasons, the only indecomposables that can support a $d_2$-differential are given by table \ref{tab:Possible Adams d_2}. In the table, if an indecomposable appears twice then that means that the value of its $d_2$ has multiple possible summands. If the proof column is empty, then the proof of that differential is contained in this Proposition. The table lists the differentials whose sources are $\rho$-torsion free elements first, then the ones whose sources are $\rho$-power-torsion elements. Each part is ordered increasingly by $s$, then $f$, then $w$.

    \begin{longtable}{llllc}
    \caption{Possible non-zero Adams $d_2$-differentials on indecomposable elements
    \label{tab:Possible Adams d_2}
    } \\
    \toprule
    $x$ & $(s,f,w)$ & $d_2(x)$ & Occurs & Proof\\
    \midrule \endfirsthead
    \caption[]{Possible non-zero Adams $d_2$-differentials on indecomposable elements} \\
    \toprule
    $x$ & $(s,f,w)$ & $d_2(x)$ & Occurs & Proof\\
    \midrule \endhead
    \bottomrule \endfoot
        $u$ & $(11, 3, 7)$ & $h_1^2 c$ & Yes & \\
        $\tau^6 a$ & $(12, 3, 0)$ & $\tau^6 P h_2$ & Yes & \\
        $\tau^6 d$ & $(14, 4, 2)$ & $\rho \tau^5 P h_2^2$ & Yes & \ref{Adams differential on t^6 d}\\
        $\tau^5 h_0^2 e$ & $(17, 6, 5)$ & $\rho \tau^4 P h_1 c$ & Yes & \ref{Adams differential on t^5 h_0^2 e}\\
        $g$ & $(20, 4, 12)$ & $h_1^2 e$ & No & \\
        $\tau^2 P a$ & $(20, 7, 8)$ & $\tau^2 P^2 h_2$ & Yes & \\
        $\tau^2 P d$ & $(22, 8, 10)$ & $\rho \tau P^2 h_2^2$ & Yes & \ref{Adams differential on t^2 P d}\\
        $\tau^4 a^2$ & $(24, 6, 8)$ & $\tau^4 P n \cdot h_0$ & No & \ref{Adams differential on t^4 a^2}\\
        $\tau^4 a^2$ & $(24, 6, 8)$ & $\rho^2 \cdot \tau^4 P e$ & No & \ref{Adams differential on t^4 a^2}\\
        $\tau^5 \Dh$ & $(25, 5, 8)$ & $\tau^4 a^2 \cdot h_0$ & No & \ref{Adams differential on t^5 Dh}\\
        $\tau^5 \Dh$ & $(25, 5, 8)$ & $\rho^2 \cdot \tau^4 a d$ & No & \ref{Adams differential on t^5 Dh}\\
        $\tau P h_0^2 e$ & $(25, 10, 13)$ & $\rho P^2 h_1 c$ & Yes & \ref{Adams differential on t P h0^2 e}\\
        $\tau^4 a d$ & $(26, 7, 10)$ & $\tau^4 P e \cdot h_0$ & Yes & \ref{Adams differential on t^4 a d}\\
        $\tau^4 a d$ & $(26, 7, 10)$ & $\rho \cdot \tau^4 P h_1 e$ & No & \ref{Adams differential on t^4 a d}\\
        $\tau^4 d^2$ & $(28, 8, 12)$ & $\rho \tau^3 P h_0^2 g$ & No & \ref{Adams differential on t^4 d^2}\\
        $\tau^4 n^2$ & $(30, 6, 12)$ & $\tau^4 h_0 a e$ & No & \\
        $P a^2$ & $(32, 10, 16)$ & $P^2 n \cdot h_0$ & No & \ref{Adams differential on P a^2}\\
        $P a^2$ & $(32, 10, 16)$ & $\rho^2 \cdot P^2 e$ & No & \ref{Adams differential on P a^2}\\
        $P a d$ & $(34, 11, 18)$ & $P^2 e \cdot h_0$ & Yes & \ref{Adams differential on P a d}\\
        $P a d$ & $(34, 11, 18)$ & $\rho \cdot P^2 h_1 e$ & No & \ref{Adams differential on P a d}\\
        $\tau^2 a^3$ & $(36, 9, 16)$ & $\tau^2 P a n \cdot h_0$ & Yes & \ref{Adams differential on t^2 a^3}\\
        $\tau^2 a^3$ & $(36, 9, 16)$ & $\rho^2 \cdot \tau^2 P a e$ & Yes & \ref{Adams differential on t^2 a^3}\\
        $\tau^2 P n^2$ & $(38, 10, 18)$ & $\tau^2 P a e \cdot h_0$ & No & \ref{Adams differential on t^2 P n^2}\\
        $\tau^2 P n^2$ & $(38, 10, 18)$ & $\rho^2 \cdot \tau^2 P d e$ & Yes & \ref{Adams differential on t^2 P n^2}\\
        $\D^2$ & $(48, 8, 24)$ & $\tau^2 a n g$ & Yes & \\
        $\tau^5 h_1$ & $(1, 1, -4)$ & $\tau^4 h_0 \cdot h_0 \cdot h_0$ & No & \\
        $\tau^5 h_1$ & $(1, 1, -4)$ & $\rho^{11} \cdot u$ & No & \\
        $\tau h_1$ & $(1, 1, 0)$ & $h_0 \cdot h_0 \cdot h_0$ & No & \\
        $\tau h_1$ & $(1, 1, 0)$ & $\rho^3 \cdot h_1^3$ & No & \\
        $\tau^2 h_2$ & $(3, 1, 0)$ & $\rho \tau^2 h_1^3$ & No & \ref{Adams differential on t^2 h2}\\
        $\tau P h_1$ & $(9, 5, 4)$ & $P h_0 \cdot h_0 \cdot h_0$ & No & \\
        $\tau P h_1$ & $(9, 5, 4)$ & $\rho^7 \cdot h_1^4 u$ & No & \\
        $a$ & $(12, 3, 6)$ & $P h_2$ & Yes & \\
        $n$ & $(15, 3, 8)$ & $h_0 \cdot d$ & Yes & \ref{Adams differential on n}\\
        $n$ & $(15, 3, 8)$ & $\rho \cdot h_1 d$ & Yes & \ref{Adams differential on n}\\
        $e$ & $(17, 4, 10)$ & $h_1^2 d$ & Yes & \\
        $\tau^7 h_0^2 e$ & $(17, 6, 3)$ & $\rho \tau^6 P h_1 c$ & Yes & \ref{Adams differential on t^7 h0^2 e}\\
        $\tau^6 P n$ & $(23, 7, 6)$ & $\tau^6 d \cdot P h_0$ & Yes & \\
        $\tau^6 P n$ & $(23, 7, 6)$ & $\rho^7 \cdot \tau^4 h_1 d^2$ & No & \\
        $\tau \Dh$ & $(25, 5, 12)$ & $h_0 \cdot a^2$ & No & \ref{Adams differential on t Dh}\\
        $\tau \Dh$ & $(25, 5, 12)$ & $\rho^2 \cdot a \cdot d$ & No & \ref{Adams differential on t Dh}\\
        $\tau \Dh$ & $(25, 5, 12)$ & $\rho^7 \cdot u \cdot g$ & No & \ref{Adams differential on t Dh}\\
        $\tau^3 P h_0^2 e$ & $(25, 10, 11)$ & $\rho \tau^2 P^2 h_1 c$ & Yes & \ref{Adams differential on t^3 P h0^2 e}\\
        $\tau^2 a n$ & $(27, 6, 12)$ & $\tau^2 P h_2 n$ & No & \\
        $\tau^2 P^2 n$ & $(31, 11, 14)$ & $\tau^2 h_0 \cdot P^2 \cdot d$ & Yes & \ref{Adams differential on t^2 P^2 n}\\
        $\tau^2 P^2 n$ & $(31, 11, 14)$ & $\rho^7 \cdot P^2 \cdot h_1 g$ & No & \ref{Adams differential on t^2 P^2 n}\\
        $\tau^2 P^2 n$ & $(31, 11, 14)$ & $\rho^{15} \cdot h_1^5 g^2$ & No & \ref{Adams differential on t^2 P^2 n}\\
        $\tau^5 \Dc + \tau^6 a g$ & $(32, 7, 12)$ & $\rho \tau^6 h_1 d e$ & No & \ref{Adams differential on t^5 Dc + t^6 a g}\\
        $\tau \Dc + \tau^2 a g$ & $(32, 7, 16)$ & $\rho \tau^2 h_1 d e$ & No & \ref{Adams differential on t Dc + t^2 a g}\\
        $\tau P \Dh$ & $(33, 9, 16)$ & $\rho^2 P a d$ & No & \ref{Adams differential on t P Dh}\\
        $\tau P \Dc + \tau^2 P a g$ & $(40, 11, 20)$ & $\rho \tau^2 P h_1 d e$ & No & \ref{Adams differential on t P Dc + t^2 P a g}\\
    \end{longtable}

    Some differentials follow immediately from comparison to the $\mathbb{C}$-motivic case. For example by \cite[Table 1]{Isa18} we have $d_2^\mathbb{C}(u) = h_1^2 c$. Then by Lemma \ref{Comparison to C-motivic differentials} we know that $\mathbb{R}$-motivically $d_2(u) \equiv h_1^2 c$ mod $\rho$. Since there is no possible $\rho$-divisible target listed for $d_2(u)$ in table \ref{tab:Possible Adams d_2}, we can conclude $d_2(u) = h_1^2 c$.

    A similar argument works for the differentials on $\tau^6 a$, $g$, $\tau^2 P a$, $\tau^4 n^2$, $\D^2$, $a$, $e$, and $\tau^2 a n$. For the elements $\tau^5 h_1$, $\tau h_1$, $\tau P h_1$, and $\tau^6 P n$ we can also use a similar argument to rule out or confirm the summands that are not $\rho$-divisible. In each of the four cases, the other summand is then ruled out because the source of the differential is $\rho$-power-torsion and the possible summand would make the target $\rho$-torsion free.

    For all other differentials see the following series of Lemmas.    
\end{proof}

As charts will become more relevant now, we remind the reader of the chart guide given in section \ref{section: Chart guide for Adams differentials}. We also remind the reader of Notation \ref{lemmas have (s, f, w) degree stated in them}. Sometimes we will need to determine a hidden extension in order to use it in the proof of an Adams differential. In that case, we will list both the degree of the source and the degree of the target. See for example Lemma \ref{hidden t^4 P-extension from t h2^2 g to r t^2 P a n}.

\begin{Lemma}\label{Adams differential on t^6 d}\label{Adams differential hitting r t^5 P h_2^2}\deg{14}{4}{2}
    $d_2(\tau^6 d) = \rho \tau^5 P h_2^2$.
\end{Lemma}

\begin{proof}
    Recall the hidden extension $\rho \cdot \tau^6 d = h_1 \cdot \tau^6 a$. By Proposition \ref{proof of Adams d_2 differentials} we know $d_2(\tau^6 a) = \tau^6 P h_2$. So we get
    \[\rho \cdot d_2(\tau^6 d) = h_1 \cdot \tau^6 P h_2.\]
    For degree reasons (e.g. by looking at the coweight $3$ mod $8$ chart), this relation forces the claimed differential.
\end{proof}

\begin{Lemma}\label{Adams differential on t^5 h_0^2 e}\label{Adams differential hitting r t^4 P h1 c}\deg{17}{6}{5}
    $d_2(\tau^5 h_0^2 e) = \rho \tau^4 P h_1 c$.
\end{Lemma}

\begin{proof}
    By Lemma \ref{Adams differential on t^6 d} we know $d_2(\tau^6 d) = \rho \tau^5 P h_2^2$. Combine this with the hidden extension $\rho \cdot \tau^5 h_0^2 e = h_1^2 \cdot \tau^6 d$ to get
    \[\rho \cdot d_2(\tau^5 h_0^2 e) = h_1^2 \cdot \rho \tau^5 P h_2^2.\]
    This forces the claimed differential.
\end{proof}

\begin{Lemma}\label{Adams differential on t^2 P d}\label{Adams differential hitting r t P^2 h2}\deg{22}{8}{10}
    $d_2(\tau^2 P d) = \rho \tau P^2 h_2^2$.
\end{Lemma}

\begin{proof}
    By Proposition \ref{proof of Adams d_2 differentials} we know $d_2(\tau^2 P a) = \tau^2 P^2 h_2$. Combine this with the hidden extension $\rho \cdot \tau^2 P d = h_1 \cdot \tau^2 P a$ to get
    \[\rho \cdot d_2(\tau^2 P d) = h_1 \cdot \tau^2 P^2 h_2.\]
    This forces the claimed differential.
\end{proof}

\begin{Lemma}\label{Adams differential on t^4 a^2}\deg{24}{6}{8}
    $d_2(\tau^4 a^2) = 0$.
\end{Lemma}

\begin{proof}
    We have the relation $\tau^6 a \cdot \tau^6 a = \tau^8 \cdot \tau^4 a^2$. By the Leibniz rule all squares are $d_2$-cycles, so this implies $\tau^8 \cdot d_2(\tau^4 a^2) = 0$. None of the possible non-trivial targets of $d_2(\tau^4 a^2)$ are $\tau^8$-torsion.
\end{proof}

\begin{Lemma}\label{Adams differential on t^5 Dh}\deg{25}{5}{8}
    $d_2(\tau^5 \Dh) = 0$.
\end{Lemma}

\begin{proof}
    Recall that there are two possible summands for $d_2(\tau^5 \Dh)$, namely $\tau^4 a^2 \cdot h_0$ and $\rho^2 \cdot \tau^4 a d$. Since $d_2^{\mathbb{C}}(\tau^5 \Dh) = 0$, we can conclude that the summand $\tau^4 a^2 \cdot h_0$ does not appear as $d_2(\tau^5 \Dh) \equiv 0$ mod $\rho$ by Lemma \ref{Comparison to C-motivic differentials}. Then we could still have $d_2(\tau^5 \Dh) = \rho^2 \cdot \tau^4 a d$, but the source is $h_1^3$-torsion and the target is $h_1$-torsion free due to hidden extensions.
\end{proof}

\begin{Lemma}\label{Adams differential on t P h0^2 e}\label{Adams differential hitting r P^2 h1 c}\deg{25}{10}{13}
    $d_2(\tau P h_0^2 e) = \rho P^2 h_1 c$.
\end{Lemma}

\begin{proof}
    By Lemma \ref{Adams differential on t^5 h_0^2 e} we know $d_2(\tau^5 h_0^2 e) = \rho \tau^4 P h_1 c$. Since $\tau^4 P$ is a $d_2$-cycle for degree reasons, we can use the relation $\tau^8 \cdot \tau P h_0^2 e = \tau^4 P \cdot \tau^5 h_0^2 e$ to obtain
    \[\tau^8 \cdot d_2(\tau P h_0^2 e) = \tau^4 P \cdot \rho \tau^4 P h_1 c = \tau^8 \cdot \rho P^2 h_1 c.\]
    This forces the claimed differential.
\end{proof}

\begin{Lemma}\label{Adams differential on t^4 a d}\label{Adams differential hitting t^4 P h0 e}\deg{26}{7}{10}
    $d_2(\tau^4 a d) = \tau^4 P e \cdot h_0$.
\end{Lemma}

\begin{proof}
    Recall that there are two possible summands for $d_2(\tau^4 a d)$, namely $\tau^4 P e \cdot h_0$ and $\rho \cdot \tau^4 P h_1 e$. First compute the $\mathbb{C}$-motivic differential
    \[d_2^{\mathbb{C}}(\tau^4 a d) = \tau^4 d \cdot d_2^{\mathbb{C}}(a) = \tau^4 d \cdot P h_2 = \tau^4 P h_0 e,\]
    where we used the $\mathbb{C}$-motivic relation $h_2 d = h_0 e$. So $\tau^4 P e \cdot h_0$ is a summand in $d_2(\tau^4 a d)$.

    To rule out the summand $\rho \cdot \tau^4 P h_1 e$, consider the hidden extension $h_1 \cdot \tau^4 a^2 = \rho \cdot \tau^4 a d$. Since $d_2(\tau^4 a^2) = 0$ by Lemma \ref{Adams differential on t^4 a^2}, we get $\rho \cdot d_2(\tau^4 a d) = 0$. But if $\rho \tau^4 P h_1 e$ is a summand in $d_2(\tau^4 a d)$, then $d_2(\tau^4 a d)$ supports a non-trivial $\rho$-multiplication.
\end{proof}

\begin{Lemma}\label{Adams differential on t^4 d^2}\deg{28}{8}{12}
    $d_2(\tau^4 d^2) = 0$.
\end{Lemma}

\begin{proof}
    Consider the hidden extension $h_1 \cdot \tau^4 a d = \rho \cdot \tau^4 d^2$. By Lemma \ref{Adams differential on t^4 a d} this implies
    \[h_1 \cdot \tau^4 P e \cdot h_0 = \rho \cdot d_2(\tau^4 d^2).\]
    The left-hand side is $0$ since $h_1 \cdot h_0 = 0$. But the only possible non-trivial value of $d_2(\tau^4 d^2)$ is $\rho \tau^3 P h_0^2 g$, which supports a $\rho$-multiplication.
\end{proof}

\begin{Lemma}\label{Adams differential on P a^2}\deg{32}{10}{16}
    $d_2(P a^2) = 0$.
\end{Lemma}

\begin{proof}
    Consider the relation\footnote{Technically, this relation depends on the choice of the $\Ext$ element $P a^2$. However, as we discussed in the introduction to this section, the choice of $P a^2$ does not affect its Adams $d_2$.} $\tau^8 \cdot P a^2 = \tau^4 P \cdot \tau^4 a^2$. By Lemma \ref{Adams differential on t^4 a^2} this implies $\tau^8 \cdot d_2(P a^2) = 0$. None of the possible non-trivial targets for this differential are $\tau^8$-torsion.
\end{proof}

\begin{Lemma}\label{Adams differential on P a d}\label{Adams differential hitting P^2 h0 e}\deg{34}{11}{18}
    $d_2(P a d) = P^2 e \cdot h_0$.
\end{Lemma}

\begin{proof}
    Consider the relation $\tau^8 \cdot P a d = \tau^4 P \cdot \tau^4 a d$. By Lemma \ref{Adams differential on t^4 a d} this implies $\tau^8 \cdot d_2(P a d) = \tau^4 P \cdot \tau^4 P e \cdot h_0 = \tau^8 \cdot P^2 e \cdot h_0$. This forces the claimed differential.
\end{proof}

To compute the differential on $\tau^2 a^3$, we need to compute the differential on the decomposable element $\tau^6 d g = \tau^6 d \cdot g$ first. For that differential we need to establish a hidden extension on the $\rho$-Bockstein $E_\infty$-page first.

\begin{Lemma}\label{hidden t^4 P-extension from t h2^2 g to r t^2 P a n}\deg{26}{6}{15} \deg{34}{10}{15}
    On the $\rho$-Bockstein $E_\infty$-page, there is a hidden $\tau^4 P$-extension from $\tau h_2^2 g$ to $\rho \tau^2 P a n$.
\end{Lemma}

\begin{proof}
    Multiply both source and target of the claimed hidden extension by $\rho$ and observe that $\rho \tau h_2^2 g$ and $\rho^2 \tau^2 P a n$ are in $\ker(\rho)$. So we can use the map $p$ from Lemma \ref{Comparison to C-motivic differentials}. Since $p$ has degree $(0, 1, 1)$, a preimage of $\rho \tau h_2^2 g$ under $p$ must have degree $(25, 5, 13)$. Therefore this preimage has to be $\Dh$. For $\rho^2 \tau^2 P a n$ one similarly finds the preimage $\tau^4 P \Dh$. As $\tau^4 P$-multiplication sends $\Dh$ to $\tau^4 P \Dh$, we get the claimed hidden extension.
\end{proof}

\begin{Lemma}\label{Adams differential on t^6 d g}\label{Adams differential hitting r^2 t^2 P a n}\deg{34}{8}{14}
    $d_2(\tau^6 d \cdot g) = \rho^2 \tau^2  P a n$.
\end{Lemma}

\begin{proof}
    By Lemma \ref{Adams differential on t^6 d} we know $d_2(\tau^6 d) = \rho \tau^5 P h_2^2$. By the Leibniz rule and Lemma \ref{hidden t^4 P-extension from t h2^2 g to r t^2 P a n} this implies
    \[d_2(\tau^6 d \cdot g) = \rho \tau^5 P h_2^2 \cdot g = \tau^4 P \cdot \rho \tau h_2^2 g = \rho^2 \tau^2 P a n.\]
\end{proof}

We need yet another hidden extension for the Adams $d_2$ on $\tau^2 a^3$.

\begin{Lemma}\label{hidden h1-extension from t^2 P a n to r t^2 P a e}\deg{35}{10}{16} \deg{36}{11}{17}
    On the $\rho$-Bockstein $E_\infty$-page there is a hidden $h_1$-extension from $\tau^2 P a n$ to $\rho \tau^2 P a e$.
\end{Lemma}

\begin{proof}
    Multiply both source and target by $\rho^2$ to obtain a hidden extension in $\ker(\rho)$. As we saw in Lemma \ref{Adams differential on t^6 d g}, the preimage under $p$ of $\rho^2 \tau^2 P a n$ is $\tau^4 P \Dh$. For $\rho^3 \tau^2 P a e$ the preimage has degree $(34, 10, 14)$, so it must be $\tau^4 P \Dh h_1$. As the latter is the $h_1$-multiple of the former, the hidden $h_1$-extension occurs.
\end{proof}

\begin{Lemma}\label{Adams differential on t^2 a^3}\label{Adams differential hitting t^3 P Dh h1^2 + r^2 t^2 P a e}\deg{36}{9}{16}
    $d_2(\tau^2 a^3) = \tau^2 P a n \cdot h_0 + \rho^2 \cdot \tau^2 P a e$.
\end{Lemma}

\begin{proof}
    To see that $\tau^2 P a n \cdot h_0$ is a summand in $d_2(\tau^2 a^3)$, compare to the $\mathbb{C}$-motivic $d_2$
    \[d_2^{\mathbb{C}}(\tau^2 a^3) = \tau^2 a^2 \cdot d_2^{\mathbb{C}}(a) = \tau^2 a^2 \cdot P h_2 =  \tau^2 P a n \cdot h_0,\]
    where we used the $\mathbb{C}$-motivic relation $h_2 a = h_0 n$.

    To see that $\rho^2 \cdot \tau^2 P a e$ is a summand, recall from Lemma \ref{Adams differential on t^6 d g} that $d_2(\tau^6 d \cdot g) = \rho^2 \tau^2 P a n$. Because of the hidden extension $h_1 \cdot \tau^6 d \cdot g = \rho \cdot \tau^2 a^3$, this implies
    \[h_1 \cdot \rho^2 \tau^2 P a n = \rho \cdot d_2(\tau^2 a^3).\]
    By Lemma \ref{hidden h1-extension from t^2 P a n to r t^2 P a e} the left-hand side is equal to $\rho^3 \tau^2 P a e$. So $\rho^2 \cdot \tau^2 P a e$ must be a summand in $d_2(\tau^2 a^3)$.
\end{proof}

For the Adams $d_2$ on $\tau^2 P n^2$ we need another hidden extension.

\begin{Lemma}\label{hidden h1-extension from t^2 P a e to r t^2 P d e}\deg{37}{11}{18} \deg{38}{12}{19}
    On the $\rho$-Bockstein $E_\infty$-page there is a hidden $h_1$-extension from $\tau^2 P a e$ to $\rho \tau^2 P d e$.
\end{Lemma}

\begin{proof}
    Multiply both source and target by $\rho^3$ to obtain a hidden extension in $\text{ker}(\rho)$. Then use the map $p$ from Lemma \ref{Comparison to C-motivic differentials}. For degree reasons, the preimages under $p$ of $\rho^3 \tau^2 P a e$ and $\rho^4 \tau^2 P d e$ are $\tau^4 P \Dh h_1$ and $\tau^4 P \Dh h_1^2$, respectively. So the hidden $h_1$-multiplication occurs.
\end{proof}

\begin{Lemma}\label{Adams differential on t^2 P n^2}\label{Adams differential hitting r^2 t^2 P d e}\deg{38}{10}{18}
    $d_2(\tau^2 P n^2) = \rho^2 \cdot \tau^2 P d e$.
\end{Lemma}

\begin{proof}
    There are two possible summands for $d_2(\tau^2 P n^2)$, namely $\tau^2 P a e \cdot h_0$ and $\rho^2 \cdot \tau^2 P d e$. To see that $\tau^2 P a e \cdot h_0$ does not occur as a summand, observe that $d_2^{\mathbb{C}}(\tau^2 P n^2) = 0$ because $\tau^2$, $P$ and $n^2$ are all $d_2^{\mathbb{C}}$-cycles.

    To see that $\rho^2 \cdot \tau^2 P d e$ does occur as a summand, we can use the hidden extension $h_1 \cdot \tau^2 a^3 = \rho \cdot \tau^2 P n^2$. By Lemma \ref{Adams differential on t^2 a^3} and Lemma \ref{hidden h1-extension from t^2 P a e to r t^2 P d e} this implies
    \[\rho \cdot d_2(\tau^2 P n^2) = h_1 \cdot (\tau^2 P a n \cdot h_0 + \rho^2 \cdot \tau^2 P a e) = h_1 \cdot \rho^2 \cdot \tau^2 P a e = \rho^3 \tau^2 P d e.\]
    This forces the claimed differential.
\end{proof}

This concludes the computation of Adams $d_2$-differentials on $\rho$-torsion free indecomposables. Now we consider sources that are $\rho$-power-torsion.

To prove that $\tau^2 h_2$ is a $d_2$-cycle, we would like to make use of the fact that $\tau \Dh$ is a $d_2$-cycle. So we will prove that first.

\begin{Lemma}\label{Adams differential on t Dh}\deg{25}{5}{12}
    $d_2(\tau \Dh) = 0$.
\end{Lemma}

\begin{proof}
    The possible non-trivial summands of $d_2(\tau \Dh)$ are $h_0 \cdot a^2$, $\rho^2 \cdot a \cdot d$, and $\rho^7 \cdot u \cdot g$. The last summand does not appear because $\tau \Dh$ is $\rho$-power-torsion and $u \cdot g$ is $\rho$-torsion free. By combining Lemma \ref{Comparison to C-motivic differentials} with $d_2^{\mathbb{C}}(\tau \Dh) = 0$, we know that $h_0 \cdot a^2$ does not occur as a summand. Lastly, note that $\tau \Dh$ is $h_1^3$-torsion because $\tau \Dh \cdot h_1^2 = h_0 \cdot a \cdot n$. But $\rho^2 \cdot a \cdot d$ supports a hidden $h_1^3$-multiplication, which is a direct consequence of $\rho^4 \tau^4 P$-multiplication on $a \cdot d$. Therefore we cannot have $d_2(\tau \Dh) = \rho^2 \cdot a \cdot d$, so it must be zero.
\end{proof}

We will also need a hidden extension for the Adams $d_2$ on $\tau^2 h_2$.

\begin{Lemma}\label{hidden t Dh-extension from t^2 h1^3 to r^3 t^2 d e}\deg{3}{3}{1} \deg{28}{8}{13}
    On the $\rho$-Bockstein $E_\infty$-page there is a hidden $\tau \Dh$-extension from $\tau^2 h_1^3$ to $\rho^3 \tau^2 d e$.
\end{Lemma}

\begin{proof}
    Multiply source and target by $\rho$ to obtain a hidden extension in $\ker(\rho)$. Now use the map $p$ from Lemma \ref{Comparison to C-motivic differentials}. For degree reasons, the respective preimages of $\rho \tau^2 h_1^3$ and $\rho^4 \tau^2 d e$ under $p$ are $\tau^3 h_1^2$ and $\tau^4 \Dh h_1^2$. Since there is a $\tau \Dh$-multiplication between these preimages, the claimed hidden $\tau \Dh$-extension occurs.
\end{proof}

\begin{Lemma}\label{Adams differential on t^2 h2}\deg{3}{1}{0}
    $d_2(\tau^2 h_2) = 0$.
\end{Lemma}

\begin{proof}
    The possible non-trivial value of $d_2(\tau^2 h_2)$ is $\rho \tau^2 h_1^3$. Note that for degree reasons $\tau^2 h_2$ does not support a hidden $\tau \Dh$-multiplication. However, by Lemma \ref{hidden t Dh-extension from t^2 h1^3 to r^3 t^2 d e} $\rho \tau^2 h_1^3$ does support a hidden $\tau \Dh$-multiplication to $\rho^4 \tau^2 d e$. Since $d_2$ is $\tau \Dh$-linear by Lemma \ref{Adams differential on t Dh}, this implies that $d_2(\tau^2 h_2)$ cannot be $\rho \tau^2 h_1^3$, so it must be zero.
\end{proof}

\begin{Lemma}\label{Adams differential on n}\label{Adams differential hitting h0 d + r h1 d}\deg{15}{3}{8}
    $d_2(n) = h_0 \cdot d + \rho \cdot h_1 d$.
\end{Lemma}

\begin{proof}
    Since $d_2^{\mathbb{C}}(n) = h_0 d$, we know that $h_0 \cdot d$ appears as a summand in $d_2(n)$ by Lemma \ref{Comparison to C-motivic differentials}.

    To see that $\rho \cdot h_1 d$ is also a summand in $d_2(n)$, we can use the hidden extension $h_1 \cdot n = \rho \cdot e$ which is a direct consequence of $\rho^4 \tau^4 P$-multiplication. Recall from Proposition \ref{proof of Adams d_2 differentials} that $d_2(e) = h_1^2 d$. Now the aforementioned hidden extension implies
    \[h_1 \cdot d_2(n) = \rho \cdot h_1^2 d \neq 0.\]
    Since $h_0 \cdot d$ is an $h_0$-multiple it does not support an $h_1$-multiplication. So $\rho \cdot h_1 d$ must also be a summand in $d_2(n)$.
\end{proof}

\begin{Lemma}\label{Adams differential on t^7 h0^2 e}\label{Adams differential hitting r t^6 P h1 c}\deg{17}{6}{3}
    $d_2(\tau^7 h_0^2 e) = \rho \tau^6 P h_1 c$.
\end{Lemma}

\begin{proof}
    As a direct consequence of $\rho^4 \tau^4 P$-multiplication, there is a hidden extension $h_1 \cdot \tau^7 h_0^2 e = \rho^2 \cdot \tau^4 P a$. By Proposition \ref{proof of Adams d_2 differentials} we know $d_2(a) = P h_2$, so
    \[h_1 \cdot d_2(\tau^7 h_0^2 e) = \rho^2 \cdot \tau^4 P \cdot P h_2 \neq 0.\]
    This forces the claimed differential.
\end{proof}

\begin{Lemma}\label{Adams differential on t^3 P h0^2 e}\label{Adams differential hitting r t^2 P^2 h1 c}\deg{25}{10}{11}
    $d_2(\tau^3 P h_0^2 e) = \rho \tau^2 P^2 h_1 c$.
\end{Lemma}

\begin{proof}
    Consider the relation $\tau^8 \cdot \tau^3 P h_0^2 e = \tau^4 P \cdot \tau^7 h_0^2 e$. Together with Lemma \ref{Adams differential on t^7 h0^2 e} we get
    \[\tau^8 \cdot d_2(\tau^3 P h_0^2 e) = \tau^4 P \cdot \rho \tau^6 P h_1 c = \tau^8 \cdot \rho \tau^2 P^2 h_1 c.\]
    This forces the claimed differential.
\end{proof}

\begin{Lemma}\label{Adams differential on t^2 P^2 n}\label{Adams differential hitting t^2 P^2 h0 d}\deg{31}{11}{14}
    $d_2(\tau^2 P^2 n) = \tau^2 h_0 \cdot P^2 \cdot d$.
\end{Lemma}

\begin{proof}
    Recall that the possible non-trivial summands for $d_2(\tau^2 P^2 n)$ are $\tau^2 h_0 \cdot P^2 \cdot d$, $\rho^7 \cdot P^2 \cdot h_1 g$, and $\rho^{15} \cdot h_1^5 g^2$. Since $d_2^{\mathbb{C}}(\tau^2 P^2 n) = \tau^2 P^2 h_0 d$, we know that this summand must also appear in $d_2(\tau^2 P^2 n)$. The other summands do not occur because then $d_2(\tau^2 P^2 n)$ would be $\rho$-torsion free, which is in contradiction to $\tau^2 P^2 n$ being $\rho$-power-torsion.
\end{proof}

The next two differentials are on $\tau^5 \Dc + \tau^6 a g$ and $\tau \Dc + \tau^2 a g$. In both cases, we could deduce their differentials by using $g$-multiplication. For the latter element this works fine, for the first there is a non-trivial amount of care needed due to the presence of $\rho$-torsion free elements in higher $\rho$-filtration. Therefore, we will prove the $d_2$ on $\tau \Dc + \tau^2 a g$ first, then prove the one on $\tau^5 \Dc + \tau^6 a g$ without having to refer to $g$-multiplication and instead referring to the previous $d_2$.

To keep the proceeding proof shorter, we will first prove that the decomposable element $\tau^2 a g = \tau^2 e \cdot n$ is an Adams $d_2$-cycle. This in turn requires a hidden extension first.

\begin{Lemma}\label{hidden t^2 e-extension from h0 d to r t^2 h1 d e}\deg{14}{5}{8} \deg{31}{9}{16}
    On the $\rho$-Bockstein $E_\infty$-page there is a hidden $\tau^2 e$-extension from $h_0 \cdot d$ to $\rho \tau^2 h_1 d e$.
\end{Lemma}

\begin{proof}
    This can be checked using the map $p$ from Lemma \ref{Comparison to C-motivic differentials}. Respective preimages under $p$ of $h_0 \cdot d$ and $\rho \tau^2 h_1 d e$ are $\tau d$ and $\tau^3 d e$. Since the latter is the $\tau^2 e$-multiple of the former, the hidden $\tau^2 e$-extension occurs.
\end{proof}

\begin{Lemma}\label{Adams differential on t^2 a g}\deg{32}{7}{16}
    $d_2(\tau^2 a g) = 0$.
\end{Lemma}

\begin{proof}
    We have the decomposition $\tau^2 a g = \tau^2 e \cdot n$. By Lemma \ref{Adams differential on n} this implies
    \[d_2(\tau^2 a g) = \tau^2 e \cdot d_2(n) = \tau^2 e \cdot (h_0 \cdot d + \rho \cdot h_1 d) = \tau^2 e \cdot h_0 d + \rho \tau^2 h_1 d e.\]
    By Lemma \ref{hidden t^2 e-extension from h0 d to r t^2 h1 d e} the two summands agree and therefore $d_2(\tau^2 a g) = 0$.
\end{proof}

\begin{Lemma}\label{Adams differential on t Dc + t^2 a g}\deg{32}{7}{16}
    $d_2(\tau \Dc + \tau^2 a g) = 0$.
\end{Lemma}

\begin{proof}
    According to table \ref{tab:Possible Adams d_2} we could have $d_2(\tau \Dc + \tau^2 a g) = \rho \tau^2 h_1 d e$. Multiplying this by $g$ would imply
    \[d_2(\tau^2 a g^2) = \rho \tau^2 h_1 d e \cdot g \neq 0.\]
    However, $d_2(\tau^2 a g^2)$ must be zero because of the decomposition $\tau^2 a g^2 = \tau^2 a g \cdot g$ and the fact that $d_2(\tau^2 a g) = 0$ which was proven in Lemma \ref{Adams differential on t^2 a g}. This means $d_2(\tau \Dc + \tau^2 a g) = 0$.
\end{proof}

For the Adams $d_2$ on $\tau^5 \Dc + \tau^6 a g$ we need one more preparatory hidden extension.

\begin{Lemma}\label{hidden t^5 h1-extension from t^6 h1 d e to r^5 t^6 P n n}\deg{32}{9}{13} \deg{33}{10}{9}
    On the $\rho$-Bockstein $E_\infty$-page there is a hidden $\tau^5 h_1$-extension from $\tau^6 h_1 d e$ to $\rho^5 \cdot \tau^6 P n \cdot n$.
\end{Lemma}

\begin{proof}
    Multiply source and target by $\rho$ to obtain a hidden extension in $\ker(\rho)$. Now use the map $p$ from Lemma \ref{Comparison to C-motivic differentials}. Respective preimages under $p$ of $\rho \tau^6 h_1 d e$ and $\rho^6 \cdot \tau^6 P n \cdot n$ are $\tau^7 d e$ and $\tau^{12} h_1 d e$. As the $\tau^5 h_1$-multiple of the first is the latter, the claimed hidden extension occurs.
\end{proof}

\begin{Lemma}\label{Adams differential on t^5 Dc + t^6 a g}\deg{32}{7}{12}
    $d_2(\tau^5 \Dc + \tau^6 a g) = 0$.
\end{Lemma}

\begin{proof}
    According to table \ref{tab:Possible Adams d_2} we could have $d_2(\tau^5 \Dc + \tau^6 a g) = \rho \tau^6 h_1 d e$. Recall from Proposition \ref{proof of Adams d_2 differentials} that $d_2(\tau^5 h_1) = 0$ and note the relation $\tau^5 h_1 \cdot (\tau^5 \Dc + \tau^6 a g) = \tau^8 \cdot \tau^2 \Dc h_1$. Multiplication by $\tau^5 h_1$ would therefore imply
    \[d_2(\tau^8 \cdot \tau^2 \Dc h_1) = \tau^5 h_1 \cdot \rho \tau^6 h_1 d e.\]
    We will now show that the left-hand side must be zero and the right-hand side is non-zero. As a consequence of this contradiction, $d_2(\tau^5 \Dc + \tau^6 a g)$ must be trivial.
    
    By Proposition \ref{proof of Adams d_2 differentials} and Lemma \ref{Adams differential on t Dc + t^2 a g} the element $\tau^8 \cdot \tau^2 \Dc h_1 = \tau^8 \cdot (\tau \Dc + \tau^2 a g) \cdot \tau h_1$ is a product of $d_2$-cycles, hence a $d_2$-cycle.

    By Lemma \ref{hidden t^5 h1-extension from t^6 h1 d e to r^5 t^6 P n n} the right-hand side is equal to $\rho^6 \cdot \tau^6 P n \cdot n \neq 0$.
\end{proof}

\begin{Lemma}\label{Adams differential on t P Dh}\deg{33}{9}{16}
    $d_2(\tau P \Dh) = 0$.
\end{Lemma}

\begin{proof}
    The only possible target for $d_2(\tau P \Dh)$ is $\rho^2 P a d$ which is $\rho$-torsion free, but $\tau P \Dh$ is $\rho^3$-torsion.
\end{proof}

\begin{Lemma}\label{Adams differential on t P Dc + t^2 P a g}\deg{40}{11}{20}
    $d_2(\tau P \Dc + \tau^2 P a g) = 0$.
\end{Lemma}

\begin{proof}
    Consider the relation $\tau^8 \cdot (\tau P \Dc + \tau^2 P a g) = \tau^4 P \cdot (\tau^5 \Dc + \tau^6 a g)$. By Lemma \ref{Adams differential on t^5 Dc + t^6 a g} this implies
    \[ \tau^8 \cdot d_2(\tau P \Dc + \tau^2 P a g) = 0.\]
    But the only possible non-trivial value for $d_2(\tau P \Dc + \tau^2 P a g)$ is $\rho \tau^2 P h_1 d e$, which is not $\tau^8$-torsion.
\end{proof}

\bibliographystyle{alpha}
\bibliography{bib}

\end{document}